\documentclass{amsart}
\usepackage[utf8]{inputenc}
\usepackage{microtype}
\usepackage{amsthm, amssymb, latexsym, amsmath, mathrsfs, amscd}
\usepackage[all]{xy}
\usepackage{tikz}
\usepackage{color}
\usepackage[english]{babel}
\usepackage{spectralsequences}
\usepackage{tikz-cd}

\usepackage{enumitem}
\usepackage{hyperref}
\def\gaia#1{\textcolor{teal}{{gaia:} #1 }}

\def\gaiabis#1{\textcolor{red}{{gaiabis:} #1 }}
\tikzcdset{scale cd/.style={every label/.append style={scale=#1},
    cells={nodes={scale=#1}}}}
\newcommand{\Z}{{\mathbb Z}}
\newcommand{\C}{\mathbb{C}}

\newcommand{\rL}{{\mathrm L}}

\newcommand{\p}[1]{{\mathbb{P}^{#1}}}

\newcommand{\cccO}{{\mathcal O}}
\newcommand{\cR}{{\mathcal R}}

\newcommand{\Ku}{{\mathrm{Ku}}}
\newcommand{\ch}{\operatorname{ch}}

\newcommand{\supp}{\operatorname{Supp}}

\newcommand{\cale}{{\mathcal E}}
\newcommand{\calf}{{\mathcal F}}
\newcommand{\calg}{{\mathcal G}}
\newcommand{\cG}{{\mathcal G}}
\newcommand{\cF}{{\mathcal F}}

\newcommand{\calh}{{\mathcal H}}
\newcommand{\cali}{{\mathcal I}}

\newcommand{\calm}{{\mathrm M}}
\newcommand{\calmi}{{\mathrm {MI}}}
\newcommand{\caln}{{\mathcal N}}

\newcommand{\calq}{{\mathcal Q}}
\newcommand{\calr}{{\mathcal R}}
\newcommand{\cals}{{\mathcal S}}

\newcommand{\calt}{{\mathcal T}}
\newcommand{\calv}{{\mathcal V}}

\newcommand{\inhom}{{\mathcal H}{\it om}}
\newcommand{\inext}{{\mathcal E}{\it xt}}

\newcommand{\Ext}{\operatorname{Ext}}
\newcommand{\ext}{\operatorname{ext}}

\newcommand{\Ogr}{\operatorname{OG}}
\newcommand{\Lgr}{\operatorname{LG}}
\newcommand{\Spl}{\operatorname{Spl}}

\newcommand{\Hom}{\operatorname{Hom}}
\newcommand{\RHom}{\operatorname{R\mathcal{H}\textit{om}}}
\newcommand{\Hilb}{\operatorname{Hilb}}
\DeclareMathOperator{\Quot}{Quot}

\DeclareMathOperator{\coker}{coker}
\DeclareMathOperator{\im}{im}

\DeclareMathOperator{\codim}{{codim}}
\DeclareMathOperator{\rk}{{rk}}

\DeclareMathOperator{\Pic}{{Pic}}
\DeclareMathOperator{\Coh}{Coh}

\DeclareMathOperator{\Gr}{{Gr}}

\newtheorem{theorem}{Theorem}[section]
\newtheorem{mthm}{Main Theorem}
\newtheorem*{mthm*}{Main Theorem}
\newtheorem{proposition}[theorem]{Proposition}

\newtheorem*{theorem*}{Theorem}
\newtheorem*{claim}{{\bf Claim}}

\newtheorem{Lemma}[theorem]{Lemma}
\newtheorem{corollary}[theorem]{Corollary}

\theoremstyle{definition}
\newtheorem{step}{Step}
\newtheorem{substep}{Substep}

\newtheorem{remark}[theorem]{Remark}

\newtheorem{definition}[theorem]{{\bf Definition}}
\newcommand{\D}[1]{{\mathbb#1}}
\newcommand{\calu}{{\mathcal U}}

\title{Higher-rank instanton sheaves on Fano threefolds}

\author{Gaia Comaschi}
\address{Gaia Comaschi, \newline
\indent Université de Pau et des Pays de l'Adour, CNRS, LMAP UMR 5142, F-64013 Pau, France
}
\email{gaia.comaschi@gmail.com}

\author{Daniele Faenzi}
\address{Daniele Faenzi, \newline
\indent
Université Bourgogne Europe, CNRS, IMB UMR 5584, F-21000 Dijon, France
}
\email{daniele.faenzi@ube.fr}

\begin{thanks}{D. F. partially supported by FanoHK ANR-20-CE40-0023, SupToPhAG/EIPHI ANR-17-EURE-0002, Capes/COFECUB Ma 926/19 and Ma1017/24.
D. F. and G. C. partially supported by Bridges "Brazil-France interplays in Gauge Theory, extremal structures and stability": ANR-21-CE40-0017.
}\end{thanks}


\begin{document}
\emergencystretch=3em

\maketitle

\begin{abstract}
    We define higher-rank instanton sheaves on smooth Fano threefolds $X$ of Picard rank one and show that their topological type depends on two integers, namely the rank $n$ (or the half of it, if the Fano index of $X$ is odd) and the charge $k$. We determine the minimal charge $k_0$ of slope-stable $n$-instanton bundles, with the possible exception of Fano threefolds of index one and degree $H_X^3 < 4$, as an integer depending only on $H_X^3$ and $n$ and we prove the existence of slope-stable $n$-instanton bundles of charge $k \ge k_0$.

    Next, we study several properties of a general element of an irreducible component, which we call the main component, of the moduli space of instanton sheaves.
    This component is defined inductively on the rank and the charge, a natural procedure based on elementary transformations along rational curves, starting with minimal instantons. 
    The questions we address include the generic splitting over certain rational curves contained in $X$, as well as stable restriction to a K3 section $S$ of $X$. We obtain applications to Lagrangian subvarieties of moduli spaces of sheaves on $S$.

    Finally, we study the acyclic extension associated with an instanton sheaf on Fano threefolds with curvilinear Kuznetsov component and give a monadic description in the case $H^3(X,\Z)=0$.
\end{abstract}

\setcounter{tocdepth}{1}
\tableofcontents

\section{Introduction}
The notion of instanton bundle in higher-dimensional algebraic geometry appeared in the context of Yang-Mills gauge theory. Indeed, via Penrose's twistor transform, anti self-dual solutions (gauge instantons) of the Yang-Mills equation on the 4-dimensional sphere $S^4$ correspond, up to a reality condition, to certain complex algebraic  vector bundles on the complex projective space $\D{P}^3$, which we call instanton bundles.
This has been carried out originally in \cite{ADHM}, where a special emphasis is put on the case of instantons of rank $2$ and on the resulting interpretation of the moduli space of instantons in terms of linear algebra data.
The ADHM data has been thoroughly investigated to study instantons as a higher-rank analogue of the Hilbert scheme on smooth projective surfaces, where deep results have been obtained on the topology of the moduli space of instantons
in connection with Donaldson invariants, see for instance \cite{gottsche-nakajima,nakajima-yoshioka}.

In higher dimension, the situation is less clear. Several notions of instanton sheaves have been explored in various settings.
Some definitions have been given in higher rank (see \cite{Jardim}) and over some other varieties besides the projective space, see e.g. \cite{Faenzi,BF,Kuz, CA, malaspina-marchesi-pons}. One may extend the definition to include sheaves which are not locally free or even move to the derived category and look at complexes of sheaves, see \cite{CM, CJMM}.
The definition of instanton, in all these generalizations, essentially consists of three main ingredients: the choice of topological invariants, constraints on the behavior of the intermediate cohomology and a notion of stability.

However, in many cases the very basic question of non-emptiness of instanton moduli spaces remains unclear. Notably, the question of whether there exist stable instanton bundles for given choices of topological invariants is unsettled in many cases, notably for Fano threefolds.

Moreover, once non-emptiness of the moduli space is proved for a given topological datum, it is natural to ask what are the main features of a sufficiently general element in one component of the moduli space.
Among these questions, one could ask about the splitting type of the bundle on rational curves on $X$, e.g. does the bundle have \textit{generic splitting}? Trivial splitting over a generic line is conjectured \cite[Conjecture 3.16]{Kuz} for an instanton bundle of rank $2$ when $i_X=2$.
Further natural questions are smoothness and irreducibility of the moduli space of instanton bundles. A positive answer for $\p{3}$ is provided in \cite{jardim-verbitsky} for instantons with trivial generic splitting. Irreducibility of the moduli space of instanton bundles of rank $2$ on $\p{3}$ is proved in \cite{Tikhomirov:even,Tikhomirov:odd}. These questions have been explored in some more specific situations, for example symplectic bundles on $\p 3$, see \cite{bruzzo-markushevich-tikhomirov}, or in special cases, e.g. see \cite{andrade-santiago-silva-sobral} for instantons of rank 3 and charge 2 on $\p 3$.
\medskip

In this article, we focus on instanton sheaves on Fano threefolds of Picard rank one and partial answers to some of the questions above, mainly about non-emptiness of the moduli space of slope-stable instanton bundles of given rank and charge. We actually construct an irreducible component of this moduli space, whose general element is unobstructed and well-behaved with respect to restrictions to rational curves and to anti-canonical sections.

\medskip

Let us start with our definition of instanton sheaf -- we refer to Remark \ref{instanton-conditions} for a discussion of related notions.
\begin{definition} \label{def:instanton}
Let $X$ be a Fano threefold of Picard rank one and index \mbox{$i_X=2q_X+r_X$}, with $r_X \in \{0,1\}$. An \textbf{instanton sheaf} is a torsion-free sheaf $E$ on $X$ such that the following conditions hold.
\begin{enumerate}[label=\roman*)]
    \item \label{instanton-i} $\ch(E)=\ch(\RHom(E,\cccO_X(-r_X)))$;
    \item \label{instanton-ii}  $H^1(E(-q_X-t))=H^2(E(-q_X+t))=0$, for $t \in \{0,1\}$;
    \item \label{instanton-iii} $E$ is slope-semistable.
\end{enumerate}
\end{definition}
The first observation is that the topological constraints on the Chern characters imply that $\ch(E)$ depends on 2 integers, namely, the (half)rank and an invariant $k$ depending on the second Chern class, which we refer to as the \textbf{charge} of the instanton. As it turns out, an instanton $E$ can be of any rank when $i_X$ is even, while $E$ must be of even rank when $i_X$ is odd. With this in mind, we write $\rk(E)=n$ when $i_X$ is even and $\rk(E)=2n$ when $i_X$ is odd. Then one has
\begin{equation} \label{character of instanton}
\ch(E)=n \ch(F_0)-k \ch(\cccO_l(q_X-1)),
\end{equation}
where $l\subset X$ is a line and $F_0$ is a slope-stable instanton bundle of the minimal possible charge and rank. Let us mention that $F_0=\cccO_X$ when $i_X$ is even, while when $i_X$ is odd, $F_0$ is a rank-2 bundle playing an important role in the geometry of $X$. Namely, when $i_X=3$, so that $X$ is a quadric threefold, $F_0$ is the spinor bundle on $X$, while, for $i_X=1$, $F_0$ is a Mukai bundle, see \cite{BMK} for the most recent developments on this.

An instanton $E$ satisfying \eqref{character of instanton} will be referred to as an $(n,k)$-\textbf{instanton}. We will frequently write $\gamma(n,k)$ for the Chern character of an $(n,k)$-instanton. We will write $\calmi_X(n,k)$ for the moduli space of Gieseker-semistable $(n,k)$-instantons, seen as an open piece of Maruyama's moduli space $\calm_X(\gamma(n,k))$ of semistable sheaves on $X$ with Chern character $\gamma(n,k)$. Here, an $S$-equivalence class is an instanton if some representative of it is an instanton sheaf. We use a similar notation when working over an anticanonical section $S$ of $X$.

Studying non-emptiness of $\calmi_X(n,k)$ involves two main steps: the first is to identify which are the vectors $v\in \oplus_{i=0}^3 H^{i,i}(X,\D{Z})\simeq \D{Z}^4$ for which instantons $E$ on $X$ with $\ch(E)=v$ may exist. The second step involves the actual construction of the sheaves.
Constructing strictly slope-semistable non-locally free instantons presents no serious difficulty. Indeed, to this aim it is sufficient to consider kernels of surjections $E^n\twoheadrightarrow T$ for $T$ a purely one-dimensional sheaf with $H^i(T(-q_X))=0$, $i=0,1$ and $E$ a slope-stable rank-2 instanton bundle.
Much greater difficulties arise to show the existence of \textbf{locally free slope-stable} instantons.
Our main goal is to show that these difficulties can be overcome. Indeed, we show the following result.
\begin{mthm}\label{thm:main}
Let $X$ be a smooth Fano threefold with $\Pic(X) \simeq \langle H_X \rangle$ and $-K_X=i_XH_X$.
When the index $i_X$ of $X$ is one, put $g$ for the genus of $X$.
Let $k_0^n$ be the integer defined as:
\begin{itemize}
    \item[$(\heartsuit)$] $k_0^n= \lceil\frac{n}{2}\rceil$ if $i_X=4;$
    \item[$(\clubsuit)$] $k_0^n= \lceil\frac{n}{3}\rceil$ if $i_X=3;$
    \item[$(\diamondsuit)$] $k_0^n= n$  if $i_X=2$;
    \item[$(\spadesuit)$] $k_0^1=0, \ k_0^n=1$ if $n\ge 2$ for $g$ odd, resp. $k_0^n=n$ for $g$ even, if $i_X=1$.
\end{itemize}
Then, for all $n\ge 2-r_X$
and $k\ge k_0^n$, $X$ supports unobstructed slope-stable $(n,k)$-instanton bundles whose general splitting on lines is $0^n$ when $i_X$ is even, resp. $(0^n,-1^n)$ when $i_X$ is odd.

Moreover, for all $n \ge 2$, there are no slope-stable $(n,k)$-instanton sheaves with $k<k_0^n$,
except perhaps when $i_X=1$ and $g\le 3$ or $g=4$ and $X$ is contained in a singular quadric.
\end{mthm}

Here, the splitting of a rank-$n$ bundle $E$ over a rational curve $C$ is denoted by the string of integers $(d_1,\ldots,d_n)$, to indicate that $E|_C \simeq \oplus \cccO_{\p 1}(d_i)$. It should be noted that, over manifolds of dimension $3$ and higher, having a necessary and sufficient condition on the topological invariants that guarantees non-emptiness of the moduli space of stable vector bundles with such invariants is in general quite rare. Our result shows that moduli of instanton bundles are nicely behaved in this respect. Still, as mentioned, we don't know much about many basic properties of such moduli space -- singularities, connectedness, number of irreducible components, dimension of the components: all these points remain unclear.
However, we do construct a generically smooth component of the moduli space, whose general element behaves nicely in terms of stability of restriction, and, to some extent, splitting over rational curves and vanishing of cohomology.

The result is based on an inductive argument.
Starting from a slope-stable $(n,k)$-instanton bundle $E_n^k$, on a Fano threefold $X$, available by induction, we can indeed produce a strictly slope-semistable $(n+1,k)$ instanton $E_{n+1}^k$, whenever $\Ext^1(E_n^k,F_0) \ne 0$ --
this is actually always the case for values of $k$ greater than $k_0^n$.
We will then prove that such a $E_{n+1}^k$ deforms to a slope-stable instanton.

When $i_X > 1$ the results we obtain are somehow optimal, as we are able to determine, for each rank, the minimal value $k_0^n$ of the charge of an instanton bundle, and we construct slope-stable instanton bundles of charge $\ge k_0^n$.
Also, the bundles we construct this way admit global sections after a twist by $\cccO_X(1)$.
The base case of the induction is worth commenting. Indeed, it is given by rank-2 ’t Hooft bundles, namely rank-2 instanton bundles $E$ such that $h^0(E(1))>0$. Rank-2 't Hooft bundles of charge $\ge k_0$ on Fano threefolds of index $\ge 2$ can be obtained via the Serre correspondence. This is well-known for $i_X=3,4$, where the bundles are obtained from sets of disjoint lines. For Del Pezzo threefolds we will construct families of elliptic curves that give rise to rank-2 't Hooft instantons.

The case $i_X=1$, treated in \S \ref{section:prime}, turns out to be a bit more difficult, even for the rank 2. We will extend some known results on the existence of rank-2 instanton bundles of arbitrary charge to all Fano threefolds of index one, provided that $H_X$ is very ample. For rank $2$, the only point to address is the existence in arbitrary charge, which was known under the additional assumption that $X$ is ordinary, that is to say $X$ carries an unobstructed line.
We will then focus on the higher-rank range. For $n\ge 2$, letting $k_0^n=1$ for $g$ odd and $k_0^n=n$ for $g$ even, we will prove the existence of slope-stable $(n,k)$-instanton bundles for all  $k\ge k_0^n$. We are able to prove that $k_0^n$ is indeed the minimal value of $k$ for slope-stable instantons when $g\ge 4$, as the cases with $g\le 3$ present some extra difficulty that we do not address here. Again we proceed by induction, constructing slope-semistable instantons at first and then showing that these always deform to slope-stable ones. One difference compared to the previous cases essentially lies in the moduli space where such deformation takes place.
\medskip

Next, we look at the restrictions $E|_S$, where $E$ is an instanton on $X$ and $S$ is a K3 surface obtained as general anticanonical divisor. We show that for the slope-stable instanton bundles $E$ constructed via the techniques presented in the previous sections, $E|_S$ is slope-stable as well.
We show the following result.

\begin{mthm}
\label{mainthm:anticanonical-stable}
Let $X$ be a smooth Fano threefold of Picard number one, assume that if $i_X=1$, $g\ge 4$ and $X$ is not contained in a singular quadric for $g=4$. Let $S$ be a general element in the linear system $|\cccO_X(-K_X)|$. Then for all  $n,k$ with $n\ge 2-r_X, \ k\ge k_0^n$, the following holds.
\begin{enumerate}[label=\roman*)]
    \item There exists an open non-empty subscheme $\calr_S(n,k)\subset \calmi_X(n,k)$
parametrising $(n,k)$-instanton bundles $E$ on $X$ satisfying the statement of Theorem \ref{thm:main} and such that the restriction $E|_S$ is slope-stable;
\item the image of $\calr_S(n,k)$ in the moduli space $\calm_S(n,k)$ under the restriction map $E \mapsto E|_S$ is a Lagrangian subvariety for the standard symplectic structure on $\calm_S(n,k)$.
\end{enumerate}
\end{mthm}

We conclude the article by describing some properties of instantons on certain Fano threefolds $X$ whose derived category has a distinctive structure.
In \S \ref{section:curvilinear}, we consider the case of Fano threefolds whose derived category contains a component equivalent to the derived category of a curve $\Gamma$.
We say that $X$ has a curvilinear Kuznetsov component. These are the Del Pezzo threefold $Y_4$ and the prime Fano threefolds $X_g, \ g\in\{7,9,10\}$. We first show that slope-stable $(n,k)$-instantons $E$ on $X$ correspond to certain vector bundles on $\Gamma$. Then we prove that for $Y_4$ and $X_{10}$ these bundles are simple. On $X_7, \ X_9$ the same holds for $E$ belonging to a 
subscheme of $\calmi_X(n,k)$ which is dense in the component we constructed.
Finally, we describe the canonical resolutions of the instanton itself in terms of the Fourier-Mukai transform of these bundles.
\medskip

The paper is organised as follows.
In the next section, we provide some standard material on Fano threefolds and instantons.
In \S \ref{section:even} we look at constructions of instantons on Fano threefolds of even Fano index, first looking at $\p{3}$ and then to Del Pezzo threefolds, and show Theorem \ref{thm:main} in this case. In \S \ref{section:odd} we look at Fano threefolds of odd Fano index constructing instantons on them
and finishing the proof of Theorem \ref{thm:main}; this is first dealt with on a quadric threefold and then on prime Fano threefolds. Next we study stability of restriction to K3 sections in \S \ref{sec:restriction} and prove Theorem \ref{mainthm:anticanonical-stable}. Finally in \S \ref{section:curvilinear} and \S \ref{section:H3=0}, we shift our focus to Fano threefolds $X$ whose derived category has a nice description, either because $D^b(X)$ has a curvilinear Kuznetsov component (such category is equivalent to $D^b(\Gamma)$, with $\Gamma$ a smooth projective curve) or because
$D^b(X)$ has a full exceptional collection.
We prove that, on such varieties, instantons, just as in the rank-2 case, admit a description given via acyclic extensions in terms of sheaves on the curve $\Gamma$, or in terms of monads.

\medskip

We would like to thank Roberto Vacca for pointing out an incomplete argument in the first version of this paper.
We are grateful to the referees for suggesting many improvements to the manuscript.

\section{Generalities}\label{sect:gen}
\subsection{Preliminaries}
Let $X$ be a smooth Fano threefold of Picard rank 1. We denote by $H_X$ the ample generator of $\Pic(X)$.
We recall that the index of $X$ is the integer $i_X$ such that the anticanonical line bundle $K_X$ is isomorphic to $\cccO_X(-i_X).$ It is well-known that $i_X\in \{1,2,3,4\}$, essentially by Kobayashi-Ochiai's Theorem, \cite{kobayashi-ochiai:characterization}, see also \cite{algebraic-geometry-V}. Let us now denote by $q_X$ and $r_X$ the integers such that $i_X=2 q_X+ r_X$, with $r_X\in \{0,1\}$. Notice that $q_X\in\{0,1,2\}$.
We will often write the Chern classes of a coherent sheaf $E$ on $X$ as integers, omitting to mention the generator of the corresponding cohomology group, which we fix to be the class $H_X$ of a hyperplane section $H_X$, respectively, the class $l_X$ of a line, the class $p_X$ of a point, for $H^{1,1}(X)$, respectively, for $H^{2,2}(X)$, $H^{3,3}(X)$.

Given a coherent sheaf $E$ on $X$, we recall that $E$ is said to be Gieseker-\textbf{(semi)stable} if it is pure and, for any proper subsheaf $F\subset E$, the following inequality holds:
$$ p_F(t) \: (\le) < p_E(t).$$
Here, $p_E(t)$ denotes the \textbf{reduced Hilbert polynomial} of the sheaf $E$, namely its Hilbert polynomial divided by its rank. We refer to \cite[\S I.1.2]{HL} for more details on this notion.
If $E$ is torsion-free, we say that $E$ is \textbf{slope-(semi)-stable} if for every proper saturated subsheaf $F\subset E$, the following inequality holds:
$$ \mu(F):=\frac{c_1(F)}{\rk(F)}\: (\le) < \frac{c_1(E)}{\rk(E)}:=\mu(E).$$
According to the convention mentioned above, the first Chern class of a coherent sheaf $E$ is an integer, so that the \textbf{slope} $\mu(E)$ is a rational number.
Sometimes we will only say "stable" or "semistable" to mean "Gieseker-stable" or "Gieseker-semistable".

\medskip
We said that our notion of instanton sheaf is based on three conditions. The first one is about the Chern character, the second one is a cohomological vanishing, and the third one is a stability condition.
From the conditions \ref{instanton-ii} and \ref{instanton-iii} of Definition \ref{def:instanton} we deduce the following result.

\begin{Lemma} \label{basic-vanishing} Let $E$ be an instanton sheaf. Then the following assertions hold:
\begin{enumerate}[label=\roman*)]
    \item $h^i(E(-q_X))=0$, for all integers $i$;
    \item $h^1(E(-q_X-t))=h^2(E(-q_X+t))=0$, for all integers $t\ge 0$.
\end{enumerate}
\end{Lemma}

\begin{proof}
The slope-semistability assumption ensures that $\hom(\cccO_X(q_X),E)=0$ and $h^3(E(-q_X))=\hom(E,\cccO_X(-q_X-r_X))=0$, hence we have
\begin{equation} \label{orthogonal}
h^i(E(-q_X))=0, \qquad \mbox{for all integers $i$}.
\end{equation}
Consider now a general hyperplane section $S:=X\cap H$ for $H\in |\cccO_X(1)|$ and the short exact sequence
$$0\to E(-1)\to E\to E|_S\to 0.$$
The previous vanishing implies $h^0(E(-q_X-t))=h^3(E(-q_X+t))=0$, for all $t\ge 0$, hence from \ref{instanton-ii} of Definition \ref{def:instanton} we deduce, inductively, that $h^1(E(-q_X-t))=h^2(E(-q_X+t))=0$ for all integers $t\ge 0$.
\end{proof}

\begin{remark}
The conditions \ref{instanton-i} and \ref{instanton-ii} of Definition \ref{def:instanton} can be rephrased in categorical terms.
Indeed we consider the bounded derived category of coherent sheaves $D^b(X)$.
If $E$ is an instanton sheaf on $X$, then \eqref{orthogonal} says that
$$E \in \cccO_X(q_X)^{\perp}:=\{ F\in D^b(X)\mid \Ext^\bullet(\cccO_X(q_X),F)=0 \}.$$

Moreover, the functor $\Phi: F\mapsto \RHom(F,\cccO_X(-r_X))$ induces an anti-equivalence of $\cccO_X(q_X)^\perp$. Indeed, if $F$ is an object of $\cccO_X(q_X)^\perp$, then, for all integers $i$:
\[
\Ext^i(\cccO_X(q_X),\Phi(F)) \simeq \Ext^i(F,\cccO_X(-q_X-r_X)) \simeq \Ext^{3-i}(\cccO_X(q_X),F)^*=0
\]
Hence, Definition \ref{def:instanton} tells us that an instanton sheaf $E$ belongs to $\cccO_X(q_X)^{\perp}$ and that $\Phi(E)$ is an object in $\cccO_X(q_X)^{\perp}$ such that $\ch(E)=\ch(\Phi(E)).$
\end{remark}

\begin{remark} \label{instanton-conditions}
Let us discuss briefly the conditions of Definition \ref{def:instanton}, also 
with respect to similar notions in the literature.
\begin{enumerate}[label=\roman*)]
    \item Let $E$ be a torsion-free sheaf satisfying item \ref{instanton-i} and \ref{instanton-iii} of Definition \ref{def:instanton}. For all the applications envisaged here, including Theorems \ref{thm:main} and \ref{mainthm:anticanonical-stable}, as well as Theorems \ref{thm:curve} and \ref{prop-monad},
    item \ref{instanton-ii} of Definition \ref{def:instanton} can be replaced by the more economical condition
    \[
    h^1(E(-q_X))=0, \qquad \mbox{if $i_X \in \{3,4\}$},
    \]
    or, in the case $i_X \in \{1,2\}$, by:
    \[
    h^1(E(-q_X))=h^1(E(-1-q_X))=h^2(E(1-q_X))=0, 
    \]
    Indeed, $\chi(E(-q_X))=0$ follows from item \ref{instanton-i} and slope-semistability easily implies $h^0(E(-q_X))=h^3(E(-q_X))=0$, so $h^2(E(-q_X))=0$ follows. Also,
    by slope-semistability, $h^0(E(1-q_X))=0$ if $i_X \ge 3$, which gives back item \ref{instanton-ii} of Definition \ref{def:instanton}, by the basic argument of Lemma \ref{basic-vanishing}.
    \item Let $E$ be a torsion-free sheaf satisfying item \ref{instanton-i} of Definition \ref{def:instanton}. Then, conditions \ref{instanton-ii} and \ref{instanton-iii} are satisfied if 
    \(
    h^1(E(-q_X))=0,
    \)
    and $E|_S$ is slope-stable,
    for some hyperplane section $S$ of $X$. Semistability on 
    $S$ is intended with respect to the restriction of $H_X$ to $S$.
    If $E$ is slope-semistable of rank two and $S$ is general enough, by \cite{Maruyama:small}, $E|_S$ is slope-semistable, cf. Remark \ref{rmk:rank2}.
    \item A different definition of instanton sheaf was given in \cite{CA}, encompassing previous work, e.g. \cite{Faenzi,jardim-miro-roig,Kuz}. A comparison between that notion and our definition (at least for locally free instanton sheaves of rank 2), is given in \cite[\S 1]{CA}. The two notions are similar, but not identical in general.
    
    In summary, if $E$ is an instanton sheaf according to Definition \ref{def:instanton}, then $\cale = E(2-q_X)$ is an instanton in the sense of \cite{CA} if $i_X \in \{3,4\}$. For 
    $i_X \in \{1,2\}$, this happens under the further condition $h^0(E(1-q_X))=h^3(E(-1-q_X))=0$. Note that this is not satisfied for minimal instantons if $i_X \in \{1,2\}$. For $i_X=1$, we see that actually instanton bundles $E$ as defined here will not satisfy $h^0(E(1))=0$ if the charge is relatively small with respect to the rank.
        
        Conversely, if $\cale$ is an \textit{ordinary instanton} on $X$, in the sense of \cite{CA}, which is slope-semistable, then $E=\cale(q_X-2)$ is an instanton sheaf according to Definition \ref{def:instanton}.
        The two notions of charge, however, are typically off by a constant that depends on the topological invariants of $X$.
    \item Later on we will also use torsion instantons, see Definition \ref{def:rank0-inst}. These have been studied on Fano threefolds notably in \cite{CM}. 
\end{enumerate}
\end{remark}

\subsection{Rank-2 instantons}
Before dealing with the higher-rank range, which is the main focus of this article, we recall some basic facts about the rank-2 instantons. This case has been largely investigated, notably on $\p 3$, starting with \cite{ADHM}, where certain special bundles, called ’t Hooft instantons, are also introduced. 
Our goal in this subsection is to introduce a similar construction for Fano threefolds of Picard number one.

\begin{remark}\label{rmk:rank2}
In the rank-2 case, the cohomological conditions of Definition \ref{def:instanton} can be reduced just to $H^i(E(-q_X))=0, \ i=1,2$. This is because, since $\rk(E)=2$, given a general hyperplane section $S\subset X$, the sheaf $E|_S$ is still slope-semistable, according to \cite{Maruyama:small}.
Therefore $h^0(E(-n))=h^3(E(-i_X+n))=0$ for all $n\ge 0,$
which allows one to show, inductively:
\[H^1(E(-q_X-n))=H^2(E(-q_X+n))=0, \qquad \forall\: n\ge 0.\]
If moreover $E$ is locally free it actually suffices to require $H^1(E(-q_X))=0$ since by Serre duality, and as $E^*\simeq E(r_X)$, $H^2(E(-q_X))\simeq H^1(E(-q_X))^*$.
\end{remark}

We recall an effective and widely applied technique to construct rank-2 bundles on a threefold $X$, namely \textbf{Serre's correspondence}. It establishes a correspondence between pairs $(F,s)$ for $F$ a rank-2 vector bundle and $s\in H^0(F)$ a global section with torsion-free cokernel and pairs $(e, C)$ for $C$ a locally complete intersection curve (l.c.i. for short) and $e\in \Ext^1(\cali_C, \cccO_{X})\simeq H^0(\omega_C(d))$ a nonvanishing extension class corresponding to:
$$ 0\rightarrow \cccO_{X}\xrightarrow{s} E\rightarrow \cali_C(c_1(E))\rightarrow 0;$$
On a side note, we mention that this can be generalised to torsion-free sheaves, see \cite{CM}.

In what follows we will need to construct rank-2 slope-stable and \textbf{unobstructed} instantons $E$, namely $\Ext^2(E,E)=0$. This can be done via the Serre correspondence starting from l.c.i. curves that are unobstructed as well, which in this context means that $h^1(\caln_{C/X})=0$.
We recall that the unobstructedness of a curve $C\subset X$ is actually equivalent to the unobstructedness of its sheaf of ideals $\cali_C$. This is due to the following result.
\begin{Lemma}\label{lem:curve-unobstructed}
Let $C$ be a locally complete intersection curve on a smooth projective threefold $X$ such that $H^1(\cccO_X)=0$. Then $H^i(\caln_C)\simeq \Ext^{i+1}(\cali_C,\cali_C)$ for $i=0,1$.
\end{Lemma}
\begin{proof}
Consider the exact sequence
$$0\to \cali_C\to \cccO_{X}\to \cccO_C\to 0.$$
Applying
$\inhom(\cali_C, \: \cdot \: )$ to it and recalling that the sheaf $\cali_C$ has homological dimension 1, we get an exact sequence of $\cccO_C$ modules:
\begin{equation}\label{eq:local-norm}
0\rightarrow \caln_C\rightarrow \inext^1(\cali_C,\cali_C)\rightarrow \inext^1(\cali_C, \cccO_{X})\xrightarrow{\gamma} \inext^1(\cali_C,\cccO_C)\rightarrow 0.
\end{equation}
We also get that
$\inext^1(\cali_C,\cccO_{X})\simeq \inext^2(\cccO_C,\cccO_{X})\simeq \omega_C(-K_X)$ is a line bundle on $C$ and the same holds for $\inext^1(\cali_C, \cccO_C)\simeq \inext^2(\cccO_C,\cccO_C)$. Therefore the surjection $\gamma$ in $\eqref{eq:local-norm}$ must actually be an isomorphism. As a consequence $\caln_C\simeq \inext^1(\cali_C,\cali_C)$.
Using $H^1(\cccO_X)=0$ and applying the local to global spectral sequence
$$H^p(\inext^q(\cali_C,\cali_C))\Rightarrow \Ext^{p+q}(\cali_C,\cali_C)$$
and $\Ext^{i+1}(\cali_C,\cali_C)\simeq H^i(\inext^1(\cali_C,\cali_C))\simeq H^i(\caln_C)$, for all $i \ge 0$.
\end{proof}

We use the Serre correspondence to construct
rank-2 \textbf{'t Hooft bundles}. These are rank-2 instanton bundles with $h^0(E(1))>1$, which arise from locally complete intersection curves $C$ of arithmetic genus $p_a(C)=3-2(q_X+r_X)$, i.e. $p_a(C)=-1$ for $i_X=3,\: 4$ and $p_a(C)=1$ for $i_X=1,2$.
This is a generalization of the classical 't Hooft instantons of \cite{ADHM}, which arise when $C$ is the union of disjoint lines in $\p 3$. Incidentally, we recall that classical 't Hooft bundles have a distinguished description in terms of the corresponding solutions of the Yang Mills equations, see \cite{Atiyah} -- we don't know if anything similar could be said for the bundles we construct here.

\begin{Lemma}\label{lem:rank2-thooft-i}
Let $X$ be a Fano threefold of Picard rank one of index $i_X\in\{3,4\}$, resp. $i_X\in\{1,2\}$, $C\subset X$ a l.c.i. curve of arithmetic genus $p_a(C)=1-\deg(C)$, resp. a connected l.c.i curve of $p_a(C)=1$, and $E$ a rank-2 vector bundle arising from a nonvanishing extension
\[e\in \Ext^1(\cali_C(1-r_X),\cccO_X(-1))\simeq H^0(\cccO_C)^*.\]
Then $E$ is a rank-2 't Hooft bundle and the following hold:
\begin{enumerate}[label=\roman*)]
    \item \label{unob-i} if $C$ is unobstructed then $E$ is unobstructed;
    \item  \label{unob-ii} if $i_X$ is odd, $E$ is slope-stable;
    \item  \label{unob-iii} if $i_X$ is even, $E$ is Gieseker-semistable if and only if it is slope-stable if and only if $C$ is non-degenerate.
    \item \label{unob-iv} the general line $l\subset X$ is not jumping for $E$, that is to say, $h^1(E|_l(-1))=0.$
\end{enumerate}
\end{Lemma}

\begin{proof}
Consider the short exact sequence defined by the extension class $e$:
\begin{equation}\label{eq:rank2-thooft}
0\rightarrow \cccO_X(-1)\rightarrow E\rightarrow \cali_C(1-r_X)\rightarrow 0.
\end{equation}

We have $H^1(E(-q_X))=H^1(\cali_C(1-r_X-q_X))$. Indeed, if $i_X\in\{3,4\}$, then $r_X+q_X=2$ hence this term vanishes. If $i_X\in\{1,2\}$, then $r_X+q_X=1$ and we get $H^1(\cali_C)=0$ since we are assuming that $C$ is connected. We can thus conclude that $E$ is an instanton bundle.

If $i_X$ is odd, since $c_1(E)=-1$, the slope-semistability of $E$ is equivalent to its slope-stability which is equivalent to $H^0(E)=0.$ This always holds since $H^0(\cali_C)=0.$
If $i_X$ is even, $E$ is slope-semistable, resp. slope-stable, if and only if $H^0(E(-1))=0=H^0(\cali_C)$, resp. $H^0(E)=0=H^0(\cali_C(1))$.
Since a direct computation shows that $p_{\cccO_{X}}>p_E$ we conclude that $E$ is Gieseker-semistable if and only if it is slope-stable which holds if and only if $C$ is not contained in any hyperplane.

Let us now prove \ref{unob-i}.
From Lemma \ref{lem:curve-unobstructed}, the unobstructedness of $C$ is equivalent to $\Ext^2(\cali_{C},\cali_{C})=0$.
Then, we apply the functors $\Hom(\cdot, E)$ and $\Hom(\cali_{C}, \cdot)$ to the short exact sequence \eqref{eq:rank2-thooft}.
Using Lemma \ref{basic-vanishing}, this gives the exact sequences:
\begin{align*}
\Ext^1(\cccO_X(-1), E)& \to
 \Ext^2(\cali_{C}(1-r_X),E) \to \Ext^2(E,E)\to 0 \\
\Ext^2(\cali_{C}(1-r_X),\cccO_X(-1))&\to \Ext^2(\cali_{C}(1-r_X),E)\to \Ext^2(\cali_{C}, \cali_{C})\to 0.
\end{align*}

Note that
 $\Ext^2(\cali_{C}(1-r_X),\cccO_X(-1))^*\simeq H^1(\cali_{C}(2-i_X-r_X))$. For $i_X\in \{3,4\}$ we find $H^1(\cali_C(-2))=0$ whilst for $i_X \in \{1,2\}$ we have $H^1(\cali_C)=0$, as $C$ is connected.
 Therefore if we assume $C$ unobstructed, $\ext^2(\cali_C,\cali_C)=0$ by Lemma \ref{lem:curve-unobstructed}, so that $\ext^2(\cali_C(1-r_X), E)=0$ and $\ext^2(E,E)=0$ as well.
\medskip

It remains to show that a general line is not jumping for $E$.
When $i_X$ is odd, take a line $l\subset X$ disjoint from the curve $C$, which is an open condition on the family of lines on $X$. Restricting \eqref{eq:rank2-thooft} to $l$ we get that $E|_l$ is an extension of $\cccO_l$ by $\cccO_l(-1)$ so that $E|_l\simeq \cccO_l\oplus \cccO_l(-1).$
When $i_X$ is even, we choose a line $l$ meeting $C$ transversely at one point $p\in C$.
This time, restricting \eqref{eq:rank2-thooft} to $l$, we get:
$$ 0\rightarrow \cccO_l(-1)\rightarrow E|_l\rightarrow \cccO_l\oplus \cccO_p\rightarrow 0.$$
Since $E|_l$ is locally free, the kernel of the induced surjection $E|_l\twoheadrightarrow \cccO_l\oplus \cccO_p\twoheadrightarrow\cccO_l$ must be $\cccO_l$ yielding $E|_l\simeq \cccO_l^2.$ By semicontinuity the same will hold for the general line on $X$.

\end{proof}

The Serre correspondence is particularly effective to construct instanton bundles of low charges,
in particular those having minimal charge (meaning the minimal one allowed by the constraints in \ref{def:instanton}), or having charge slightly greater than the minimal one.
In the upcoming sections we will describe in detail the behaviour of the slope-stable rank-2 instantons having the lowest possible charge. We will see that these bundles are actually 't Hooft instantons and for each value of $i_X$ we will also describe the corresponding curves on $X$.

\subsection{Instanton bundles of higher rank} We choose to focus now our attention on instanton sheaves of higher rank. 

\subsubsection{The Chern character of instantons}
To begin with, let us look at their Chern character. 
From now on we will denote by $F_0$ a minimal slope-stable instanton on $X$, in the sense that it has the minimal possible values of rank and charge admitted by Definition \ref{def:instanton}.
As already pointed out in the previous section, $F_0\simeq \cccO_{X}$ for $i_X$ even whilst $F_0$ is a locally free rank-2 instanton for $i_X$ odd. Note that the main features of these bundles in the odd index cases will be summarized in Sections \ref{section:quadric} and \ref{section:prime}.

\begin{Lemma}\label{lem:chern}
Let $E$ be an instanton. Then there exists a pair of integers $(n,k)$ such that
\begin{equation}\label{eq:chern-char}
\ch(E)=n \ch(F_0)-k\ch(\cccO_l(q_X-1)),
\end{equation}
for $F_0$ a minimal instanton.
\end{Lemma}

\begin{proof}
Let us start with the case $i_X$ even. For an instanton $E$, the condition $\ch(E)=\ch(\RHom(E,\cccO_{X}))$ implies that the odd Chern classes of $E$ vanish. Hence $\ch(E)$ has the form $(n, 0, -k, 0)=n \ch(F_0)-k\ch(\cccO_X(q_X-1))$ for $n=\rk(E)$ and $k=c_2(E)$.

Assume now that $i_X$ is odd and let $\ch_j(E)$ be the $j$-th Chern character of a nonzero instanton sheaf $E$, for $j \in \{0,\ldots,3\}$. The condition $\ch(E)=\ch(\RHom(E,\cccO_X(-1))$ imposes
\begin{equation} \label{conditions odd}
 c_1(E)=-\frac{\rk(E) H_X}{2}, \qquad \ch_3(E)=-\frac{c_1(E)H_X^2}{12}-\frac{\ch_2(E)}{2}.
\end{equation}
Note that the first equation implies that the rank of $E$ must be even, hence $\rk(E)=2n$ for some positive integer $n$.
Let now $E^m$ be a rank-2 instanton of second Chern class $c_2(E^m)=m.$ From the conditions above $c_1(E^m)=-H_X$ and $c_3(E^m)=0$. Then, going back to our arbitrary instanton sheaf $E$ of rank $2n$, we check directly that, under the condition \eqref{conditions odd}, we have
\[
\ch(E)-n\ch(E^m)=-k_m l_X+\frac{k_m}{2}p_X=-k_m\ch(\cccO_l(q_X-1)),
\]
where $k_m=\ch_2(E)-n\ch_2(E^m)$ is an integer that depends on $m$. Setting $E^m=F_0$, we obtain the formula in display in \eqref{eq:chern-char}.
\end{proof}

\begin{remark}\label{rmk:chern-symm}
The invariance of the Chern character of an instanton $E$ with respect to the functor \mbox{$E\mapsto \RHom(E,\cccO_X(-r_X))$} also tells us that $\chi(E,F_0)=\chi(F_0,E)$.
\end{remark}

From now on, an instanton $E$ on $X$ whose Chern character can be expressed as in \eqref{eq:chern-char} will be referred to as a $(n,k)$-\textbf{instanton}.
The Chern character $\ch(E)\in \bigoplus_{i \ge 0} H^{i,i}(X,\mathbb{Z})$ of a $(n,k)$-instanton will be denoted by $\gamma(n,k).$
The number $k$ will be referred to as the \textbf{charge} of the instanton. This will coincide with the second Chern class of $E$ when $i_X$ is even whilst when $i_X$ is odd we have the relation
$$
k=\frac{nd_X}{2}-\frac{n^2d_X}{2} -n c_2(F_0)+c_2(E), \qquad \mbox{where $d_X=H_X^3$}.
$$

\subsubsection{Existence of slope-semistable instanton sheaves of higher rank}

The main existence issue we will deal with is the following: for which values of $n$ and $k$, does $X$ carry an $(n,k)$-instanton? If one just wants to look at slope-semistable, not necessarily locally free $(n,k)$-instantons, it is actually quite easy to answer this question. Sheaves $F$ of this kind can indeed be obtained by taking kernels of surjections $E\twoheadrightarrow T$ for $E$ and $T$, respectively, a slope-semistable and a \textbf{torsion instanton}, in which case we will say that $F$ is the \textbf{elementary transformation} of $E$ along $T$. We refer to \cite{CM} for a more detailed discussion of the features of the instanton sheaves obtained using this technique.
\begin{definition}\label{def:rank0-inst}
Let $X$ be a Fano threefold of Picard rank one. A \textbf{torsion} instanton is a coherent sheaf $T$ of pure dimension one and such that $h^i(T(-q_X))=0$ for $i=0,1$. The integer $d=\deg(\supp(T))$ is called the \textbf{degree} of $T$.
\end{definition}

\begin{Lemma}\label{lem:et}
Let $X$ be a Fano threefold of even (respectively odd) index. Then for all $n\ge 2 $ (respectively, for all $n\ge 1$) and $k\ge 0$, $X$ carries a strictly slope-semistable $(n,k)$-instanton.
\end{Lemma}
\begin{proof}
Instantons $E$ satisfying the Lemma can be obtained via elementary transformation $F_0^n\twoheadrightarrow T$, for $T$ a torsion instanton of degree $k$.
A torsion instanton $T$ of such a kind always exists: it will be enough to take $T\simeq \bigoplus_{i=1}^k\cccO_{l_i}((q_X-1))$ for disjoint lines $l_1,\ldots, l_k$ on $X$, with the additional assumption that none of these lines is jumping for $F_0$, i.e. $h^1(F_0|_{l_i}(-1))=0$,
for $i_X$ odd. That a general line on $X$ indeed satisfies this condition is ensured by the fact that $F_0$ is a 't Hooft bundle,  see Lemma \ref{lem:rank2-thooft-i}. A more detailed description of the corresponding curves will be presented later on.

A sheaf $E$ obtained in this way is slope-semistable since this holds for its double dual, as indeed we have $E^{**}\simeq F_0^n$. Since the kernel of the composition $F_0\hookrightarrow F_0^n\to T$ injects into $E$, $E$ is strictly slope-semistable.
It is immediate to check that $h^1(E(-q_X-n))=H^2(E(-q_X+n))=0$ for $n\in\{ 0,1\}$, so we conclude that $E$ is a $(n,k)$-instanton.
\end{proof}

\subsubsection{Minimal charge of slope-stable instanton bundle of higher rank}

Things become more intricate after strengthening our stability assumption. When we pass to the higher-rank cases we can immediately observe that checking the stability of sheaves becomes typically much more tricky.
Here, we look for \textbf{slope-stable instanton bundles}, which is among the strongest conditions that one can desire. To accomplish this task we go through two main steps; in the first place, for each $n$, we need to determine the minimal value of $k$ for which we might have existence $(n,k)$-instantons. If we assume $i_X\ge 2$, whenever $n$ is fixed, it is easy to establish a first lower bound for the value of $k$ that must be respected if  we hope to have the existence of $(n,k)$-instantons that are at least Gieseker-semistable.

\begin{Lemma}\label{lem:min-k-gen}
Let $E_n^k$ be a 
$(n,k)$-instanton on a smooth Fano threefold of Picard number one and index $i_X\ge 2$.
Assume that $E_n^k$ is Gieseker-semistable if $i_X$ is even index and that $E_n^k$ is slope-semistable if $i_X=3$.
Then $k\ge \lceil \frac{n}{q_X}\rceil$ and $k\ge 0$, respectively.
\end{Lemma}
\begin{proof}
The inequality is obtained imposing $\chi(E_n^k)\le 0$. Indeed, if $\chi(E_n^k)>0$, then $E_n^k$ admits a non-zero global section, because
Lemma \ref{basic-vanishing} gives
$h^2(E_n^k)=0$. But this is impossible due to the semistability assumption on $E_n^k$.
\end{proof}

Trying to establish a lower bound for $k$ when $i_X=1$ in a similar way is not effective. Indeed, when $i_X=1$, any instanton sheaf $E$ verifies $\chi(E)=0$.

We will see that in the even index cases, this lower bound is actually sharp; whilst for the odd indices further constraints will be required.

The second main step consists in the actual construction of instantons for all values of $k$ greater than the minimal one.

\begin{remark}\label{rmk:et}
The technique of the elementary transformation that we applied in Lemma \ref{lem:et} is also used in \cite{Faenzi} to prove the existence of rank-2 instanton bundles of high charge on Del Pezzo threefolds. Performing the elementary transformation of a slope-stable instanton bundle $E$ along a line $l\subset X$ which is not a jumping line of $E$, we get
$$ 0\rightarrow F\rightarrow E\rightarrow \cccO_l\rightarrow 0.$$
Here, $F$ is slope-stable non-locally free instanton sheaf of charge $c_2(F)=c_2(E)+1$, with $F^{**}\simeq E$. 
 
 As it turns out, the sheaf $F$ deforms flatly to a locally free instanton. This is proved showing that the dimension of the non-locally free deformations of $F$ is strictly less than the dimension of the Gieseker-moduli space at $[F]$. We will recall the key ingredients of this argument in \S \ref{section:prime}, where we will adapt it to prove the existence of rank-2 instanton bundles on Fano threefolds of index 1,  see Proposition \ref{prop:rank2-index1}.

It is natural to ask if a similar approach can be exploited in the higher rank cases,
but unfortunately the answer is negative.
The reasons for the effectiveness of this technique in the rank-2 case are essentially two. First, the fact that for any non-locally free rank-2 instanton $E$, the sheaves $E^{**}$ and $E^{**}/E$ are, respectively, an instanton bundle and a torsion instanton. This affords an estimate of the dimension of the non-locally free deformations of $E$. Second, for rank-2 instantons, local freeness is equivalent to reflexivity.
But both these properties no longer hold for rank greater than 2, see \cite{CM} for details.
\end{remark}

\subsubsection{Slope-destabilizing subsheaves and unobstructedness}

Starting from a slope-stable $(n,k)$-instanton $E_n^k$ on a Fano threefold $X$, we will construct a strictly slope-semistable $(n+1,k)$-instanton $E_{n+1}^k$ arising from an extension $\Ext^1(E_{n}^k,\cccO_{X})$ for $i_X$ even, resp. $\Ext^1(E_{n}^k, F_0)$ for $i_X$ odd.
The fact that $E_{n+1}^k$ is obtained in this way enables us to classify all its slope-destabilizing subsheaves.
This is due to the following result.
\begin{Lemma}\label{lem:ext-ss}
Let $X$ be an $m$-dimensional projective variety of Picard rank one and let $E$ be a sheaf arising from a non-zero extension $e\in \Ext^1(F, G),$ with $F$ and $G$ slope-stable vector bundles with $\mu(F)=\mu(G)=\mu(E)$ and $G\not\simeq F$. Then $E$ is a strictly slope-semistable vector bundle and $F$ is, up to isomorphism, the unique torsion-free quotient of $E$ having slope $\mu(E).$
\end{Lemma}

\begin{proof}
By construction $E$ fits in a short exact sequence
$$0\rightarrow G\rightarrow E\rightarrow F\rightarrow 0.$$
So $E$ cannot be slope-stable, since it contains $G$. If  $E$ is unstable it would admit a slope-semistable subsheaf $L$ with $\rk(L)<\rk(E)$, $\mu(L)>\mu(E)$. But the slope-stability of $F$ implies that $\Hom(L,F)=0$ so that $L$ would inject in $G$ which is clearly impossible as $G$ is slope-stable as well. Therefore $E$ is strictly slope-semistable.

To prove the lemma we need to show that every subsheaf $L\hookrightarrow E$ of slope $\mu(E)$ such that $E/L$ is torsion-free is equal to $G$.
For such a subsheaf, which must be slope-semistable, consider the composition $L\to F$, denote by $Q$ its image and by $K$ the kernel of $L\twoheadrightarrow Q$.

Let us suppose that $Q\ne 0.$
Since $F$ is slope-stable, we must have $\mu(Q)=\mu(L)=\mu(K)=\mu(G)=\mu(L)=\mu(E)$.
Now, if $K\ne 0$, we must have $G\simeq K$, for otherwise $G/K$ would be a torsion subsheaf of $E/L$. So $E/L\simeq F/Q$ which is torsion unless $F\simeq Q$, so $E\simeq L$ as well which is not possible since $L$ is a proper subsheaf of $E$.

Then we would have $K=0$, so $L\simeq Q$ and $E/L$ would be an extension of $F/L$ by $G$. Note that $F/L\ne 0$, since otherwise $E\twoheadrightarrow G$ which would imply that the sequence splits. Note that, since $F$ is slope-stable and $\mu(F)=\mu(L)$, we must have $\rk(F)=\rk(L)$ and $c_1(F)=c_1(L)$, 
so $\codim(\supp(F/L))\ge 2$. Hence $\Ext^1(F/L,G)\simeq \Ext^{m-1}(G,F/L\otimes \omega_X)^*\simeq H^{m-1}(G^*\otimes F/L\otimes \omega_X)^*=0$ which again contradicts the torsion-freeness of $E/L$.

Then we must have $Q=0$, which gives $K\simeq L$ hence $L\hookrightarrow G$ so again, since $G$ is slope-stable and $\mu(L)=\mu(G),$ we get $\rk(L)=\rk(G)$, whence $G/L$ is torsion.
But $G/L$ injects into $E/L$, which is torsion-free, so $G = L$.
\end{proof}

The way we use the previous lemma is to characterize the strictly slope-semistable deformations of $E_{n+1}^k$, in the sense that a general deformation will also admit a slope-stable $(n,k)$-instanton as its unique torsion-free quotient having the same slope. This allows one to determine the dimension of the family of strictly slope-semistable deformations of $E_{n+1}^k$. Based on this, we will prove that $E_{n+1}^k$ deforms to a slope-stable instanton.

In order to control the deformations of $E_{n+1}^k$, we need to check its unobstructedness. Once again this will be proved inductively: starting from an unobstructed $(n,k)$-instanton $E_n^k$ on a Fano threefold of even index, resp. $E_{n}^k$ such that $\Ext^2(F_0,E_n^k)=\Ext^2(E_n^k,F_0)=0$ on a Fano threefold $X$ of odd index, we get the unobstructedness of the extensions of $E_{n}^k$ by $\cccO_{X}$, resp. $F_0$.
\begin{Lemma}\label{lem:ext-unobstructed}
Let $E$ and $F$ be coherent sheaves on a scheme $X$. Suppose that $E$ and $F$ are unobstructed and that $\Ext^2(E,F)=0=\Ext^2(F,E)$.
Let $e\in \Ext^1(F,E), \ e\ne 0$ be an extension class corresponding to a non-split exact sequence:
$$ 0\rightarrow F\rightarrow E'\rightarrow E\rightarrow 0.$$
Then $E'$ is unobstructed.
\end{Lemma}
\begin{proof}
Since $\Ext^2(F,F)=0=\Ext^2(E,F)$, we have that $\Ext^2(E',F)$ vanishes as well. Therefore $\Ext^2(E',E')$ injects into $\Ext^2(E',E)$.
But the vanishing of $\Ext^2(F,E)$ implies that $\Ext^2(E,E)$ surjects onto $\Ext^2(E',E)$ so that the latter must be zero as we are assuming $E$ unobstructed. We thus conclude that $\Ext^2(E',E')=0$.
\end{proof}

If $i_X\ne 1$, we will be able to prove the existence of slope-stable instanton bundles that moreover admit global sections after a twist by $\cccO_X(1)$.
We remark that for most values of $k$, $(n,k)$-instantons $E$ such that $h^0(E(1))\ne 0$ form proper closed subsets of their moduli space.

\section{Instantons on Fano threefolds of even index}

\label{section:even}

We will begin by treating the case when $i_X$ is even.
Recall that in this context $F_0\simeq \cccO_{X}$ and the Chern character $\gamma(n,k)$ of a $(n,k)$-instanton $E_n^k$ has the form:
\begin{equation} \label{chern-n-k}
    \gamma(n,k)=n \ch(\cccO_{X}) -k \ch(\cccO_l(q_X-1))=(n,0,-k,0).
\end{equation}

The integer $n$ is just the rank of $E$ while $k$ is the second Chern class of $E$. We saw that the instanton conditions impose $k\ge \lceil \frac{n}{q_X}\rceil$, see Lemma \ref{lem:min-k-gen} -- recall that $q_X=2$ if $X=\p 3$ and $q_X=1$ if $X$ is a Del Pezzo threefold.

We will show that this lower bound is sharp, constructing inductively slope-stable $(n,k)$-instanton bundles for all $k\ge \lceil\frac{n}{q_X}\rceil$. These will be obtained by deforming strictly slope-semistable $(n,k)$-instantons arising as extensions of a slope-stable $(n-1,k)$ instanton by $\cccO_{X}$.

\subsection{Instantons of high rank on $\D{P}^3$}

We first direct our attention towards the case $X=\D{P}^3$, that is to say, $i_X=4$.
The first thing that we need to do is to recall the well-known construction of slope-stable unobstructed rank-2 instanton bundles of charge $k\ge 1$ that will provide us the base of our inductive argument. We just pick the construction from Lemma \ref{lem:rank2-thooft-i}, through unobstructed curves of degree $k+1$ and arithmetic genus $-k$, for instance, the union of $k+1$ disjoint lines, giving "classical" 't Hooft bundles.

\begin{Lemma}\label{lem:rank2-thooft+unobstructed}
For $k>0,$ let $C_k$ be the union of $k+1$ disjoint lines and let $E_2^k$ be a 't Hooft bundle arising from a non-zero extension class in $\Ext^1(\cali_{C_k}(1),\cccO_{\p{3}}(-1))$. Then $E_2^k$ is slope-stable and unobstructed and has trivial splitting on a sufficiently general line.
\end{Lemma}
\begin{proof}
This is simply Lemma \ref{lem:rank2-thooft-i}. The only thing left to check is the unobstructedness of $C_k$; but this is immediate since, if $C_k$ is supported on $l_0,\ldots ,l_k$, then $\caln_{C_k} \simeq \bigoplus_{i=0}^k \cccO^2_{l_i}(1)$ hence $H^1(\caln_{C_k/\D{P}^3})=0.$
\end{proof}

The minimal charge for a Gieseker-semistable instanton is 1. Moreover, for a rank-2 instanton $E$ of charge 1, $\chi(E(1))>0$. We conclude that a minimal rank-2 instanton indeed is a 't Hooft instanton. For higher rank, we have the following result.

\begin{theorem}\label{thm:exist-P3} Let $n,k$ be integers with $n \ge 2$ and $k \ge \lceil \frac{n}{2} \rceil$.
Then there exists a slope-stable unobstructed $(n,k)$-instanton bundle $E$ on $\D{P}^3$
satisfying the following:
\begin{enumerate}[label=\roman*)]

\item\label{T1}
the cokernel sheaf $F$ of a general morphism $\cccO_{\D{P}^3}(-1)\xrightarrow{s} E$
is unobstructed, slope-stable, torsion-free, of homological dimension at most 1, and satisfies:
\begin{align*}
    &\ext^i(F,\cccO_{\D{P}^3}(-j))=0, && \mbox{for $i\ne 1, \: j=0,1;$} \\
   &h^i(F)=0, && \mbox{for $i\ne 1$},\\
    &h^i(F(1))=0,&& \mbox{for $i\ge 2$}.
\end{align*}
\item \label{T2} For all integers $q$ with $1\le q\le n-1$ we have $H^0(\bigwedge ^q E)=0$.
\item \label{T3} For a general line $l\subset \D{P}^3$, $E$ has trivial splitting over $l$, namely $E|_l\simeq \cccO_l^{n}$.
\end{enumerate}
\end{theorem}

Before passing to the actual proof of the Theorem let us summarize its main steps.
To begin with we review the base of the induction (Step \ref{step1}), given by rank-2 't Hooft instanton bundles satisfying the Theorem.

We then pass to the inductive step:
from a $(n,k)$-instanton bundle $E_n^k$ satisfying the Theorem, we construct a strictly slope-semistable, but Gieseker-unstable $(n+1,k)$-instanton $E_{n+1}^k$ from a non-split extension class in $\Ext^1(E_n^k,\cccO_{\p{3}})$, see Step \ref{step2}.

Next, we lift a general $s\in H^0(E_n^k(1))$ to $H^0(E_{n+1}^k(1))$ and we prove that the cokernel $F_{n+1}^k$ of the corresponding morphism $\cccO_{\p{3}}(-1)\to E_{n+1}^k$ is unobstructed and slope-stable (see Steps \ref{step3} and \ref{step4}). By construction $F_{n+1}^k$ arises from a non-zero element in $\Ext^1(F_n^k, \cccO_{\p{3}})$ for $F_n^k:=\coker(s).$

In Step \ref{step5} we show that $F_{n+1}^k$ deforms, in the Gieseker-moduli space $\calm_{\D{P}^3}(\ch(F_{n+1}^k))$, to a sheaf admitting no quotient of type $F_n^k$.

Finally, via universal extensions, we will construct a projective bundle over  $\calm_{\D{P}^3}(\ch(F_{n+1}^k))$ that will play the role of the moduli space of extensions of sheaves in
$\calm_{\D{P}^3}(\ch(F_{n+1}^k))$ by $\cccO_{\p{3}}(-1)$ (see Step \ref{step6}). Thanks to the conclusions drawn in Step 5, we show that, in this moduli space, $[E_{n+1}^k]$ deforms to a slope-stable $(n+1,k)$-instanton bundle.

\begin{proof}
We prove the theorem by induction on the rank $n$.
\begin{step}[Check the base of the induction] \label{step1}
For $n=2$, the theorem holds due to the existence of the rank-2 't Hooft bundles. By Lemma \ref{lem:rank2-thooft+unobstructed}, 't Hooft $E_2^k$ of rank-2 is unobstructed, slope-stable (which is equivalent to satisfying \ref{T2}) and satisfies \ref{T3}. A non-zero global section $s$ of $E_2^k(1)$ vanishes on a locally complete intersection curve $C_k$ of degree $k+1$ and arithmetic genus $-k$.
These are supported on $\le k+1$ disjoint lines and will consist of the union of $k+1$ disjoint lines, generically. We checked that the curves $C_k$ of this kind are unobstructed and $\cali_{C_k}(1)$ satisfies \ref{T1}.
\end{step}

\begin{step}[Construction of an instanton of higher rank]  \label{step2}
Let us now pass to the proof of the induction step.
Consider integers $n\ge 2$ and $k\ge \lceil \frac{n+1}{2}\rceil$. By induction there exists an instanton bundle $E_{n}^k$ of rank $n$, charge $k$ and satisfying the theorem. We have $\ext^3(E_{n}^k, \cccO_{\p{3}})=0$ by stability and $\ext^2(E_{n}^k,\cccO_{\p{3}})=0$ by \ref{T2}. Since $\chi(E_{n}^k,\cccO_{\p{3}})=n-2k<0$ we deduce that there
exists a non-split extension:
\begin{equation}\label{eq:semistable}
    0\rightarrow \cccO_{\p{3}}\rightarrow E_{n+1}^k\rightarrow E_{n}^k\rightarrow 0.
\end{equation}
with $E_{n+1}^k$ a strictly slope-semistable (and Gieseker-unstable) instanton bundle.

We claim that $E_{n+1}^k$ deforms to an instanton bundle satisfying the statement of the theorem. To begin with, we observe that, due to Lemma \ref{lem:ext-unobstructed}, $E_{n+1}^k$ is unobstructed. Note also that, since $E_{n}^k$ satisfies \ref{T3}, the same holds for $E_{n+1}^k$, as to check it, it suffices to restrict \eqref{eq:semistable} to a general line $l\subset \D{P}^3$.

We take now a general global section $s'\in H^0(E_{n}^k(1))$ and we lift it to a global section $s\in H^0(E_{n+1}^k(1))$. Since $s$ does not lie in the image of $H^0(\cccO_{\p{3}}(1))\hookrightarrow H^0(E_{n+1}^k(1))$, $\coker(s)$ is a torsion-free sheaf $F_{n+1}^k$.
\end{step}
\begin{step}[Check the unobstructedness of $F_{n+1}^k$] \label{step3}
By construction, the sheaf $F_{n+1}^k$ fits in a commutative diagram:
\begin{equation}\label{cd-non-inst-onemore}
\begin{tikzcd}
&  &\cccO_{\p{3}}(-1)\arrow[r, equal]\arrow[d, "s"] &\cccO_{\p{3}}(-1)\arrow[d, "s'"]   \\
0  \arrow[r] & \cccO_{\p{3}}\arrow[r]\arrow[d, equal] & E_{n+1}^k\arrow[r]\arrow[d, equal] & E_{n}^k\arrow [r]\arrow[d]& 0\\
0\arrow[r] &\cccO_{\p{3}}\arrow[r] & F_{n+1}^k\arrow[r]& F_{n}^k \arrow[r]
& 0
\end{tikzcd}
\end{equation}
Since $s'$ is general, we may assume that $F_{n}^k$ satisfies \ref{T2}.
From the diagram we compute that $F_{n+1}^k$ has homological dimension $\le 1$ and that $F_{n+1}^k$ satisfies all cohomological conditions of \ref{T1} except for $h^0(F_{n+1}^k)$, as actually $h^0(F_{n+1}^k)=1$.
Notice that since $F_{n}^k$ satisfies \ref{T1}, $F_{n+1}^k$ is unobstructed by Lemma \ref{lem:ext-unobstructed}.
\end{step}

\begin{step}[Show that $F_{n+1}^k$ is slope-stable] \label{step4}
We use the last row of \eqref{cd-non-inst-onemore}. Suppose, by contradiction, that $F_{n+1}^k$ is not slope-stable. Then, since $c_1(F_{n+1}^k)=1$, we have a  slope-semistable subsheaf $L$ of $F_{n+1}^k$  with $c_1(L)\ge 1$. The sheaf $L$ maps to zero via the map $F_{n+1}^k\to F_{n}^k$, since $F_{n}^k$ is slope-stable of first Chern class 1, while the $c_1$ of the image of $L$ in $F_{n}^k$ must be at least one. But then $L\to F_{n+1}^k$ would factor through $\cccO_{\p{3}}$, which is clearly impossible.
\end{step}

\begin{step}[Deform $F_{n+1}^k$ to a slope-stable sheaf having no quotient of type $F_n^k$] \label{step5}
Next, we claim that $F_{n+1}^k$ deforms to a slope-stable sheaf having no global sections.
A general deformation $F$ of $F_{n+1}^k$ is an unobstructed slope-stable sheaf of homological dimension $\le 1$, satisfying $\ext^i(F,\cccO_{\p{3}}(-j))=0$ for $i\ne 1$, $h^i(F(j))=0$ for $i\ge 2$ and $j=0, 1$. Let us check that $h^0(F)=0$. Assuming the contrary, we get an exact sequence
\begin{equation}\label{ses}
0\rightarrow \cccO_{\p{3}} \rightarrow F\rightarrow F'\rightarrow 0.
\end{equation}
Here, $F'$ has $c_1(F')=c_1(F)=1$ and admits no quotients of degree $\le 0$, due to the slope-stability of $F$, hence it is slope-stable. From \eqref{ses} we deduce that, for all $j \in \{0,1\}$, one has $\ext^i(F,\cccO_{\p{3}}(-j))=0$ for $i\ne 1$ and $h^i(F(j))=0$ for $i\ge 2$. We can readily prove that, as $F$ is unobstructed, the same holds for $F'$. Thus the sheaf $F'$ represents a smooth point in  \mbox{$\calm_{\D{P}^3}(v)$}, the moduli space of Gieseker-semistable torsion-free sheaves having Chern character $v=\ch(F_{n}^k)$.

The assumption that a general deformation $F$ has nonvanishing global sections would then lead to $\ext^1(F_{n+1}^k,F_{n+1}^k) \le \ext^1(F_{n}^k, F_{n}^k)+ \ext^1(F_{n}^k,\cccO_{\p{3}})-1$. However, this is a contradiction, for a direct computation shows, by our choice of $k$:
\begin{align*}
    \ext^1(F_{n+1}^k,F_{n+1}^k)-\ext^1(F_{n}^k,F_{n}^k)- \ext^1(F_{n}^k,\cccO_{\p{3}})-1&=\\
     -\chi(F_{n+1}^k,F_{n+1}^k)+\chi(F_{n}^k,F_{n}^k)+ \chi(F_{n}^k,\cccO_{\p{3}})-1&=\chi(\cccO_{\p{3}},F_{n}^k)=\chi(E_{n}^k)>0.
\end{align*}
\end{step}
\medskip
\begin{step}[Use the universal extension to construct a slope-stable deformation of $E_{n+1}^k$] \label{step6}
Summarizing the previous steps, we proved that $F_{n+1}^k$ deforms to a slope-stable sheaf having no global sections, so we may consider a family of sheaves $\mathbf{F}$ over $U\times \p{3}$, where $U$ is  an appropriate open subset $U$ of $\calm_{\D{P}^3}(\ch(F_{n+1}^k))$, specializing to $F_{n+1}^k$ and all sheaves $\mathbf{F}_t$ corresponding to points $t\in U$ satisfy \ref{T1}. Let us use it to construct a family $\mathbf{E}$ specializing $E_{n+1}^k$ and whose general element satisfies the theorem.
Consider the relative Ext sheaf $$\cale:=\inext^1_{p_1}(\mathbf{F},\cccO_{U \times \D{P}^3}(-1)).$$
Here, $p_1$ is the projection onto the first factor, $\inext^1_{p_1}(\,\cdot\,,\,\cdot\,)=R^1({p_1}_*\inhom(\,\cdot\,, \,\cdot\, ))$
and for a coherent sheaf $\calf$ on $U\times \D{P}^3$ we define $\calf(n):=\calf\otimes {p_2}^*\cccO_{\p{3}}(n)$.
Since $\Ext^2(\mathbf{F}_t,\cccO_{\p{3}}(-1))=0$ for all $t$ in $U$, we get $\inext^2_{p_1}(\mathbf{F},\cccO_{U\times \D{P}^3}(-1))=0$ and $\inext^1_{p_1}(\mathbf{F},\cccO_{U\times \D{P}^3}(-1))$ commutes with base change (see \cite[Theorem 1.4]{Lange} and \cite[Theorem A.5]{Kleiman}).

Moreover, since for all $t\in U$, $\ext^1(\mathbf{F}_t,\cccO_{\p{3}}(-1))=-\chi(\mathbf{F}_t,\cccO_{\p{3}}(-1))$, $\inext^1_{p_1}(\mathbf{F},\cccO_{U\times \D{P}^3}(-1))$ is locally free on $U$ hence ${p_1}_*(\inhom(\mathbf{F}, \cccO_{U\times \D{P}^3}(-1))$ commutes with base change as well (due to \cite[Theorem 1.4]{Lange}) and is, a fortiori, the zero sheaf.

Applying now \cite[Corollary 4.5]{Lange}, we get the existence of a universal extension on $\D{P}(\cale)\times \D{P}^3$:
\begin{equation}
0\rightarrow \cccO_{\D{P}(\cale)\times \D{P}^3}\otimes {p_1}^*\cccO_{\D{P}(\cale)}(-1)\rightarrow \mathbf{E}\rightarrow \hat{\mathbf{F}}\rightarrow 0
\end{equation}
where the sheaf $\hat{\mathbf{F}}$ on $\D{P}(\cale)\times \D{P}^3$ is the pullback of $\mathbf{F}$.

We obtain a family $\mathbf{E}$ of instanton bundles over $\D{P}(\cale)$. In view of the unobstructedness of $E^k_{n+1}$ and semicontinuity of Ext, up to shrinking $U$, we can assume that this family consists of unobstructed instantons.
By construction, a general element of the family satisfies \eqref{T1}.

Finally let us show that, for $x$ general in $\D{P}(\cale)$ we have $ H^0(\bigwedge^q \mathbf{E}_x)=0$ for $q=1,\ldots, n-1$. Taking exterior powers of the exact sequence \eqref{eq:semistable}, we obtain exact sequences:
$$ 0\rightarrow \bigwedge ^{q-1}E^k_{n}\rightarrow \bigwedge^q E_{n+1}^k\rightarrow \bigwedge ^{q}E^k_{n}\rightarrow 0$$
for $2\le q\le n-1$. Due to the inductive hypothesis, we deduce the vanishing of $H^0(\bigwedge ^q E_{n+1}^k)$ for $2\le q\le n-1$. Still from \eqref{eq:semistable}, we also check that $H^0(\bigwedge ^{n} E_{n+1}^k)\simeq \Hom(E_{n+1}^k,\cccO_{\p{3}})=0$. Indeed, applying the functor $\Hom(\:\cdot\:,\cccO_{\p{3}})$,  we get that, by construction, the induced morphism $\Hom(\cccO_{\p{3}},\cccO_{\p{3}})\to \Ext^1(E_{n}^k,\cccO_{\p{3}})$ is injective, for this is nothing but the map sending the identity to the extension class defining \eqref{eq:semistable}. 

Consider now a general deformation $\mathbf{E}_x$ of $E^k_n$. By semicontinuity $H^0(\bigwedge ^q \mathbf{E}_x)=0$ for $2\le q\le n-1$.
In addition to this, we have that, by construction, $\mathbf{E}_x$ fits in
$$ 0\rightarrow \cccO_{\p{3}}(-1)\rightarrow \mathbf{E}_x \rightarrow F\rightarrow 0,$$
with $F=\mathbf{F}_{\pi(x)}
$, where $\pi : \D{P}(\cale) \to U$ is the projection. Since we are assuming $x$ general, the sheaf $F$ will have no nonzero global section. Therefore the same holds for $\mathbf{E}_x$. So, finally, $H^0(\bigwedge ^q \mathbf{E}_x)=0$ for $1\le q\le n-1$. This implies slope-stability of $\mathbf{E}_x$ by the Hoppe's criterion, see for instance \cite[Theorem 3]{jardim-menet-prata-sa-earp}.
\end{step}
\end{proof}

As an immediate consequence of the theorem, we obtain that the moduli space $\calmi_{\D{P}^3}(n,k)$ of $(n,k)$-instanton bundles is non-empty and that it has an irreducible component whose general point corresponds to a slope-stable and unobstructed bundle. We deduce the following result.

\begin{corollary}
For all $n\ge 2, \ k\ge \lceil \frac{n}{2}\rceil$, there exists a generically smooth $(1-n^2+4nk)$-dimensional component of the moduli space $\calmi_{\D{P}^3}(n,k)$ of $(n,k)$-instantons whose general point is a slope-stable locally free instanton having general splitting $0^n$.
\end{corollary}

Looking at the sheaves arising as cokernels of the (twisted) global sections of the instanton bundles constructed in the previous theorem, we get the following result.

\begin{corollary}
For all $n\ge 2, \ k\ge \lceil \frac{n}{2}\rceil$, the moduli space \mbox{$\calm_{\D{P}^3}(\gamma(n,k)-\ch(\cccO_{\p{3}}(-1)))$} of Gieseker-semistable torsion-free sheaves with Chern character \mbox{$\gamma(n,k)-\ch(\cccO_{\p{3}}(-1))$}
has a generically smooth component of dimension $n(4-n)+4k(n-1)$ whose general point is a slope-stable sheaf.
\end{corollary}

\subsection{Instanton bundles of high rank on Del Pezzo threefolds}

Now, we focus on Fano threefolds of index $2$, called Del Pezzo threefolds, and Picard number 1.
We recall that there exists five families of deformation equivalent threefolds $Y_d$ of this kind, distinguished by the degree $d$ of the ample generator of the Picard group, with $d \in \{1,2,3,4,5\}$. For the rest of the section, $Y$ will denote any Fano threefold of index $2$ and $d$ will denote its degree.

As we pointed out, on a Fano threefold $Y_d$ of index $2$, the Chern character $\gamma(n,k)$ of a $(n,k)$-instanton $E_n^k$ has the same form as on $\D{P}^3,$ $\gamma(n,k)=(n, 0, -k, 0).$
We will prove that, for all $n\ge 2$, $Y$ carries slope-stable $(n,k)$-instanton bundles for all $k\ge n$. This was the lower bound for $k$ determined in Lemma \ref{lem:min-k-gen}.
We adopt an inductive method almost equivalent to the one used on the projective space.
Also this time the base of the induction will consist of the rank-2 bundles analogous to 't Hooft bundles, as shown in Lemma \ref{lem:rank2-thooft-i}. They are obtained via Serre's correspondence from elliptic curves.
More precisely a rank-2 't Hooft bundle of charge $k$ corresponds to a l.c.i. unobstructed and non-degenerate elliptic curve of degree $d+k$.
Given an integer $d$, we let $\Hilb_{dt}(Y)$ be the Hilbert scheme of closed subschemes of $Y$ having Hilbert polynomial $dt$, so that curves (i.e. purely one-dimensional closed subschemes of $Y$) of arithmetic genus $1$ and degree $d$ give points of this Hilbert scheme.
We will prove the following result:
\begin{proposition}\label{prop:elliptic}
For all $k\ge 2$, there exist a generically smooth $2(d+k)$-dimensional component $\mathcal{H}_{d+k}$ of $\Hilb_{(d+k)t}(Y)$ whose generic point is a smooth elliptic curve which is non-degenerate and unobstructed.
\end{proposition}

We will prove the proposition by induction on $k$.
From a smooth, unobstructed elliptic curve $C\subset Y$ of degree $d_C$, we  construct a smooth, unobstructed elliptic curve of degree $d_C+1$, by \textit{smoothing} a nodal curve $C'$ consisting of the union of $C$ and of an \textbf{ordinary line}, i.e. such that $\caln_l\simeq \cccO_l^2$, meeting $C$ quasi-transversely (meaning with distinct tangent vectors) at a point $p$.

\begin{Lemma}\label{lem:smoothing-elliptic}
Let $C\subset Y$ be a smooth, unobstructed elliptic curve on $Y$ and let $l\subset Y$ be an ordinary line meeting $C$ quasi-transversely at a point.
Then $C\cup l$ deforms to a smooth and unobstructed curve.
\end{Lemma}
\begin{proof}
We apply \cite[Theorem 4.1]{HH} and we refer to loc. cit. for further details.
Denote by $C':=C\cup l$ and by $p$ the intersection point $p:=C\cap l$. The normal sheaf $\caln_{C'}$ of $C'$ fits in a short exact sequence:
\begin{equation}\label{eq:normal-union}
    0\rightarrow \caln_{C'}\rightarrow \caln_{C'}|_C\oplus \caln_{C'}|_l\rightarrow \caln_{C'}\otimes \cccO_p\rightarrow 0.
\end{equation}
Let us consider the restriction $\caln_{C'}|_C$. This is a rank-2 vector bundle on $C$ fitting in:
\begin{equation}\label{eq:elm-elliptic}
0\rightarrow \caln_C\rightarrow \caln_{C'}|_C\rightarrow \cccO_p\rightarrow 0;
\end{equation}
since by assumption $C$ is unobstructed, $H^1(\caln_{C'}|_C)=0$ and $H^0(\caln_{C'}|_C)\to H^0(\cccO_p)$ is surjective.
According to \cite{HH}, the restriction $\caln_C'|_l$ fits in an exact sequence of $\cccO_l$ modules:
\begin{equation}\label{eq:elm-line}
0\rightarrow \caln_{C'}|_l(-p)\rightarrow \caln_{C'}|_l\rightarrow \caln_{C'}\otimes\cccO_p\rightarrow 0;
\end{equation}
since also $\caln_{C'}|_l$ arises from an extension class in $\Ext^1(\cccO_p,\caln_l)$ and since $l$ is ordinary, $h^1(\caln_{C'}|_l(-p))=0$ so that
$h^1(\caln_{C'}|_l)=0$ as well and $H^0(\caln_{C'}|_l)\to H^0(\caln_{C'}\otimes \cccO_p)$ is surjective.
These results, together with (\ref{eq:normal-union}), lead us to conclude that $H^1(\caln_{C'})=0$ and $H^0(\caln_{C'})\to H^0(\cccO_p)$ is surjective.
Then, applying  \cite[Proposition 1.1]{HH}, we conclude that $C'$ is a smoothable and unobstructed elliptic curve.
\end{proof}


\begin{proof}[Proof of Proposition \ref{prop:elliptic}]
We argue by induction on $k$. For $k=2$ this can be found e.g. in \cite[Proof of Theorem D]{Faenzi}. We sketch here the key steps leading to the result. To begin with we consider a general hyperplane section $S$ of $Y$, so $S$ is a smooth Del Pezzo surface of degree $d$. It is easy to construct a smooth degree $d+2$ elliptic curve $C$ on $S$ that is unobstructed on $Y$. This is done realising $S$ as the blow-up of $\D{P}^2$ along $9-d$ points $p_1, \ldots, p_{9-d}$, and pulling back a general plane cubic passing through $7-d$ among $p_1,\ldots,p_{9-d}$. For such a curve we have $h^0(\caln_C)=2d+4$ and $h^1(\caln_C)=0$. However, a dimension count shows that the family of elliptic curves of degree $d+2$ which are contained in a hyperplane section has dimension $2d+3$. Accordingly, $C$ belongs to a generically smooth $(2d+4)$-dimensional component $\mathcal{H}_{d+2}$ whose general element is a smooth, unobstructed and non-degenerate curve.

Let us now pass to the inductive step. For $k>2$, consider a smooth and unobstructed curve $C$ of degree $d_C:=d+k$, corresponding to a general point in $\mathcal{H}_{d_C}$.
Recall that for a general point $p\in Y$ we have a finite number of lines passing through $p$. Since $C$ is general, we can find an ordinary line $l$ meeting $C$ transversely at a point $p\in C$. The union $C'=C\cup l$ is an elliptic curve of degree $d_C+1$. According to Lemma \ref{lem:smoothing-elliptic}, this curve is unobstructed.  From the short exact sequences \eqref{eq:normal-union}, \eqref{eq:elm-elliptic} and \eqref{eq:elm-line}, we compute $h^0(\caln_{C'})=h^0(\caln_{C'}|_C)+h^0(\caln_{C'}|_l)-2=h^0(\caln_C)+h^0(\caln_l)=2(d_C+1)=2(d+k+1)$.
Then $C'$ defines a smooth point of a $2(d+k+1)$-dimensional component $\mathcal{H}_{d+k+1}$ of $\Hilb_{(d+k+1)t}$. Since, still according to Lemma \ref{lem:smoothing-elliptic}, $C'$ deforms to a smooth curve $\Tilde{C}$, a general point of $\mathcal{H}_{d+k+1}$ corresponds to a smooth elliptic curve on $Y$. Finally, since $C'$ is unobstructed and non-degenerate, assuming $\Tilde{C}$ general enough, we have that, by semicontinuity, the same holds for $\Tilde{C}$.
\end{proof}

\begin{Lemma}\label{lem:rank-2}
For all $k\ge 2$ there exists an unobstructed slope-stable rank-2 instanton bundle $E_2^k$ fitting in a short exact sequence
\begin{equation}\label{eq:'t Hooft-index2}
0\rightarrow \cccO_{Y}(-1)\rightarrow E\rightarrow \cali_C(1)\rightarrow 0
\end{equation}
with $C$ a smooth, non-degenerate and unobstructed curve in $\mathcal{H}_{d+k}$. Moreover such a bundle $E_2^k$ has splitting $(0,0)$ on the general line.
\end{Lemma}
\begin{proof}
Once again this is a consequence of Lemma \ref{lem:rank2-thooft-i}.
\end{proof}

Again, actually all minimal rank-2 instantons $E$ (i.e. rank-2 instantons having charge 2) are 't Hooft instantons, since $\chi(E(1))>0$.
We can finally adapt the argument of Theorem \ref{thm:exist-P3} to check the existence of slope-stable instanton bundles of rank $r$ for all $r\ge 2$.
We omit the detailed proof since there is no major modification with respect to Theorem \ref{thm:exist-P3}.

\begin{theorem} \label{thm:existence-even}
For all $n\ge 2$ and $k\ge n$ there exists a slope-stable unobstructed $(n,k)$-instanton bundle $E$ on $Y$ satisfying the following:
\begin{enumerate}[label=\roman*)]
\item
the cokernel sheaf $F$ of a general map $\cccO_{Y}(-1)\xrightarrow{s} E$ is an unobstructed slope-stable torsion-free sheaf of homological dimension at most 1 such that \begin{align*}
    &\ext^i(F,\cccO_{Y}(-j))=0 && \mbox{for all $i\ne 1$ and $j \in \{0,1\}$}\\
    &h^i(F)=0, && \mbox{for all $i\ne 1$}.
\end{align*}
\item For any integer $q$ with $1\le q\le n-1$, we have $h^0(\bigwedge ^q E)=0$.
\item For a general line $l\subset Y,\ E|_l\simeq \cccO_l^n$.
\end{enumerate}

\end{theorem}
\begin{proof}
The steps of the proof are equivalent to the ones we used for Theorem \ref{thm:exist-P3}.
\end{proof}

\begin{corollary}
For all $n\ge 2, \ k\ge n$, there exists a generically smooth  $(2kn-n^2+1)$-dimensional component of the moduli space $\calmi_Y(n,k)$ of $(n,k)$-instanton bundles whose general point is a slope-stable locally free instanton having general splitting $0^n$.
\end{corollary}

\begin{corollary}
For all $n\ge 2, \ k\ge n$, the Gieseker moduli space \mbox{$\calm_Y(\gamma(n,k)-\ch(\cccO_{Y}(-1)))$} of Gieseker-semistable torsion-free sheaves with Chern character $\gamma(n,k)-\ch(\cccO_{Y}(-1))$
admits a generically smooth component of dimension $n(d+2-n)+2k(n-1)$ whose general point is a slope-stable sheaf.
\end{corollary}

\section{Instantons on Fano threefolds of odd index}

\label{section:odd}

Now we deal with Fano threefolds of odd index. In these cases the invariance of the Chern character with respect to $\RHom(\ \cdot\ , \cccO_{X}(-r_X))$ imposes several constraints. Notably, the rank $r$ of an instanton sheaf must be even, we write $r=2n$. 
To prove the existence of slope-stable instanton bundles, we will need to modify non-trivially the methods applied in the even-index cases.
The central idea remains to construct strictly slope-semistable sheaves that are not Gieseker-stable, more specifically having $F_0$ as their maximal destabilizing subsheaf, and then to prove that these deform to slope-stable ones.
For $i_X$ odd the strictly slope-semistable sheaves that we construct will arise from extensions of slope-stable instantons by a minimal instanton $F_0$, which this time will no longer be decomposable. Note that this is consistent with the requirement $r=2n$.

\subsection{Instantons of high rank on 3-dimensional quadrics}

\label{section:quadric}

We start by treating the case of the smooth quadric hypersurface $Q\subset \D{P}^4$, the only smooth Fano threefold of index 3.
On $Q$, the minimal instanton is the \textbf{spinor bundle} $\cals$. This is the only rank-2 slope-stable instanton bundle of charge 1.
The Chern character $\gamma(n,k)\in \oplus_{i=1}^3 H^{i,i}(Q,\mathbb{Z})\simeq \D{Z}$ of a $(n,k)$-instanton $E_n^k$ on $Q$
has the form:
$$\gamma(n,k)=n\: \ch(\cals) - k\:\ch(\cccO_l)=\left(2n,-n,-k,\frac{k}{2}-\frac{n}{6}\right).$$

We recollect some well-known features of the spinor bundle that will be needed later -- for further details we refer e.g. to \cite{AS,Ott}.
The bundle $\cals$ is exceptional, namely it is simple and $\Ext^p(\cals,\cals)=0$ for $p=1,2,3$.
Via the Serre correspondence, $\cals$ can be obtained from any line $l\subset Q$ and a non-zero extension $e\in \Ext^1(\cali_l,\cccO_{Q}(-1))\simeq H^1(\cccO_l(-2))\simeq \C$:
\begin{equation}\label{eq:Serre-spinor}
0\rightarrow \cccO_{Q}(-1)\rightarrow \cals \rightarrow \cali_l\rightarrow 0.
\end{equation}

The spinor bundle $\cals$ is uniform, namely, given any line $l\subset Q$,
we have $\caln_{l/X}\simeq \cals(1)|_l\simeq \cccO_l\oplus \cccO_l(1)$ hence $\cals$ has splitting type $(-1,0)$ on $l$.
The Hilbert scheme of lines $\Hilb_{t+1}(Q)$ is identified with $\D{P}(H^0(\cals(1)))\simeq \D{P}^3$.
The dual $\cals^*\simeq \cals(1)$ is an Ulrich bundle, that is to say it satisfies $H^i(\cals(-i))=0$ for $i<3$ and $H^i(\cals^*(-i))=0$ for $i>0$.
The bundle $\cals^*$ is globally generated and the surjective evaluation map $\cccO_{Q}^4\simeq \cccO_{Q}\otimes H^0(\cals(1))\to \cals(1)$ induces an embedding $Q\hookrightarrow \mathbb{G}(2,4)$ whose image is a hyperplane section of $\mathbb{G} (2,4)$. This way, $\cals$ is identified with the pullback of the tautological rank-2 vector bundle $\calu$. Pulling back the tautological exact sequence on $\mathbb{G} (2,4)$ one sees that $\cals$ fits in:
\begin{equation}\label{eq:spinor-Euler}
0\rightarrow \cals \rightarrow \cccO_{Q}^4\rightarrow \cals^*\rightarrow 0.
\end{equation}

In the next lemma, we determine the minimal charge of Gieseker-stable $(n,k)$-instantons.

\begin{Lemma}\label{lem:min-quadric}
Let $E_n^k$ be a Gieseker-stable $(n,k)$-instanton on $Q$, not isomorphic to $\cals$. Then
$$k\ge\lceil \frac{n}{3}\rceil.$$
\end{Lemma}

\begin{proof}
Applying Riemann-Roch we compute $\chi(\cals, E_{n}^k)=\chi(E_{n}^k,\cals)= n-3k$
and $\chi(E_{n}^k)= -k$.
By definition $H^2(E_{n}^k)=0$. Also, stability of the involved sheaves implies $\Ext^3(\cals^*, E_{n}^k)=0$.

Then, from \eqref{eq:spinor-Euler} we compute $\Ext^2(\cals, E_{n}^k)=0$. This means that whenever $\chi(\cals, E_{n}^k)>0$ there exists a morphism $\cals\to E_{n}^k.$
Looking at this morphism one sees that, if $k\ge 0$, a necessary condition to have slope-semistability (cf. Lemma \ref{lem:min-k-gen}) is $p_{\cals}\ge p_{E_{n}^k}$ with equality holding if and only if $k=0$.

This proves that, unless $E_n^k\simeq \cals$,  Gieseker-stability requires $3k\ge n$ hence $k\ge \lceil \frac{n}{3}\rceil$.
\end{proof}

The short exact sequence \eqref{eq:Serre-spinor} tells us that $\cals$ might be thought of as a minimal 't Hooft bundle on $Q$. From Lemma \ref{lem:rank2-thooft+unobstructed}, applying the construction to curves supported on disjoint lines, we get the following result.
This is also studied in \cite[\S 2.3]{Faenzi}.

\begin{Lemma}\label{lem:rk2-thooft-quadric}
For all $k \ge 0$, the quadric threefold $Q$ carries rank-2 't Hooft slope-stable and unobstructed instanton bundles $E$ of charge $k$.
Moreover a general bundle of this kind satisfies
$H^1(E|_{R_d})=0$ for a general rational curve $R_d$ of degree $d=1,2$.
\end{Lemma}

\begin{proof}
Consider the union $C$ of $k+1$ disjoint lines $l_0, \ldots, l_k$ lying on $Q$. Then $C$ is a l.c.i. curve of degree $k+1$ and arithmetic genus $-k$.
Note that unobstructedness holds obviously on each line on $Q$ and this implies unobstructedness of the curve $C$ as well.
From Lemma \ref{lem:rank2-thooft+unobstructed}, a vector bundle $E$ arising from  a nowhere vanishing extension class $[e]\in \Ext^1(\cali_{C},\cccO_{Q}(-1))\simeq \D{C}^{k+1}$ is a slope-stable 't Hooft bundle which is unobstructed. Finally
let us check that a general line and a general conic are not jumping for $E$.

For a line $l$ such that $l\cap l_i=\emptyset$ for all $i\in \{0,\ldots, k\}$, we have $\cali_{C}\otimes \cccO_l\simeq \cccO_l$ hence $E|_l\simeq \cccO_l\oplus \cccO_l(-1).$
A general line $l\subset Q$ is disjoint from each line $l_i$. Indeed, in the projective space $\D{P}^3$ of lines in $Q$, those meeting a fixed line $l_i$ are parametrised by a plane $\D{P}^2_{l_i}$ hence the aforementioned lines correspond to points in the dense open subscheme $\D{P}^3\setminus \bigcup_{i=0}^k\D{P}^2_{l_i}$.

Assuming now $C$ general, we can suppose that the points $[l_i]\in \D{P}^3$, for $i=0,\ldots,k$, are in general position -- in particular no triples of points are aligned. Choose one of these points, say $[l_0]$, and a general point $[l']\in \D{P}^3$, distinct from $[l_i], \ i=1,\ldots, k.$ A smooth conic $R_2$ in $\langle l_1, l'\rangle\cap Q$ meets  $C$ in just one point $p$ so that $\cali_{C}\otimes \cccO_{R_2}\simeq \cccO_{R_2}(-p)\oplus \cccO_p$. Accordingly $E|_{R_2}\simeq \cccO_{R_2}(-p)^{\oplus 2}$. Since the vanishing of $H^1(E|_{R_2})$ is an open condition on the Hilbert scheme of conics on $Q$,
we conclude that a general conic on $Q$ is not jumping for $E$.
\end{proof}


The existence of rank-2 't Hooft bundles suggests using them as base of an inductive argument inspired by Theorem \ref{thm:exist-P3},
affording slope-stable instanton bundles of any rank $r\ge 2$ that still admit a global section after a twist by $\cccO_{Q}(1)$. The inductive step will be based on the fact that, whenever $E_{n}^k$ has charge $k$ greater than the minimum $k_0^n$, we have $\Ext^1(E_{n}^k,\cals)>0$. This allows us to construct strictly slope-semistable bundles, which then deform to slope-stable instantons.

\begin{theorem}\label{thm:exists-quadric}
    For all $n\ge 1$ and $k\ge \lceil\frac{n}{3}\rceil$ there exists a slope-stable unobstructed 't Hooft $(n,k)$-instanton bundle $E$ on $Q$ satisfying the following.
    \begin{enumerate}[label=\roman*)]
        \item \label{quadric-i}
                    We have $\ext^i(E,\cals)=0$ for $i\ne 1$.
        \item \label{quadric-ii}
        for $s\in H^0(E(1))$ general, the cokernel sheaf of  $\cccO_{Q}(-1)\xrightarrow{s}E$ is an unobstructed slope-stable torsion-free sheaf $F$ of homological dimension at most 1 such that
        \begin{align*}
            &\ext^2(F,\cals)=0, &&\Hom(\cals, F)=0, \\
            &\Ext^i(F,\cccO_{Q}(-1))=0, && \mbox{for $i\ne 1$};
        \end{align*}
        \item \label{quadric-iii} for a general rational curve $R_d\subset Q$ of degree $d=1,2$, we have $H^1(E|_{R_d})=0$.
        \end{enumerate}
        \end{theorem}
\begin{proof} The proof is similar in spirit to the one of Theorems \ref{thm:exist-P3} and \ref{thm:existence-even}.
However, there are several substantial differences, so we carry out the argument in detail.

We construct
slope-semistable instanton bundles $E_{n+1}^k$ via non-split extensions in
$\Ext^1(E_n^k,\cals)$ for $E_n^k$ a $(n,k)$-instanton satisfying the theorem. We will use a different method to show that the cokernel $F_{n+1}^k$ of a global section of $E_{n+1}^k$, obtained lifting a general $s\in H^0(E_n^k(1))$, flatly deforms to a sheaf that does not admit $\cals$ as a subsheaf. By construction $F_{n+1}^k$ will be an extension of $F_n^k$ by $\cals$, where $F_n^k$ is the cokernel of $s$. This time we will not be able to rely on Hoppe’s criterion, so we will need to slightly adjust our approach.

Let us briefly outline the main idea behind this step of the proof, as this type of argument will be widely used throughout the rest of the paper.
The locus of points in $\calm_{Q}(\ch(F_{n+1}^k))$ corresponding to sheaves admitting a quotient having Chern character $\ch(F_n^k)$ is a closed subscheme $Z\subset \calm_{Q}(\ch(F_{n+1}^k)).$
To show the existence of the aforementioned deformation of $F_{n+1}^k$, we will prove that there is a unique component of $\calm_{Q}(\ch(F_{n+1}^k))$ passing through $[F_{n+1}^k]$.
We show that this component is smooth at $[F_{n+1}^k]$ and of dimension $\ext^1(F_{n+1}^k,F_{n+1}^k)$ and that it cannot be entirely contained in $Z$. The last point will be done computing that, at least locally around $[F_{n+1}^k]$, one has $\dim(Z)< \ext^1(F_{n+1}^k,F_{n+1}^k)$.

\setcounter{step}{0}
\begin{step}[Check the base of the induction]
For $n=1$ we proved that the 't Hooft bundles constructed in Lemma \ref{lem:rk2-thooft-quadric} are unobstructed, slope-stable and have the expected splitting on a general line.
One checks that, given the union $C_m$ of $m$ disjoint lines, the ideal sheaf $\cali_{C_m}$ satisfies \ref{quadric-ii}, so that for any 't Hooft bundle $E$ constructed from $C_m$ we have  $\Ext^2(E,\cals)=0$ and, moreover, a general section $s\in H^0(E(1))$ satisfies \ref{quadric-ii} as well. All these conditions are of course open.
\end{step}

For the inductive step, suppose that the theorem holds for a given integer $n$.
Let us take an integer $k$ such that $k\ge \lceil \frac{(n+1)}{3}\rceil$ and prove the result for $(n+1,k)$-instantons of all such $k$.

\begin{step}[Construction of an instanton of higher rank]
By the inductive assumption, there exists a $(n,k)$ 't Hooft bundle $E_{n}^k$ satisfying the theorem and moreover such a bundle satisfies $\chi(E_{n}^k, \cals)=n-3k \le -1$. Therefore, there exists a non-zero element $[e]\in \Ext^1(E_{n}^k,\cals)$ which defines a non-split short exact sequence:
\begin{equation}\label{eq:ext-ss}
 0\rightarrow \cals \to E_{n+1}^k\rightarrow E_{n}^k\rightarrow 0.
\end{equation}
Here, $E_{n+1}^k$ is a strictly slope-semistable, Gieseker-unstable $(n+1,k)$-instanton.
Applying Lemmas \ref{lem:ext-ss} and \ref{lem:ext-unobstructed}, we deduce that $E_{n+1}^k$ is unobstructed and that $E_{n}^k$ is its only torsion-free quotient  having slope $-\frac{1}{2}$.
From the short exact sequence \eqref{eq:ext-ss} it is also immediate to check that, since $E_{n}^k$ satisfies \ref{quadric-iii}, the same holds for $E_{n+1}^k.$
\end{step}

Our final goal is to deform $E_{n+1}^k$ to an unobstructed slope-stable 't Hooft instanton.
We start by looking at the cokernel sheaf $F^k_{n+1}$ of a global section of $E_{n+1}^k$.

\begin{step}[Check unobstructedness of $F_{n+1}^k$]
As for the even-index case, we start by lifting a global section $s\in H^0(E_{n}^k(1))$ satisfying \ref{quadric-ii} to $s'\in H^0(E_{n+1}^k(1))$. Note that taking such a lift is possible due to the vanishing of $H^1(\cals(1))$. This leads to a commutative diagram:
\begin{equation}\label{cd-non-inst-quadric}
\begin{tikzcd}
&  &\cccO_{Q}(-1)\arrow[r, equal] \arrow[d, "s'"] &\cccO_{Q}(-1)\arrow[d, "s"] &  \\
0  \arrow[r] & \cals\arrow[r]\arrow[d, equal] & E_{n+1}^k\arrow[r]\arrow[d, two heads] & E_{n}^k\arrow [r]\arrow[d, two heads]& 0\\
0\arrow[r] &\cals\arrow[r] & F_{n+1}^{k}\arrow[r]& F_{n}^k \arrow[r]
& 0\\
\end{tikzcd}
\end{equation}

Note that the vanishing $\Ext^2(F_{n}^k, F_{n}^k)=0=\Ext^2(F_{n}^k,\cals)=0$ is ensured by the inductive hypothesis. Also, we have $\Ext^2(\cals, F_{n}^k)=0$, since $\Ext^2(\cals, E_{n}^k)=0$, see the proof of Lemma \ref{lem:min-quadric}. Then, the sheaf $F_{n+1}^k$ is unobstructed due to Lemma \ref{lem:ext-unobstructed}.
\end{step}

\begin{step}[Show that $F_{n+1}^k$ is slope-stable]
From the last row of Diagram \eqref{cd-non-inst-quadric} we readily check that $F_{n+1}^{k}$ does not admit subsheaves of degree $\le 0$. Therefore the slope of any subsheaf $L$ of $F_{n+1}^{k}$ having even rank $2m$ for $m\le n$, resp. odd rank $2m+1$ for $m\le n-1$, and degree $-d, \ d>0$, can be written in the form
$$\mu(L)= -\frac{1}{2}+ \frac{m-d}{2m}, \qquad \mbox{or} \qquad \mu(L)= -\frac{1}{2} + \frac{2(m-d)+1}{2(2m+1)}.$$ 

Let us suppose now that $L$ is a slope-semistable subsheaf of $F_{n+1}^{k}$ of rank $\le 2n$. If the composition $L\to F_{n}^k$ is 0, $L$  injects into $\cals$ hence $\mu(L)\le -\frac{1}{2}< \mu(F_{n+1}^{k})$ so this case is excluded. 
Then the induced map $L\to F_{n}^k$ is non-zero. Let us denote by $G\subset F_n^k$ the image of this map.
If $\rk(G)< 2n-1$,  by the stability assumption on $F_{n}^k$, we must have $$\mu(L)< \mu(F_{n}^k)= -\frac{1}{2}+ \frac{1}{2(2n-1)}.$$
Now, if $L$ has even rank $2m$ and degree $-d$, the last inequality leads to $$(m-d)<\frac{m}{2n-1}\le 1,$$
so that $(m-d)\le 0$ and $\mu(L)\le - \frac{1}{2}< \mu(F_{n+1}^k)$, so $L$ would not destabilize $F^k_{n+1}$.
Then $L$ has odd rank $2m+1$, so the inequality yields $$2(m-d)+1< \frac{2m+1}{2n-1}\le 1,$$ so that $\frac{2(m-d)+1}{2(2m+1)}\le 0$ and consequently $\mu(L)\le -\frac{1}{2}< \mu(F_{n+1}^k)$, which is excluded.
\medskip

To finish the proof of the slope-stability, it remains to treat the case $\rk(G)=2n-1$ and $\mu(G)=\mu(F_{n}^k)$, which implies $\dim(\supp(F_{n}^k/G))\le 1$. Denote by $K$ the kernel of the surjection $L\to G$. We have $K\ne 0$, since otherwise $F_{n+1}^k/L\simeq \cals\oplus F_{n}^k/G$, as $\ext^1(F_{n}^k/G,\cals)=h^2(\cals\otimes F_{n}^k/G(-2))=0$. But in this case we would get the existence of a surjection $F_{n+1}^k\to \cals$ which is impossible since the last row of \eqref{cd-non-inst-quadric} does not split. Therefore $\rk(K)=1$, $$\mu(F_{n+1}^k/L)=\mu(\cals/K)\ge 0> \mu(F_{n+1}^k)$$ so that
$L$ does not destabilize $F_{n+1}^k$.
\end{step}

In the next step, we use our main intuition on the construction of slope-stable deformations, essentially based on a parameter count.
\begin{step}[Deform $F_{n+1}^k$ to a slope-stable sheaf having no quotient of type $F_{n}^k$]
Let us  denote by $v$ and $v'$, the Chern characters of $F_{n+1}^k$ and $F_{n}^k$, respectively. We show that $F_{n+1}^k$ deforms, in the Gieseker moduli space $\calm_Q(v)$ of semistable sheaves with Chern character $v$, to a slope-stable sheaf that admits no torsion-free quotients with Chern character $v'$.

The locus $Z_{v'}\subset \calm_Q(v)$ parametrising sheaves admitting a torsion-free quotient of Chern character $v'$ is closed in $\calm_Q(v)$ (see \cite[\S 2.3]{HL}). Note that $[F_{n+1}^k]$ is a smooth point of this moduli space, in particular it belongs to a unique irreducible component $Y$ of $\calm_Q(v)$. Hence, to prove our claim, it is enough to show that $ \dim(Z_{v'}\cap Y)< \ext^1(F_{n+1}^k,F_{n+1}^k)$, as this will ensure that $Y\not\subset Z_{v'}$.

We observe that actually, for $[F]\in Z_{v'}$, the kernel $L$ of a surjection $F\twoheadrightarrow T$ with $\ch(T)=v'$ must be exactly $\cals$. Indeed, one first observes that, if $L$ was not slope-stable, it would admit a rank-1 torsion-free subsheaf of degree $\ge 0$, which is clearly impossible because of the semistability of $F$. Then $L$ must be slope-stable and the claim follows,
since $\cals$ is the unique slope-stable torsion-free sheaf on $Q$ with Chern character $\ch(\cals)$, as we already recalled -- see e.g. \cite{AS}.
\medskip

Suppose now, by contradiction that $Y\subset Z_{v'}$ and let $\mathbb{F}_t$ be a deformation of $F_{n+1}^{k}$ in $Y$. For $t$ general we would then have that $\mathbb{F}_t$ is a slope-stable sheaf satisfying \ref{quadric-ii} fitting in:
\begin{equation}\label{eq:subsheaves-quadric}
    0\rightarrow \cals\rightarrow  \mathbb{F}_t\rightarrow F'\rightarrow 0
\end{equation}
for $F'$ an unobstructed slope-stable sheaf in $\calm_Q(v')$.
But the family of sheaves fitting in a short exact sequence of this form has dimension at most $\ext^1(F_{n}^{k},F_{n}^{k})+ \ext^1(F_{n}^k,\cals)-1$.
So the assumption $Y\subset Z_{v'}$ implies 
$$ \ext^1(F_{n+1}^k,F_{n+1}^k)\le \ext^1(F_{n}^k,F_{n}^k)+\ext^1(F_{n}^k,\cals)-1.$$
However, a direct computation shows that:
\begin{align*}
   \ext^1(F_{n+1}^k,F_{n+1}^k)-\ext^1(F_{n}^k,F_{n}^k)-\ext^1(F_{n}^k,\cals)+1 &= \\
   -\chi(F_{n+1}^k,F_{n+1}^k)+\chi(F_{n}^k,F_{n}^k)+\chi(F_{n}^k,\cals)+1 &= 
   3k-n>0, 
\end{align*}
by the assumptions on $n$ and $k$. So this is a contradiction and our step is thus proved.
\end{step}

\begin{step}[Use the universal extension to construct a slope-stable deformation of $E_{n+1}^k$]
Applying the same arguments of the proof of Theorem \ref{thm:exist-P3} we can now construct a family of instantons specializing to $E_{n+1}^k$ using universal extensions.
\begin{substep}[Define the universal extension]
We consider a family defined by a coherent sheaf $\mathbf{F}$ on $Q\times U$, where $U\subset \calm_Q(v)$ is an open neighborhood of $[F_{n+1}^k]$, whose general element satisfies \ref{quadric-ii} and does not lie in $Z_{v'}$. Since $\ext^i(\mathbf{F}_t, \cccO_{Q}(-1))=0$ for all $i\ne 1$, we deduce that \mbox{$\cale:=\inext^1_{p_1}(\mathbf{F}, \cccO_{U\times Q}(-1))$}
commutes with base change and is locally free. Also, 
there is a universal extension on $\D{P}(\cale)\times Q$:
\begin{equation}
0\rightarrow \cccO_{\D{P}(\cale)\times Q}\otimes p_1^*\cccO_{\D{P}(\cale)}(-1)\rightarrow \mathbf{E}\rightarrow \hat{\mathbf{F}}\rightarrow 0
\end{equation}
where $\hat{\mathbf{F}}$ is the pullback of $\mathbf{F}$ to $\D{P}(\cale)\times Q$ -- we refer to the proof of Theorem \ref{thm:exist-P3} and the references therein for further details.
\end{substep}

To conclude this step our goal is to show that, for $x\in \D{P}(\cale)$ general enough, $\mathbf{E}_x$ is slope-stable.

\begin{substep}\label{substep-sing-1}[Make preliminary observations about slope-semistability and duality]
We first note that on $\D{P}(\cale)$ there exists an open dense subset $V$ of points $x$ such that $\mathbf{E}_x$ is a slope-semistable vector bundle -- recall that slope-semistability and local freeness are open conditions.
Next, observe that the map $\pi_{\D{P}(\cale)} : \D{P}(\cale) \to U$ is flat and of finite type, hence
 $\pi_{\D{P}(\cale)}(V)$ is open in $U$. Up to shrinking $U$, we can suppose that $\mathbf{E}$ is a family of slope-semistable sheaves.

 We also note that the slope-stability of $\mathbf{E}_x$ is equivalent to the slope-stability of $\mathbf{E}_x^*$. Indeed,  consider the dual family $\mathbf{E}^*:=\inhom(\mathbf{E}, \cccO_{\D{P}(\cale)\times X})$. Since for all $x\in \D{P}(\cale),\ \mathbf{E}_x$ is locally free 
 we have that $(\mathbf{E}^*)_x\simeq \mathbf{E}_x^*$.
\end{substep}

\begin{substep}\label{substep-sing-2}[Consider the Hilbert polynomials of potentially destabilizing quotients]
For any $x\in \D{P}(\cale)$, the sheaf $(\mathbf{E}^*)_x$ is strictly slope-semistable if and only if it admits a torsion-free quotient having Hilbert polynomial $R$ such that $\mu(R)=\frac{1}{2}.$
Then we consider the following set of polynomials:
\begin{align*}
A:=\{R\in \D{Q}[t]\mid \mu(R)=\frac{1}{2} \ \text{and}\ \exists\ x\in \D{P}(\cale) \ \text{s.t.}\ & \mathbf{E}^*_x \ \text{admits a torsion-free quotient} \\ & \text{with Hilbert polynomial $R$}.\}
\end{align*}
The set $A$ is finite, see \cite[Lemma 1.7.9 and Proposition 2.3.1]{HL}.

Incidentally, we warn the reader that in loc. cit. the "usual" slope $\mu(R)$ of a polynomial $R=\sum_{i=0}^3 \frac{\alpha_i(R)}{i!} t^i$ is actually replaced by the quantity $\hat{\mu}(R):=\frac{\alpha_{3}(R)}{\alpha_2(R)}$. Nevertheless since $\mu(R)$ differs from $\hat{\mu}(R)$ only by the constant factor $\alpha_3(\cccO_{Q})$ and the constant term $\alpha_2(\cccO_{Q})$, the statement still holds.

Therefore the locus $W^{ss}$ of points $x\in \D{P}(\cale)$ such that $\mathbf{E}_x$ is strictly slope-semistable is closed.
To see this, note that the locus $W^{ss}$ is the union of closed subschemes $W_R, \ R\in A$, where $W_R$ is the image of the natural morphism $\Quot(\mathbf{E}^*,R)\to \D{P}(\cale)$, the closedness of $W_R$ being ensured by the projectivity of this morphism.
\end{substep}

\begin{substep}\label{substep-sing-3}[Restrict the set of Hilbert polynomials of destabilising quotients to a singleton]
We claim that, up to shrinking $U$, we can suppose that actually $A$ consists of the single element $\{P_{\cals^*}\}$.
 Set $x_0 \in \D{P}(\cale)$ for the point corresponding to $[E_{n+1}^k]$.
 The first assertion is that, without loss of generality, we may consider only the polynomials $R$ of $A$ such that $x_0\in W_R$. Indeed, if $x_0$ does not lie in $W_R$, we may replace $\D{P}(\cale)$ with $\bigcap_{R\not\in A_{x_0}} \D{P}(\cale)\setminus W_R$.
 Indeed, the latter is an open dense neighborhood of the point $x_0$.

 The second assertion is that, by Lemma \ref{lem:ext-ss} the polynomials in $A_{x_0}$ are necessarily of the form
  $P_{\cals^*}+l$, where $l=a t + b$ a polynomial of degree at most one such that $l \gg 0$, which means $a\ge 0$.
 Indeed, we note that these are the only possible Hilbert polynomials of quotients of $E_{n+1}^*$, as their torsion-free part must be $\cals^*$ and their torsion part can have at most dimension one, due to semistability. Then, if $R$ is not of this form we would have that $x_0\not\in W_R$, against the first assertion.

The third assertion is that a rank-2 torsion-free sheaf with Hilbert polynomial $P_{\cals^*} + l$ with $l$ as described above is slope-unstable.
Indeed, there is no Gieseker-semistable (hence slope-semistable) rank-2 torsion-free sheaf $G$ on $Q$ with $c_1=1$ and $P_G \gg P_{\cals}^*$. To see this, notice that if $a_i>0$, we would have $c_2(G)<c_2(\cals^*)$ which would contradict the Bogomolov inequality.
So we should have $a=0$, which implies  $G^{**}\simeq \cals^*$ so that $b\le 0=-\mathrm{length}(\cals^*/G)$. 

Therefore, since we are assuming that each sheaf in $\D{P}(\cale)$ is slope-semistable, there is no $x\in\D{P}(\cale)$ such that $\mathbb{E}_x^*$ admits a torsion-free quotient of Hilbert polynomial $P_{\cals^*}+l \gg P_{\cals^*}$. This ends the proof of our claim $A=\{P_{\cals^*}\}$.
\end{substep}

Since we know that $W_{P_{\cals^*}}$ is a proper closed subscheme of $\D{P}(\cale)$ (as $Z_{P'}\subsetneq U$), we can conclude that a general element of $\mathbf{E}$ is slope-stable.

Finally, by semicontinuity, as $E_{n+1}^k$ has splitting $(-1^{n+1}, 0^{n+1})$ on the general line, the same will hold for $x\in \D{P}(\cale)$ general. The statement about conics has the same proof.
\end{step}
\end{proof}

 Let us now consider the moduli space $\calmi_{Q}(n,k)$ of $(n,k)$-instanton bundles. The previous Theorem shows that $\calmi_{Q}(n,k)$ is non-empty and that it contains points corresponding to unobstructed and slope-stable bundles. We deduce the next result.
\begin{corollary}
There exists a generically smooth $(1-n^2+6nk)$-dimensional component of the moduli space $\calmi_{Q}(n,k)$ of $(n,k)$-instanton bundles whose general point is a slope-stable locally free instanton having general splitting $(0^n,-1^n)$.
\end{corollary}
As moreover we constructed elements $E$ of $\calmi_{Q}(n,k)$ admitting a global section $\cccO_{Q}(-1)\to E$ whose cokernel is unobstructed and slope-stable as well, we also deduce the following result.

\begin{corollary}
For all $n\ge 2, \ k\ge \lceil \frac{n}{3}\rceil$, there exists a generically smooth component of the Gieseker moduli space $\calm_{Q}(\gamma(n,k)-\ch(\cccO_{Q}(-1)))$ of dimension $n(4-n)+k(6n-1)$ whose general point is a slope-stable sheaf.
\end{corollary}

\subsection{Instantons on Fano threefolds of index 1}

\label{section:prime}

We finally treat the case of the Fano threefolds of index one (also referred to as prime Fano threefolds). We restrict the framework to \textbf{non-hyperelliptic} prime Fano threefolds, which means that the ample generator $H_X$ of the Picard group is actually very ample.
We recall that prime Fano threefolds are classified, up to deformation equivalence, by their \textbf{genus} $g$: this is the integer defined as the genus of a codimension 2 smooth linear section of $X$, or equivalently by the equality $H_X^3=2g-2$. 
The genus $g$ of a non-hyperelliptic smooth prime Fano threefold $X$ takes values in $\{3,\ldots, 12\} \setminus \{11\}$. Also, when $g=3,$ $X$ must be a quartic hypersurface in $\D{P}^4$.

All through the current section, $X$, or $X_g$, denotes a prime Fano threefold with very ample generator of $\Pic(X)$ and $g$ denotes its genus.
As we saw in section \ref{sect:gen}, given an instanton $E$ on $X$, there is a unique pair of integers $(n,k)$ such that
\[\gamma(n,k):=\ch(E)=n \ch(F_0) - k \ch(\cccO_l(-1)),\]
where $F_0$ is a minimal instanton bundle of rank 2. In particular $\rk(E)=2n$.
When $g$ is even, the minimal instanton $F_0$ is a Mukai bundle according to the
terminology used in \cite{BMK}.
We recall a few basic properties of $F_0$. Put:
\[
m_g=\left\lceil \frac{g}{2}\right\rceil+1.
\]

In the next lemma we recollect results from \cite{Cili-prime, ciliberto} or \cite{BMK}, for the case where $g$ is even. To state it, we recall that a vector bundle $E$ on a threefold $X$ is said to be ACM if it has no intermediate cohomology, namely $H^i(E(m))=0$, for all integers $m$ and for $i=1,2$.

\begin{Lemma}\label{lem:minimal-index1}
Let $X$ be a smooth prime Fano threefold of genus $g$ and such that $-K_X$ is very ample. Then the following hold:
\begin{enumerate}[label=\roman*)]
\item\label{lem-i} there are no slope-semistable rank-2 sheaves on $X$ with $c_1=-1,\ c_2<m_g$ and $c_3=0$;
\item\label{lem-ii} $\calm_X(\gamma(1,0))$ is non-empty of dimension $\delta$, where $\delta$ is the remainder of the division of $g$ by 2. Every point in $\calm_X(\gamma(1,0))$ corresponds to a slope-stable locally free instanton bundle;
    \item\label{lem-iii} $\calm_X(\gamma(1,0))$ admits a reduced component whose general point is an unobstructed bundle except for the case where $g=4$ and $X$ is contained in a singular quadric;
    \item\label{lem-iv} if $g\ge 6$, for any pair of points $[F_0], \ [F_0']$ in $\calm_X(\gamma(1,0))$, $\ext^2(F_0,F_0')=0$.
    In particular $\calm_X(\gamma(1,0))$ is smooth;
    \item \label{lem-v} if $g\ge 6$, for all $[F_0]\in \calm_X(\gamma(1,0)),$ $F_0(1)$ is globally generated and ACM.
\end{enumerate}
\end{Lemma}

It may be useful to display the Chern character of a $(n,k)$-instanton on $X$. It has the form
\begin{equation}\label{eq:chern-1}
\gamma(n,k)=\ch(E_n^k)=2n-n H_X+(n(g-2-m_g)-k)l_X+\left(\frac{(k+m_g)}{2}+ \frac{(1-g)}{3}\right)p_X.
\end{equation}
\subsubsection{Rank-2 instanton bundles}
The first step towards the construction of instantons of high rank is the rank-2 case, since these are the base of our inductive argument.
Recall that for rank-2 vector bundles of degree $-1$, slope-stability and Gieseker-stability coincide.
Also, recall the terminology of an \textbf{ordinary threefold}. A line $l$ contained in a smooth Fano threefold of Picard rank one and index one has normal bundle $\cccO_l \oplus \cccO_l(-1)$, so $l$ is \textbf{ordinary}, or
$\cccO_l(1) \oplus \cccO_l(-2)$, in which case $l$ is obstructed. In the latter case, $l$ is a singular point of the Hilbert scheme of lines in $X$, which is of pure dimension $1$.
We say that $X$ is \textbf{ordinary} if $X$ contains an ordinary line. By \cite{gruson-skiti-nagaraj, prokhorov:exotic}, the threefold $X$ is ordinary if $g \ge 9$ and $X$ is not the Mukai-Umemura threefold.
\medskip

The existence of rank-2 instanton bundles
is known in the following cases:
\begin{theorem}\label{thm:index1-rank2}
Let $X$ be a Fano threefold with $\Pic(X)=\langle -K_X \rangle$ and $-K_X$ very ample.
\begin{itemize}
    \item There exists a slope-stable unobstructed instanton bundle on $X$ with $c_2=m$ for $m\in  \{m_g,\ldots, g+3\}$ except for $g=4$, $c_2=m_g$ and $X$ contained in a singular quadric.
    \item There exists a slope-stable ACM instanton bundle with $c_2=m$ on $X$ for \mbox{$m\in \{m_g,\ldots, g+3\}.$}
     \item If $X$ is ordinary, there exists a slope-stable unobstructed rank-2 instanton bundle of $c_2=m$, for all $m\ge m_g$.
 \end{itemize}
\end{theorem}

The existence result when $X$ carries ordinary lines, presented in \cite{Faenzi} and going back to \cite{BF}, is obtained inductively, performing elementary transformations
of rank-2 unobstructed slope-stable instanton bundles along
$\cccO_l(-1)$ for $l\subset X$ an ordinary line, see Remark \ref{rmk:et}.

The existence for charges ranging in $\{m_g,\ldots, g+3\}$ can be found in \cite{Cili-prime, ciliberto} and is based on the proof of the existence of smooth, unobstructed elliptic curves on $X$ of degree \mbox{$d\in\{m_g,\ldots ,g+3\}$}. The instanton bundles are obtained from these elliptic curves via the Serre construction.

\medskip

Combining the results of \cite{Cili-prime, ciliberto} and adopting the method used in \cite{Faenzi}, we can prove the existence of rank-2 instanton bundles of charge greater than the minimal one on the Fano threefolds of index one and genus at least three. This can be done using the technique of the elementary transformation used in \cite{Faenzi}, with the only difference being that we will be performing these along torsion instantons (see Definition \ref{def:rank0-inst}) supported on unobstructed conics, instead of lines. Note that $X$ might indeed fail to carry unobstructed lines (so $X$ is \textbf{exotic}) but it will contain unobstructed smooth conics.
The torsion instantons supported on such a curve $C$ are of the form $\cccO_C(-p)$ for a point $p\in C$.
Since we are assuming $H_X$ very ample, the torsion instantons supported on lines and smooth conics are elements of the Gieseker moduli spaces $\calm_X(0,0, d,\frac{-d}{2})$ of purely one-dimensional sheaves with Hilbert polynomial $dt$, with $d\le 2$.

\begin{Lemma}\label{lem:rank0-deg2}
Let $d \le 2$ be an integer and $Q$ be a coherent sheaf of pure dimension one having Hilbert polynomial $dt$ and such that $h^i(Q)=0$ for $i=0,1$. Then one of the following holds
\begin{enumerate}[label=\roman*)]
    \item $Q\simeq \cccO_l(-1)$ for a line $l\subset X$, or:
    \item $Q\simeq \cccO_C(-p)$ for a smooth conic $C\subset X$ and a point $p\in C$, or:
    \item \label{Lemma-Q-iii} there are two possibly coincident lines $l, \: l'$ on $X$ such that  $Q$ fits in $$0\to \cccO_l(-1)\to Q\to \cccO_{l'}(-1)\to 0.$$
\end{enumerate}
\end{Lemma}
\begin{proof}
We start by recalling that, for a curve, being Cohen-Macaulay is equivalent to having neither embedded nor isolated points. A one-dimensional sheaf $R$ is supported on a Cohen-Macaulay curve if $h^0(R(t))=0$ for some integer $t$.
Looking at $Q$, since $\chi(Q(1))=d>0$, we have $h^0(Q(1))>0$, therefore there exists a non-zero morphism
$\cccO_{X}(-1)\to Q$ whose image must be of the form $\cccO_D(-1)$ for a Cohen-Macaulay curve $D\subset \supp(Q)$. 

In particular $\deg(D)\le 2$.
If $\deg(D)=1$, $D$ must be a line $l$ hence $Q$ fits in
$$ 0\rightarrow \cccO_l(-1)\rightarrow Q\rightarrow Q'\rightarrow 0.$$

Here, $Q'$ has Hilbert polynomial $(d-1)t$. If $d=1$, then $Q'=0$ and $Q\simeq \cccO_l(-1)$. If $d=2$, since $H^i(Q')=0$ for $i=0,1$, we get that $Q'$ is a sheaf of pure dimension one and $P_{Q'}(t)=t+1.$ Applying the same argument as above, we conclude that $Q'\simeq \cccO_l'(-1)$ for a line $l'$.

In the case $\deg(D)=2$, so that, necessarily $d=2$, $Q$ fits in
$$ 0\rightarrow \cccO_D(-1)\rightarrow Q\rightarrow Z\rightarrow 0$$
for $Z$ a zero-dimensional scheme of length $1+p_a(D)$.
This implies $p_a(D)\ge -1$.
If $D$ is integral it is a smooth conic, therefore $Z$ consists of a point $p\in D$, as otherwise $Q\simeq\cccO_p\oplus \cccO_D(-1)$ contradicting the fact that $Q$ has pure dimension one. Then $Q$ is a degree-1 line bundle on $D$.
Finally, we show that if $D$ is not integral, $Q$ is always of the form \ref{Lemma-Q-iii}.

If $p_a(D)=-1$, $Q\simeq \cccO_D(-1)$ and this follows from the fact that, for a Cohen-Macaulay degree 2
curve $D$ of arithmetic genus $-1$,  $\cccO_D$  arises as a non-split extension $\Ext^1(\cccO_{l'}, \cccO_{l})$ for two possibly coincident lines $l,\ l'$.

It remains to analyse the case when $D$ is a singular plane conic and $Z\simeq \cccO_p$ for $p\in D$.
For a curve $D$ of this kind, $\cccO_D(-1)$  fits in
$$0\rightarrow \cccO_l(-2)\rightarrow \cccO_D(-1)\rightarrow \cccO_{l'}(-1)\rightarrow 0$$
for two possibly coincident lines $l, \: l'$. Then we have a commutative diagram:
\begin{equation}
\begin{tikzcd}
& \cccO_{l}(-2)\arrow[equal]{r}\arrow[d] &\cccO_{l}(-2)\arrow[d, ] & \\
0  \arrow[r] & \cccO_D(-1)\arrow[r]\arrow[d, two heads] & Q \arrow[r]\arrow[d, two heads] & \cccO_p\arrow [r]\arrow[equal]{d}& 0\\
0\arrow[r] &\cccO_{l'}(-1)\arrow[r] & Q'\arrow[r]& \cccO_p \arrow[r]
& 0
\end{tikzcd}
\end{equation}

If $Q'\simeq \cccO_{l'}(-1)\oplus \cccO_p$, the sheaf $Q$ surjects onto $\cccO_{l'}(-1)$ and the kernel of this surjection must necessarily be $\cccO_l(-1)$.
If $Q'\simeq \cccO_{l'}$, we get $\Hom(\cccO_{l'}(-1), Q)\ne 0$, indeed $\hom(\cccO_{l'}(-1),\cccO_{l}(-2))>\ext^1(\cccO_{l'}(-1),\cccO_{l}(-2))$. This last inequality is obvious if $l\neq l'$, while for $l=l'$ it holds
since, for a line $l$ supporting a non-reduced conic, we must have $\cali_l|_l\simeq \cccO_l(-1)\oplus \cccO_l(2)$ and $\Hom(\cali_l,\cccO_l(-1))\simeq \Ext^1(\cccO_l,\cccO_l(-1))\simeq \D{C}$. 
We refer also to \cite[Lemma 4.2]{BF} for this argument. 
Note that this is consistent with the fact that there exists  a unique structure of a conic supported on $l$.
A morphism $\cccO_{l'}(-1)\to Q$ must be injective and
again we find that its cokernel must be $\cccO_l(-1)$.
\end{proof}
Using this result we can prove the existence of rank-2 instanton bundles having charge strictly greater than the minimal one.

\begin{proposition}\label{prop:rank2-index1}Let $X$ be a Fano threefold of index one and such that $-K_X$ is very ample. Moreover, assume that $X$ is not contained in a singular quadric for $g=4$.
Then, for all $k\ge 0$, there exists a $(1,k)$-instanton bundle $E$ satisfying the following:
\begin{enumerate}[label=\roman*)]
    \item \label{prop-i}  $E$ is slope-stable and unobstructed;
    \item \label{prop-ii} for a sufficiently general smooth rational curve $R_d\subset X$, of degree $d\in\{1,2\}$, we have $H^1(E|_{R_d})=0$.
\end{enumerate}
If $g=4$ and $X$ is contained in a singular quadric the same statement holds for all $k\ge 1$.
\end{proposition}
\begin{proof}
We prove the result by induction on $k$.
We know from Theorem \ref{thm:index1-rank2} the existence of a $(1,k)$-instanton satisfying \ref{prop-i} for $k\in \{0,1\}$ except for the case where $g=4$ and $X$ is contained in a singular quadric, in which case the result holds for $k\in\{1,2\}$.
A $(1,k)$-instanton as just described is always a 't Hooft instanton, as $\chi(E(1))>0$ and $h^2(E(1))=0$. Thus, it fits in a short exact sequence:
$$ 0\rightarrow \cccO_{X}(-1)\rightarrow E\rightarrow \cali_C\rightarrow 0,$$
for $C$ a l.c.i. elliptic curve on $X$ of degree $m_g+k$.

From this short exact sequence one checks that a general line and a general conic are not jumping for $E$. Indeed, for any line $l\subset X$ disjoint from $C$, it is immediate to compute that $E|_l\simeq \cccO_l(-1)\oplus \cccO_l$. To verify that the same holds for a general conic $R_2$, we recall the following facts.

First, for a general point $p\in X$, the family of l.c.i. elliptic curves passing through $p$ is parametrised by a smooth
codimension 2 subscheme $\calh_p$ of the Hilbert scheme $\calh_{m_g+k,1}$ of l.c.i. elliptic curves of degree $m_g+k$.

Second, for such a point $p$ we have that
\begin{equation}\label{eq:curve-point}
h^0(\caln_{C/X}(-p))=h^0(\caln_{C/X})-2, \qquad h^1(\caln_{C/X}(-p))=0, \qquad \mbox{for all curves $C$ in $\calh_p$},
\end{equation}
so that for all $(1,k)$-instantons $E$, $h^0(E(1)\otimes \cali_p)=h^0(E(1))-2, \ h^i(E(1)\otimes \cali_p)=0$ for $i\ne 0$.

Note that this argument fails a priori for $(1,1)$-instantons on a prime Fano threefold $X$ of genus $g=3$ and for $(1,2)$-instantons on $X$ of genus 4 --  we need to look at this situation to treat the case when $X$ is contained in a singular quadric. Indeed, for a general bundle $E$ of this kind we have $h^0(E(1))=2$. Nevertheless we do have that for $p\in X$ general, $\calh_p$ is non-empty.
Indeed, we argue as follows. Note that for a bundle $E$ of such a kind and a point $p\in X$ we get a short exact sequence:
$$ 0\rightarrow \cali_p\rightarrow E(1)\otimes \cali_p\rightarrow \cali_C(1)\otimes \cali_p\rightarrow 0.$$

We deduce that for a general curve $[C]\in \calh_{3+k,1}$ (so that, in particular the curve might be assumed to be normal, according to \cite{Cili-prime, ciliberto}), if $E$ is a $(1,k)$-instanton bundle obtained from $C$ by the Serre correspondence, given a point $p$, one has $h^0(E(1)\otimes \cali_p)\ne 0$ only if $p$ belongs to the linear span $\langle C\rangle$ of $C$.
More precisely, one can identify the pencil $\D{P}(H^0(E(1)))$ with a complete linear system of elliptic curves on a hyperplane section $S_E$ of $X$ admitting an elliptic fibration, and for $S_E$ general, $H^0(E(1)\otimes \cali_p)$ corresponds to the unique elliptic curve in this linear system passing through $p$. Thus a general element in $\calh_p$ corresponds to a hyperplane section through $p$ admitting an elliptic fibration. We can then conclude that for $p$ general $\calh_p$ is a smooth codimension 2 subscheme of $\calh_{3+k,1}$ and that its general element satisfies \eqref{eq:curve-point}.

\medskip
Third, still assuming $p$ general, we also have that a general point $q\in X$ distinct from $p$ imposes 2 linear conditions on each vector space $H^0(E(1))$ that are independent from \mbox{$H^0(E(1)\otimes \cali_p)$}. Therefore the degree $m_g+k$ elliptic curves passing through both $p$ and $q$ are parametrised by a codimension 4 subscheme of $\calh_{m_g+k,1}$.

\medskip

As a consequence of these facts we have that for a general conic $R_2$ (which, without loss of generality, we can assume to be unobstructed), and for a general point $p\in R_2$, a general element $[C]\in\calh_p$ satisfies $C\cap R_2=\{p\}$ so that $\cali_C\otimes \cccO_{R_2}\simeq \cccO_{R_2}(-p)\oplus \cccO_p$. For an instanton $E$, obtained through the Serre correspondence from such a curve $C$, $E|_{R_2}$ is an extension of two degree $-1$ line bundles on $R_2$, so that $E|_{R_2}\simeq \cccO_{R_2}(-p)^2$ or, equivalently,  $h^1(E|_{R_2})=0$.
To show that the proposition holds for higher values of $k$ it will be sufficient to prove the following claim:
\begin{claim}
Let $E, \: R_2, \: \phi$ be, respectively, a $(1,k)$-instanton bundle satisfying \ref{prop-i}, \ref{prop-ii},
an unobstructed smooth conic $R_2\subset X$ that is not jumping for $E$ and an epimorphism \mbox{$\phi: E\twoheadrightarrow \cccO_{R_2}(-p), \ p\in R_2$.}
Then $F:=\ker(\phi)$ deforms to a slope-stable and unobstructed $(1,k+2)$ instanton bundle which satisfies \ref{prop-i}, \ref{prop-ii}.
\end{claim}

Let us prove the claim (we follow essentially \cite{Faenzi}). By construction, we have
\begin{equation}\label{eq:et-conic}
0\rightarrow F\rightarrow E\rightarrow \cccO_{R_2}(-p)\rightarrow 0.
\end{equation}

One checks that $F$ is slope-stable, since $E=F^{**}$ and $E$ is slope-stable by assumption. Also, $H^i(F)=0$ for $i=1,2$. Note that from \eqref{eq:et-conic} we have that $F$ satisfies \ref{prop-ii}. Indeed, there is a non-empty open subset of the Hilbert scheme of lines, resp. conics, in $X$ such that a curve $R'$ of such subset is disjoint from $R_2$. For the curve $R'$, we have $E|_{R'}\simeq F|_{R'}$.

Let us check unobstructedness of $F$. Applying $\Hom(\:\cdot\:, F)$ and $\Hom(E, \: \cdot\:)$ to $\eqref{eq:et-conic}$, by the assumptions on $E$, we have \[\Ext^2(E,E)=0, \qquad \Ext^2(E,\cccO_{R_2}(-p))\simeq H^2(E^*|_{R_2}(-p))=0,\]
so that $\Ext^2(E,F)=0$. Restricting \eqref{eq:et-conic} to $R_2$ we check the vanishing of $\Ext^3(\cccO_{R_2}(-p),F)\simeq \Hom(F,\cccO_{R_2}(-3p))^*=0$. This  shows $\Ext^2(F,F)=0$.

Consider now the vector $\gamma(1,k+2)$ in $\oplus_{0 \le i \le 3}H^{i,i}(X)$, see \eqref{eq:chern-1}. We have that $F$ corresponds to a smooth point in the Gieseker-Maruyama moduli space $\calm_X(\gamma(1,k+2))$.
Since $F$ is unobstructed, it can indeed be deformed as a smooth point of this moduli space. So there exists a family $\mathbf{F}$, i.e. a sheaf on  $X\times T$, flat over a smooth curve $T$ with a marked point $0\in T$, whose central fibre $\mathbf{F}_0$ is isomorphic to $F$ and whose general element $\mathbf{F}_t$ satisfies the cohomological vanishing of \eqref{def:instanton} and \ref{prop-i}, \ref{prop-ii}.
Indeed, all these conditions are open.

We show that for $t$ sufficiently general, $\mathbf{F}_t$ is locally free, hence it is an instanton bundle.
If  $\mathbf{F}_t$ is not locally free, according to \cite[Theorem 24]{CM}
 it would fit in a short exact sequence:
\begin{equation}\label{eq:et-transformation}
0\rightarrow \mathbf{F}_t\rightarrow \mathbf{F}_t^{**}\rightarrow Q\rightarrow 0,
\end{equation}
where the double dual $\mathbf{F}_t^{**}$ is a slope-stable $(1,k+2-d)$ instanton bundle and $Q$ is a torsion instanton of degree $d$.

By semicontinuity of cohomology, we find that, for $n\ll 0$, $h^2(\mathbf{F}_t(n))=h^1(Q(n))=-\chi(Q(n))=-dn\le -2n$, hence $d\le 2$.
From Lemma \ref{lem:rank0-deg2}, we know that $Q$ can only belong to 3 families of 1-dimensional sheaves. We show that all the corresponding families of non-locally free instantons are divisorial in $\calmi_X(1,k+2)$. This suffices to conclude that for $t\in T$ general, $\mathbf{F}_t$ is an instanton bundle.

If $d=1$, we have $Q\simeq \cccO_l(-1)$ for a line $l\subset X$.
Sheaves of this kind belong to a family of dimension one, therefore the general non-locally free deformations of $F$ that are singular along a line are parametrised by a family of dimension at most $\ext^1(\mathbf{F}_t^{**},\mathbf{F}_t^{**})+1$ for $\mathbf{F}^{**}_t$ a slope-stable unobstructed $(1,k+1)$ instanton bundle. But by Riemann-Roch we compute that $\ext^1(\mathbf{F}_t,\mathbf{F}_t)=\ext^1(\mathbf{F}_t^{**},\mathbf{F}_t^{**})+2$.

If $d=2$, $c_2(\mathbf{F}^{**}_t)=m$ and $Q$ is either a degree -1 line bundle on a smooth conic or it arises from an extension class in $\Ext^1(\cccO_{l'}(-1),\cccO_{l}(-1))$. Now, as $c_2(\mathbf{F}_t)=c_2(\mathbf{F}_t^{**})+2$, by Riemann-Roch we compute that $\ext^1(\mathbf{F}_t,\mathbf{F}_t)=\ext^1(\mathbf{F}_t^{**},\mathbf{F}_t^{**})+4$. But the families of instantons that are singular along a degree 2 curve have dimension $\ext^1(\mathbf{F}_t^{**},\mathbf{F}_t^{**})+3$. Indeed, $Q$ varies in a 2-dimensional family, and we have a one-dimensional family of epimorphisms $\mathbf{F}_t^{**}\twoheadrightarrow Q$. We can thus conclude that a general deformation of $F$ is an instanton bundle that, by semicontinuity, satisfies \ref{prop-i} and \ref{prop-ii}. 
\end{proof}

\begin{remark}\label{rmk:very-ample-1}
The hypothesis that $-K_X$ is very ample is necessary for the argument used in Proposition \ref{prop:rank2-index1}.
Indeed, this assumption ensures that a curve of degree $2$ in $X$ is rational, which implies that the dimension of the family of torsion instantons on $X$ supported on degree 2 curves is 2.
This will no longer be true if $-K_X$ is not very ample. For example, if $X$ is a degree 2 covering of a Fano threefold $Y$, we have irrational curves of degree 2 obtained pulling back lines on $Y$. The torsion instantons supported on a curve of this kind are not all isomorphic, actually they vary in a family of dimension one.
\end{remark}

\subsubsection{Instanton bundles of higher rank}
Let us now pass to the higher-rank range.
The approach will be similar to what we did earlier, namely, we produce strictly slope-semistable instanton bundles that can then be deformed into slope-stable instantons.

However, this time we meet some additional difficulty. In the first place, with our approach, we can no longer hope to prove the existence of instanton bundles $E$ such that $h^0(E(1))\ne 0$: in the rank-2 cases this indeed only holds for a values of $c_2$ varying in $\{m_g,\ldots, g+3 \}$.
The second difficulty, actually the main one, is the identification of the minimal charge: we are indeed able to determine its exact value only when $g\ge 4$.
This computation is a consequence of the following lemma.

\begin{Lemma} \label{lem:ext2-index1} \label{lem:min-index-1}
Let $E$ be a $(n,k)$-instanton on $X$. Then the following hold.
\begin{enumerate}[label=\alph*)]
    \item \label{ext vanishing}  If $g\ge 6$ then:
    \begin{enumerate}[label=\roman*)]
        \item \label{ext-van-1} $\ext^2(F_0,E)=0;$
        \item \label{ext-van-2} $\ext^2(E,F_0)=0.$
    \end{enumerate}
Moreover, items \ref{ext-van-1} and \ref{ext-van-2} hold for $g\in \{4,5\}$ as well if $E$ is slope-stable and Gieseker-semistable, respectively, and not a minimal instanton.
    \item \label{semi-stable-1} If $g\ge 6$, then $k\ge 0$. The same holds for $g\in \{4,5\}$ and $E$ Gieseker-semistable.
    \item \label{minimal-charge} If $g\ge 6$ and $E$ is Gieseker-stable and not a minimal instanton, then $k\ge 1$ for $g$ odd and $k\ge n$ for $g$ even. The same holds for $g\in\{4,5\}$ if $E$ is slope-stable.
\end{enumerate}
\end{Lemma}

\begin{proof}[Proof of \ref{ext vanishing}]
Let us consider a minimal rank-2 instanton bundle $F_0$.
Set $\delta$ for the remainder of the division by $2$ of $g$.
It is well-known (see for instance \cite{BF}) that the evaluation map $H^0(F_0^*)\otimes \cccO_{X}\to F_0^*$ is surjective, in other words $F_0^*=F_0(1)$ is globally generated. Also, the evaluation map has a locally free kernel $K_0$, so that, in view of the equality $m_g+1-\delta= h^0(F_0(1))=\chi(F_0(1))$, we have a short exact sequence of vector bundles:
\begin{equation}\label{eq:gg}
    0\rightarrow K_0\rightarrow \cccO_{X} ^{m_g+1-\delta}\rightarrow F_0^*\rightarrow 0
\end{equation}

The vector bundle $K_0$ is slope-stable. We reproduce here a well-known argument to show it, for the reader's convenience.
From \eqref{eq:gg} we can check that the bundle $K_0$ satisfies $H ^0(\wedge^i K_0)=0$ for $i=1,\ldots, \rk(K_0)$. Indeed, from its very definition, first of all we have $h^0(K_0)=0$. Then, all other vanishing are deduced, recursively, from the inclusion
$$\wedge^i K_0\hookrightarrow \cccO_{X}^{m_g + 1- \delta} \otimes \wedge^{i-1} K_0,$$ valid for all integers $i > 0$.
Such inclusion, in turn, comes from tensoring \eqref{eq:gg} with $\wedge^{i-1} K_0$ and using the inclusion $\wedge^i K_0 \hookrightarrow K_0 \otimes \wedge^{i-1} K_0$.
Therefore $K_0$ is slope-stable by Hoppe's criterion, see \cite[Theorem 3]{jardim-menet-prata-sa-earp}. Note that the assumption $g \ge 6$ guarantees $m_g+1-\delta \ge 5$ so $\rk(K_0) \ge 3$, hence
\begin{equation} \label{slope K}
\mu(K_0^*) = \frac{1}{\rk(K_0)} \le \frac 13.
\end{equation}

Let us now take an instanton sheaf $E$ and apply $\Hom(\cdot, E)$ to the dual of \eqref{eq:gg}.
Doing so we get an isomorphism $\Ext^2(F_0,E)\simeq\Ext^3(K_0^*,E)\simeq \Hom(E(1),K_0^*)^*$. Now, since $E(1)$ is slope-semistable of slope $1/2$, by \eqref{slope K} we get  $\Hom(E(1),K_0^*)=0$ and \ref{ext vanishing} is proved.
Applying $\Hom(E,\cdot)$ to \eqref{eq:gg} tensored by $\cccO_{X}(-1)$ we get that, as $\Ext^2(E,\cccO_{X}(-1))\simeq H^1(E)^*=0$, $\Ext^2(F_0,E)$ injects into $\Ext^3(E,K_0(-1))\simeq \Hom(K_0,E)^*.$
As $c_1(K_0)=-1$ and $\rk(K_0)\ge 3$ for $g\ge 6$, the slope-semistability of $E$ ensures that $\hom(K_0,E)=0$.

Finally, let us prove item \ref{ext vanishing} for $g \in \{4,5\}$, in which case $m_g+1-\delta=4$ and $K_0\simeq K_0^*(-1)$ is also a minimal instanton. If $E$ is slope-stable and not a minimal instanton, $\Hom(E(1),K_0^*)=0$ which again allows us to conclude that $\Ext^2(F_0,E)=0$.
If $\Ext^2(E,F_0)\ne 0,$ we would have $\Hom(K_0,E)\ne 0.$
Now, such a morphism must necessarily be injective which cannot happen if $k\ge 1$ and $E$ is Gieseker-semistable.
\end{proof}

\begin{remark}
We cannot hope to get such a strong statement for the case $g=3$, namely for $X\subset \D{P}^4$ a smooth quartic hypersurface.
For now, note that in this case the cokernel $F_0$ of the coevaluation map $F_0\hookrightarrow \cccO_{X}^3$ is $\cali_l(1)$, where $l\subset X$ is a line.
Therefore we see that, already in rank 2, taking a rank-2 instanton $E$ corresponding to an elliptic curve $C'$ containing $l$ to get $\hom(E, \cali_l)=\ext^3(\cali_l(1),E)=\ext^2(F_0,E)\ne 0.$
\end{remark}

\begin{proof}[End of the proof of Lemma \ref{lem:ext2-index1}]
To prove points \ref{semi-stable-1} and \ref{minimal-charge} we start by computing $\chi(E,F_0)=\chi(F_0,E)$. Setting again $\delta$ to be the remainder of the division of $g$ by $2$, this is done applying Riemann-Roch, which leads to:
\[\chi(E,F_0)=\chi(F_0,E)=(1-\delta) n -k. \]

Let us prove item \ref{semi-stable-1}. We argue by induction on $n$. For $n=1$ the statement holds due to Lemma \ref{lem:minimal-index1}. Let us suppose now that $n>1$. If $g\ge 6$, $\ext^2(F_0,E)=0$, so that $k<0$ implies $\hom(F_0,E)>0$. A non-zero morphism $F_0 \to F$ must necessarily be injective so that its cokernel should be a $(n-1,k)$ instanton, which is not possible due to the inductive hypothesis. For the cases $g\in \{4,5\}$ the same argument applies if  $\ext^2(F_0,E)=0$, which in particular is true for $E$ slope-stable.
Then, we assume $\ext^2(F_0,E)\ne 0$ so that, by \eqref{eq:gg} $\ext^3(K_0^*,E)=\hom(E,K_0)\ne 0.$
But assuming $E$ Gieseker-semistable, a morphism $E\to K_0$ can exist only for $k\ge 0$.
\bigskip

Finally, let us prove \ref{minimal-charge}. Let $E$ be a $(n,k)$-instanton satisfying the statement. In particular, by item \ref{ext vanishing}, $\ext^2(F_0,E)=0$. The stability assumption on $E$ and the fact that $E$ is not a minimal instanton, ensure that $\Hom(F_0,E)=0$ so that $\chi(F_0,E)<0$.
This proves that, when $g$ is even, we must have $k \ge n$. So \ref{minimal-charge} is proved in this case. When $g$ is odd, this gives $k \ge 0$.

To prove $k \ge 1$ we need a little extra work.
We consider $\gamma=\gamma(1,0)=\ch(F_0)$ and we look at the moduli space $M:=\calm_X(\gamma)$.
Assume $g$ is odd, hence $\chi(E,F_0)=-k$. It is proved in \cite{BF} that, when $g$ is odd, $M$ is a connected projective curve, possibly singular and/or reducible, and that $M$ is a fine moduli space. Therefore, there exists a universal sheaf $\cale $ over $X \times M$ such that any closed point $y$ of $M$ corresponds to a sheaf $\cale_y$ which is a minimal instanton $F_0$.
Denoting by $p$ and $q$ the projections from $X \times M$ onto $X$ and $M$, by Lemma \ref{lem:ext2-index1}, we have:
\begin{equation} \label{due che si annullano}
R^k q_*(p^*(E) \otimes \cale^*)=0, \qquad \mbox{for $k \ne 0,1$.}
\end{equation}
Also, $q_*(p^*(E)\otimes \cale^*) =0$, as $H^0(E \otimes \cale^*_y)=0$ if $y$ is a sufficiently general point of each component of $M$. On the other hand, in view of \eqref{due che si annullano}, we may use base change of cohomology, so that the sheaf $R^1 q_*(p^*(E)\otimes \cale^*)$ is supported at a point $y \in M$ if and only if
\[
\ext^1_X(\cale_y,E) = \hom_X(\cale_y,E) \ne 0,
\]
where the equality of the first two quantities follows from Riemann-Roch and Lemma \ref{lem:ext2-index1}. Note that, if $E \simeq \cale_z$ for some point $z \in M$, then
\[
\ext^1_X(\cale_y,\cale_z) = \hom_X(\cale_y,\cale_z) = 1 \Leftrightarrow z=y,
\]
otherwise, if $z \ne y$, we have $\ext^1_X(\cale_y,\cale_z) = \hom_X(\cale_y,\cale_z) = 0$. This implies that $R^1 q_*(p^*(\cale_z) \otimes \cale^*) \simeq  \cccO_z$.
Since $\ch(E)=n\ch(F_0)=n\ch(\cale_z)$, we deduce,
via Grothendieck-Riemann-Roch, that
\[
\ch(R^1 q_*(p^*(E) \otimes \cale^*)) = n \ch(\cccO_z).
\]
Therefore, the support of $R^1 q_*(p^*(E) \otimes \cale^*)$ is not empty, hence there is $y \in M$ such that $\hom_X(\cale_y,E) \ne 0$. This is impossible if $E$ is slope-stable and not isomorphic to $\cale_y$.
This finishes the proof of \ref{minimal-charge}.
\end{proof}

For all $n\ge 2$, we finally prove the existence of slope-stable $(n,k)$-instanton bundles for any values of $k\ge n$ for $g$ even (hence for any values of $k$ greater than or equal to the minimal one), provided that for $g=4$, $X$ is not contained in a singular quadric, and for any $k\ge 1$ for $g$ odd and greater than 3.
Once again, we construct, inductively, strictly slope-semistable bundles that can be deformed into slope-stable ones. The primary difference with respect to the cases where the index is greater than 1 lies in the moduli space where this deformation occurs. Indeed, since for $i_X=1$, rank-2 't Hooft bundles do not necessarily exist for arbitrary values of the charge, our inductive method cannot produce sheaves that admit global sections only after a twist by $\cccO_{X}(1)$. For this reason the deformation will take place in a moduli space of simple sheaves rather than in a moduli space of extensions.
\begin{theorem}\label{thm:exist-index-1}
Let $X$ be a Fano threefold with $\Pic(X)=\langle- K_X \rangle$ and genus $g\ge 4$. Assume that moreover $X$ is not contained in a singular quadric whenever $g=4$. Then, for all $n\ge 1$ and  all $k\ge n$ if $g$ even, resp. for all $k\ge 1$ if $g$ odd, there exists a $(n,k)$-instanton bundle $E_n^k$ such that the following hold:
\begin{enumerate}[label=\roman*)]
    \item \label{theorem-i} $E_n^k$ is slope-stable and unobstructed;
    \item \label{theorem-ii} for a general rational curve $R_d\subset X$ of degree $d\in\{1,2\}$, we have $H^1(E_n^k|_{R_d})=0$;
    \item \label{theorem-iii} $\Ext^2(F_0,E_n^k)=\Ext^2(E_n^k, F_0)=0$
    for all minimal instantons $F_0$.

\end{enumerate}
\end{theorem}

The main idea behind the proof is the same as the one used for indices $\{2,3,4\}$, namely, to construct inductively some $(n+1,k)$-instantons, which would be unobstructed and strictly slope-semistable, via non-split extensions in $\Ext^1(E_n^k, F_0)$, where $E_n^k$ is an $(n,k)$-instanton coming from the previous step of the induction, and then to show that they deform into slope-stable instantons.

The key difference lies in the moduli space where this deformation takes place and the main reason for this difference is that on $X$ we do not have existence of rank-2 't Hooft bundles for every value of the charge. This time, we will deform $E_{n+1}^k$ in a moduli space of simple sheaves $\Spl_X(\gamma(n+1,k))$.
To achieve this, we will once again use a dimension counting argument to show that the only component of $\Spl_X(\gamma(n+1,k))$ through $[E_{n+1}^k]$ cannot be entirely contained in the closed subscheme $Z_{\gamma(n,k)} \subset \Spl_X(\gamma(n+1,k))$, which parametrises sheaves admitting quotients with Chern character $\gamma(n,k).$

\setcounter{step}{0}
\begin{proof}
The argument is by induction, namely, for given $n$, we take care of the minimal value of $k$ and work by induction on $k$; this allows us to settle the induction step and move to higher $n$.
\begin{step}[Check the base of the induction]
For $n=1$, \ref{theorem-i} and \ref{theorem-ii} hold due to Proposition \ref{prop:rank2-index1} whilst \ref{theorem-iii} is true by Lemma \ref{lem:min-index-1}.
\end{step}

\begin{step}[Construction of an instanton of higher rank]
Suppose now that the theorem holds for $n\ge 1$ and let $k$ be an integer $\ge n+1$ for $g$ even, resp. $k\ge 1$ for $g$ odd.
By the inductive hypothesis there exists a slope-stable
$(n,k)$-instanton $E_n^k$ satisfying the Theorem. Since $\chi(E_n^k,F_0)<0$ and $\Ext^3(E_n^k,F_0)=0$ by semi-stability, we have $\ext^1(E_n^k,F_0)\ne 0$ hence there exists a non-split extension:
\begin{equation}\label{eq:ext-index-1}
 0\rightarrow F_0\rightarrow E_{n+1}^k\rightarrow E_n^k\rightarrow 0.
\end{equation}

The sheaf $E_{n+1}^k$ is a $(n+1,k)$-instanton bundle and from Lemma \ref{lem:ext-ss} we know that $E_{n+1}^k$ is a strictly slope-semistable sheaf and that moreover $E_{n}^k$ is its only torsion-free quotient having slope $-\frac{1}{2}.$
The sheaf $E_{n+1}^k$ is unobstructed, as it follows from Lemma \ref{lem:ext-unobstructed} and it satisfies \ref{theorem-ii} since this holds for both $F_0$ and $E_n^k$. 

Note that the sheaf $E_{n+1}^k$ is not Gieseker-stable, as it is destabilized by $F_0$. Still, we may consider it as a point in the moduli space $\Spl_X(\gamma(n+1,k))$ of simple sheaves having Chern character $\gamma(n+1,k)$. For an account on this moduli space
we refer to \cite{AK,KO}. 

The fact that $E_{n+1}^k$ is simple can easily be proved by the short exact sequence \eqref{eq:ext-index-1}. Indeed, the slope-stability of $F_0$ and $E_n^k$ implies $\Hom(E_n^k,F_0)=0=\Hom(F_0,E_n^k)$. Accordingly, we get $\Hom(E_{n+1}^k,F_0)=0$, as otherwise the short exact sequence would split, and \[\Hom(E_{n+1}^k,E_{n+1}^k)\simeq \Hom(E_{n+1}^k,E_n^k)\simeq \Hom(E_{n}^k,E_n^k)\simeq \D{C}.\]
\end{step}
\begin{step}[Deform $E_{n+1}^k$ to a slope-stable instanton in the moduli space of simple sheaves]
The unobstructedness of $E_{n+1}^k$ ensures that the corresponding point $[E_{n+1}^k]$ belongs to a unique component $Y$ of $\Spl_X(\gamma(n+1,k))$ which has dimension $\ext^1(E_{n+1}^k, E_{n+1}^k)$.
Applying the same argument presented in the proof of Theorem \ref{thm:exists-quadric}, we have that the strictly slope-semistable simple sheaves are parametrised by a closed subscheme $Z\subset \Spl_X(\gamma(n+1,k))$. Thus, in order to prove that $[E_{n+1}^k]$ deforms to a slope-stable sheaf, it will be enough to prove that $Y\cap Z$ is a strict closed subscheme of $Y$.

Next, notice the following fact. A torsion-free sheaf with Chern character $\gamma(n+1,k)$ is strictly slope-semistable if and only if it admits a torsion-free quotient with slope $-\frac{1}{2}$. We deduce that $Z$ can be written as a finite union $Z=\cup_{v\in A} Z_v$ where each $Z_v$ is the closed subscheme parametrising sheaves admitting a torsion-free quotient with Chern character $v=(v_0,\ldots,v_3)$ with $\frac{v_0}{v_1}=-\frac{1}{2}$.
Here, the finiteness of $A$ follows from \cite[Lemma 1.7.9]{HL}.

Without loss of generality, we can assume that the set $A$ is actually the singleton $\{\gamma(n,k)\}$. Since on $X$ there are no slope-semistable rank-2 torsion-free sheaves with $c_1=-1$ and Hilbert polynomial $P>P_{F_0}$, the same arguments presented in the proof of Theorem \ref{thm:exists-quadric}, Steps \ref{substep-sing-1}-\ref{substep-sing-3}, allow indeed to conclude what follows: a general strictly slope-semistable deformation $E'$ of $E_{n+1}^k$ is a locally free sheaf such that the only torsion-free quotients of slope $\frac{1}{2}$ of the dual $E'^*$ have Hilbert polynomial $P_{F_0^*}$. 
But as every slope-semistable rank-2 torsion-free sheaf with Hilbert polynomial $P_{F_0^*}$ is the dual of a minimal instanton, this is equivalent to state that all torsion-free quotients of strictly slope-semistable deformations of $E_{n+1}^k$ are $(n,k)$-instantons, so that $A=\{\gamma(n,k)\}$.  

Therefore, we have that $Y\subset Z$ if and only if $Y\subset Z_{\gamma(n,k)}$. Suppose that this was the case.
Since all slope-semistable torsion-free sheaves with
Chern character $\gamma(1,0)$ are actually slope-stable instanton bundles (in view of \cite[Proposition 3.4]{BF}), we deduce that all closed points of $Z_{\gamma(n,k)}$ correspond to sheaves arising from extension classes in $\Ext^1(F,F_0)$ for $F$ a $(n,k)$-instanton.
Let us consider now an open neighborhood $U\subset Y$ of $[E_{n+1}^k]$.
Up to shrinking $U$, we can suppose that the latter is disjoint from $Z_v$ for every $v\in A, \ v\ne \gamma(n,k).$ This implies that given a sheaf $E$ corresponding to a point in $U,$ all torsion-free quotients of $E$ having Chern character $\gamma(n,k)$ must be slope-stable.
Therefore all points $[E]$ in $U$ admit a unique (up to isomorphism) torsion-free slope-stable quotient $F$ with Chern character $\gamma(n,k)$. This will in particular ensure that $F$ satisfies $\ext^2(F_0,F)=\ext^2(F,F_0)=0$ (this vanishing is obtained applying Lemma \ref{lem:min-index-1}).
Summing up we would obtain:
\begin{equation}\label{eq:dim-chiuso}
\dim(Z_{\gamma(n,k)})=\dim(\calm_X(\gamma(1,0)))+ \ext^1(F,F)+\ext^1(F,F_0)-1,
\end{equation}
which from a direct computation would lead to \[\ext^1(E_{n+1}^k, E_{n+1}^k)-\dim(Z_{\gamma(n,k)})= k-n\chi(F_0,F_0)=-\chi(E_n^k,F_0),\]
which is strictly greater than zero by the assumptions on $n$ and $k$. This is a contradiction since we were assuming $U\subset Z_{\gamma(n,k)}.$
A general deformation of $[E_{n+1}^k]$ in $\Spl_X(\gamma(n+1,k))$ is therefore a slope-stable instanton bundle. Since moreover $E_{n+1}^k$ is unobstructed and satisfies \ref{theorem-ii}, as it can be checked immediately from the defining extension, and \ref{theorem-iii}, the same will hold, by semicontinuity, for a general deformation.
\end{step}
\end{proof}

This ends the proof of the main theorem from the introduction.
Also this time, as a direct consequence of Theorem \ref{thm:exist-index-1}, we get the following.
\begin{corollary}
If $g$ is even, resp. $g$ is odd, for all $n\ge 2$, $k\ge n$, resp. $k\ge 1$, there exists a generically smooth $1-n^2+2nk$, resp. $(1+2nk)$-dimensional component of the moduli space $\calmi_X(n,k)$ of $(n,k)$-instanton bundles whose general point is a slope-stable locally free instanton having general splitting $(0^n,-1^n)$.
\end{corollary}

The existence of ACM instantons on prime Fano threefolds is proved in \cite{Cili-prime} and \cite{BF} (in the ordinary case), as we recollected in Theorem \ref{thm:index1-rank2}. We deduce the following:

\begin{corollary}\label{cor:ACM}
Let $X$ be a prime Fano threefold of genus $g\ge 4$ and assume that $X$ is not contained in a singular quadric for $g=4$.
\begin{itemize}
\item if $g$ is odd, there exists a slope-stable ACM $(n,k)$-instanton bundle for all $ n\ge 2$ and all $k\in \{1,\ldots, g+3-m_g\}$;
\item if $g$ is even, there exists a slope-stable ACM $(n,k)$-instanton bundle for $n\in \{2,\dots,g+3-m_g\}$ and $k\in \{n,\ldots ,g+3-m_g\}$.
\end{itemize}
\end{corollary}
\begin{proof}
The result is obtained applying the construction of the proof of Theorem \ref{thm:exist-index-1}, just noticing that an extension of ACM bundles is ACM and that being ACM is an open condition.
\end{proof}

\section{Stability of restriction to K3 surfaces}

\label{sec:restriction}

Let $X$ be a Fano threefold of index $i_X$ and Picard number one and let $S$ be an element in the linear system $|\cccO_X(-K_X)|$. Then, if $S$ is general enough, it is a smooth $K3$ surface of Picard rank one. Let now $E$ be a $(n,k)$-instanton bundle. We investigate the stability of the restriction $E|_S$. We set $k_0^n$ to be the integer defined in the Main Theorem \ref{thm:main}.
We move on to show Theorem \ref{mainthm:anticanonical-stable}.

\begin{theorem} \label{thm:anticanonical-stable}
Let $X$ be a Fano threefold of index $i_X$ with $g\ge 4$ and $X$ not contained in a singular quadric when $g=4$ for $i_X=1$. Let $S$ be a general element in the linear system $|\cccO_X(-K_X)|$. Then for all integers $n,k$ with $n\ge 2-r_X, \ k\ge k_0^n$, there exists an open non-empty subscheme $\calr_S(n,k)\subset \calmi_X(n,k)$ parametrising slope-stable $(n,k)$-instanton $E$ on $X$ such that $E|_S$ is slope-stable.
\end{theorem}
\begin{proof}
We prove the result by induction on $n$, along several steps.

\setcounter{step}{0}

\begin{step}[The base of the induction] \label{K3-step1}
Let $S\subset X$ be a smooth anticanonical divisor having Picard rank one -- recall that these conditions are satisfied for $S$ general in $|\cccO_X(-K_X)|$. Denote by $H_S$ the restriction of $H_X\in |\cccO_X(1)|$ to $S$, so that $\Pic(S)\simeq \D{Z}H_S$.
For any $n=2-r_X$ and $k\ge k_0^n$, we consider a slope-stable $(n,k)$-instanton $E$.

Without loss of generality, we may assume that moreover $E$ satisfies Theorem
\ref{thm:exist-P3}, \ref{thm:exists-quadric}, \ref{thm:existence-even},  \ref{thm:exist-index-1}, for $i_X=4,\:3,\:2,\:1$, respectively.
Consider the short exact sequence:
\begin{equation}\label{eq:restriction-anticanonical}
0\rightarrow E(-i_X)\rightarrow E\rightarrow E|_S\rightarrow 0.
\end{equation}

As $\Pic(S)\simeq \D{Z}$ and $E|_S$ is a rank-2 vector bundle with $c_1(E|_S)=-r_X H_S$, we see that $E|_S$ is slope-stable. Indeed from \eqref{eq:restriction-anticanonical}, since $H^0(E)$ and $H^1(E(-i_X))$ both vanish, $H^0(E|_S)=0$, which is enough to prove slope-stability since $\Pic(S)\simeq \D{Z}H_S$.
\end{step}

So Step \ref{K3-step1} is proved and we can pass to the induction step.

\begin{step}[Constructing and restricting an instanton of rank $n+1$]

For $n\ge 2-r_X, k\ge k_0^{n+1}$ we take a slope-stable $(n,k)$-instanton bundle $E_n^k$ satisfying the statement of the Theorem. Let $F_0$ be a minimal instanton on $X$. Incidentally, recall that this implies $F_0\simeq \cccO_{X}$ when $i_X$ is even, while, when $i_X$ is odd, $F_0$ is a minimal rank-2 instanton bundle.

Note that $\ext^1(E_n^k,F_0(-i_X))=\ext^2(F_0,E_n^k)=0$. This is clear for even index since in that case $F_0\simeq \cccO_X$. For odd index this is shown in Lemmas \ref{lem:min-quadric} and \ref{lem:ext2-index1}.

Then, by the short exact sequence
$$ 0\rightarrow F_0(-i_X)\rightarrow F_0\rightarrow F_0|_S\rightarrow 0, $$
we see that the following linear map is injective
$$\Ext^1(E_n^k,F_0)\to \Ext^1(E_n^k,F_0|_S)\simeq H^1({E_n^k}^*\otimes_{\cccO_S} F_0|_S)\simeq \Ext^1(E_n^k|_S, F_0|_S).$$

So, given a non-zero element $e$ in $\Ext^1(E_n^k,F_0)$ and denoting by $E$ the $(n+1,k)$-instanton defined by $e$, we have that $E|_S$ fits in a non-split short exact sequence, defined by $e|_S$:
$$ 0\rightarrow F_0|_S\rightarrow E|_S\rightarrow E_n^k|_S\rightarrow 0.$$

Since by assumption $E_n^k|_S$ is slope-stable and $\Pic(S)\simeq \D{Z}H_S$, Lemma \ref{lem:ext-ss} implies that $E|_S$ is strictly slope-semistable and $E_n^k|_S$ is its only torsion-free quotient having slope $\frac{-r_X}{2}$.
\end{step}

We have thus constructed an $(n+1,k)$-instanton and checked its restriction to $S$. The next goal is to show that, for a general deformation $E'$ of $E$, $E'|_S$ can be "slope-destabilized"
at most by $F_0|_S$. The proof will essentially consist in a dimension argument, which will involve explaining this last sentence rigorously.

\begin{step}[The case $i_X \ne 1$] \label{K3-step2}
Let us suppose, at first, that $i_X>1$. From the proofs of Theorems \ref{thm:exist-P3}, \ref{thm:exists-quadric}, \ref{thm:existence-even}, we know that $E$ is given by a point in a projective bundle $\D{P}(\cale)\to U$ over an open subscheme $U$ of $\calm_X(v)$, the moduli space of torsion-free Gieseker-semistable sheaves having Chern character $$v:=\gamma(n,k)-\ch(\cccO_{X}(-1))+\ch(F_0).$$
Moreover there exists an open neighborhood $\calv$ of $[E]$ in $\D{P}(\cale)$ admitting a family $\mathbb{E}\in \Coh(\calv\times X)$ such that, for all $x\in \calv$, distinct from $[E]$, the bundle $\mathbb{E}_x$ satisfies Theorem \ref{thm:existence-even}, \ref{thm:exists-quadric}, \ref{thm:existence-even} for $i_X=4,\:3,\:2$ respectively.
Let us now denote by $\mathbb{E}|_S$ the pullback of $\mathbb{E}$ to $\calv\times S$.
The locus of point $x\in \calv$ such that ${(\mathbb{E}|_S)}_x$ is not slope-stable is closed in $\calv.$ Indeed, let $A$ be the set of degree 2 polynomials $P\in \D{Q}[t]$ with $\mu(P)=\frac{-r_X}{2}$ and such that there exists $x\in \calv$ for which $(\mathbb{E}|_S)_x$ admits a pure 2-dimensional quotient $F$ with $P_F=P$.
The set $A$ is finite and for all $P\in A$, the points $x\in \calv$ such that ${(\mathbb{E})|_S}_x$ admits a pure 2-dimensional quotient $F$ with $P_F=P$ define a closed subscheme $Z_P\subset \calv$, we refer again to \cite[§ 2.3]{HL}.

Notice that for $P\ne P_{E_n^k|_S}$, the closed subset $Z_P$ is strictly contained in $\calv$, as $[E]$ itself does not belong to $Z_P$.
Moreover the only Gieseker-stable sheaf on $S$ with $\ch(F_0|_S)$ is $F_0|_S$ itself, see for instance \cite[Theorem 6.1.6]{HL}, as $F_0|_S$ is Gieseker-stable and $\chi(F_0|_S,F_0|_S) = 2$.
We deduce that for a general $x\in \calv$, letting $E':=\mathbb{E}_x$, $(E'|_S)$ is not slope-stable only if $\Hom(F_0|_S,E'|_S)\ne 0$.

Let us check that this is a contradiction. From \eqref{eq:restriction-anticanonical}, applying $\Hom(F_0,\cdot)$, we get:
$$ 0\rightarrow \Hom(F_0,E')\rightarrow \Hom(F_0,E'|_S)\rightarrow \Ext^1(F_0,E'(-i_X)).$$
Then, it suffices to show that
both the left side and the right side terms of this sequence vanish.
However, this is easy to check, indeed, first, we have $\Hom(F_0,E')=0$ by slope-stability. Second, we have
$\Ext^1(F_0,E'(-i_X))\simeq \Ext^2(E',F_0)^*$, and this vanishes by the instanton cohomological vanishing when $i_X$ is even, while for $i_X=3$ this vanishing follows by the assumption that $E'$ satisfies Theorem \ref{thm:exists-quadric}.
This concludes Step \ref{K3-step2} and proves the theorem for $i_X \ne 1$.
\end{step}

\begin{step}[The case $i_X=1$]For the case $i_X=1$ the minimal instanton $F_0|_S$ may fail to be isolated in its moduli space -- actually this occurs when the genus is odd. So the previous argument needs some further adjustment before we can get to the same conclusion.

Let $\calt\to \calv\subset \Spl(\gamma(n+1,k))$ be an étale neighborhood of $[E]$ admitting a family of sheaves $\cale$ over $\calt\times X$ specializing to $E$. Consider then the pullback of $\cale$ to $\calt\times S$, which will be simply denoted by $\cale|_S$.
 We claim that, up to shrinking $\calv$, we can suppose that for all $x\in \calt$, the sheaf $\cale_x|_S$ is slope-semistable and that moreover it can at most be slope-destabilized by the restriction to $S$ of a minimal instanton on $X$.

Recall that, as we learned from the proof of Theorem \ref{thm:exist-index-1}, up to shrinking $\calv$, we have that 
all slope-semistable deformations of $E$ belong to the closed subscheme $Z_{\gamma(n,k)}$ image of $\Quot(\cale, \gamma(n,k))$.
Since by semicontinuity we can assume that for every point $x\in \calt, \ {\cale_x}|_S$ is simple, we get a well defined restriction map
$$\rho: \calv\to \Spl_S(\gamma_S(n+1,k)), \ \hspace{2mm} E\mapsto E|_S.$$

Up to replacing $\Spl_S(n+1,k)$ by an appropriate neighborhood $\calu$ of $[E|_S]$ we suppose the existence of a Poincaré bundle $\mathbf{E}$ on $\calu$. This allows us to define a closed subscheme $W_{\gamma_S(n,k)}$, 
image of $\Quot(\mathbf{E}, \gamma_S(n,k))\xrightarrow{\pi}\calu$, that  parametrises bundles on $S$ having quotients with Chern character $\gamma_S(n,k).$
The same arguments used to characterize the general slope-semistable deformations of $E$ lead to conclude that a general deformation of $E|_S$ is not slope-stable if and only if it admits a quotient having Chern character $\gamma_S(n,k).$
Note that this relies on the fact that, on $S$ as well as on $X$, there are no  semistable torsion-free rank-2 sheaves having Hilbert polynomial strictly greater than $P_{F_0|_S}$.

The deformations of $E$ whose restriction to $S$ admits a quotient having Chern character $\gamma_S(n,k):=\ch(E_n^k|_S)$ are parametrised by the closed 
subscheme $Z_{\gamma_S(n,k)}$ image of $\Quot(\cale|_S, \gamma_S(n,k))\xrightarrow{\pi} \calu$. This closed subscheme is the preimage $\rho^{-1}(W_{\gamma_S(n,k)})$. The characterization of $W_{\gamma(n,k)}$ presented above allows us to conclude that $Z_{\gamma_S(n,k)}$ parametrises sheaves whose restriction to $S$ is strictly $\mu$-semistable. 

By construction, we have $Z_{\gamma(n,k)}\subset Z_{\gamma_S(n,k)}$. We prove that these schemes actually coincide at the point $[E]$.
From the equality $Z_{\gamma_S(n,k)}=\rho^{-1}(W_{\gamma_S(n,k)})$, we get:
$$\dim T_{[E]}Z_{\gamma_S(n,k)}\le \dim(d_{[E]}\rho^{-1}T_{[E|_S]}W_{\gamma_S(n,k)})=\dim(T_{[E|_S]}W_{\gamma_S(n,k)}\cap \Ext^1(E,E)). $$
To understand the right side of this equation we recall that, due to the vanishing of $\Ext^2(E,E),$ the space $\Ext^1(E,E)$ injects via $d_{[E]}\rho$ into $\Ext^1(E|_S,E|_S)$, because $d_{[E]}\rho$ identifies with 
the linear map $\Ext^1(E,E)\to \Ext^1(E,E_S)$ induced by $E\twoheadrightarrow E_S$. 
This last point was already observed by
\cite{tyurin:fano-vs-CY}.

Note also that the cokernel of this injection is $\Ext^2(E,E(-1))\simeq \Ext^1(E,E)^*$ which tells us that $\Ext^1(E,E)$ is a maximal isotropic subspace of $\Ext^1(E|_S,E|_S)$ with respect to the skew-symmetric form $\omega$ defined by the Yoneda pairing.
We explicitly compute the dimension of $T_{[E|_S]}W_{\gamma_S(n,k)}\cap \Ext^ 1(E,E).$
Denote by $q_S$ the epimorphism $E|_S\twoheadrightarrow E_n^k|_S$. We recall the tangent space to the Quot scheme $\Quot(\mathbf{E},\gamma_S(n,k))$ at the point $q_S\in \pi^{-1}(E|_S)$ fits in an exact sequence (see e.g. \cite{HL}, Proposition 2.2.7):
$$ 0\rightarrow \Hom(F_0|_S,E_n^k|_S)\to T_{q_S}\Quot(\mathbf{E},\gamma_S(n,k))\xrightarrow{d_{q_S}\pi} T_{[E|_S]}\calu \to \Ext^1(F_0|_S, E_n^k|_S).$$
As $T_{[E|_S]}\calu\simeq \Ext^1(E|_S,E|_S)$ and as $\Hom(F_0|_S,E_n^k|_S)=0$ by the inductive hypothesis, we deduce that $T_{[E|_S]}W_{\gamma_S(n,k)}$ identifies with the kernel of the linear map:
$$ \sigma_S:\Ext^1(E|_S,E|_S)\simeq T_{E|_S}\calu\to \Ext^1(F_0|_S,E^n_k|_S).$$
Observe that $\sigma_S$ is nothing but the composition of the canonical linear maps
$\Ext^1(E|_S,E|_S)\to \Ext^1(E|_S,E_n^k|_S)$ and $\Ext^1(E|_S,E_n^k|_S)\to \Ext^1(F_0|_S,E_n^k|_S)$, see e.g. \cite{HL}, Section 2.2.

Therefore, the  space $T_{[E|_S]}W_{\gamma_S(n,k)}\cap \Ext^1(E,E)$ is the kernel of 
the restriction of $\sigma_S$ to $\Ext^1(E,E)$. But the latter factors through $\Ext^1(F_0,E)$, giving rise to the commutative diagram:
\begin{equation}\label{cd-k3}
\begin{tikzcd}[column sep=20pt]
0  \arrow[r,] & \Ext^1(E,E)\arrow[r]\arrow[d, "\sigma", two heads] & \Ext^1(E|_S,E|_S)\arrow[r, "\sigma"]\arrow[d, "\sigma_S"] & \Ext^1(E,E)^*\arrow [r]\arrow[d ]& 0\\
0\arrow[r] & \Ext^1(F_0,E_n^k)\arrow[r] & \Ext^1(F_0|_S, E_n^k|_S)\arrow[r]& \Ext^1(E_n^k, F_0)^* \arrow[r]
& 0
\end{tikzcd}
\end{equation}
where the exactness of the rows is due to the vanishing of $\ext^2(E,E),\: \ext^2(F_0,E_n^k)$ and $\ext^2(E_n^k,F_0)$.

Let us now look at the first column. The morphism $\sigma$ is nothing but the composition of the natural linear maps $\Ext^1(E,E)\to \Ext^1(E,E_n^k)$ and $\Ext^1(F_0,E_n^k)$ induced by $E\xrightarrow{q} E_n^k\to 0$. Since $\ext^2(E,F_0)=\ext^2(E_n^k,E_n^k)=0$, $\sigma$ is surjective. 
As moreover $\Hom(F_0,E_n^k)=0$ and $T_{[E]}\calt\simeq \Ext^1(E,E)$ by \cite[Proposition 2.2.7]{HL}, we get a short exact sequence:
$$ 0\rightarrow T_{[E]}Z_{\gamma(n,k)}\to \Ext^1(E,E)\to \Ext^1(F_0,E_n^k)\to 0.$$
Now, from the diagram we deduce that the linear space $T_{[E|_S]}W_{\gamma_S(n,k)}\cap \Ext^1(E,E)$ identifies with  $\ker(\sigma)\simeq T_{[E]}Z_{\gamma(n,k)}$.
This vector space has dimension $\ext^1(E,E)-\ext^1(F_0,E_n^k)$. 

But we know from the proof of Theorem \ref{thm:exist-index-1}-\eqref{eq:dim-chiuso}, that the local dimension of $Z_{\gamma(n,k)}$ at $[E]$ indeed is at least $\ext^1(E,E)-\ext^1(F_0,E_n^k)=\ext^1(E,E)-\chi(F_0,E_n^k)$.
We conclude that $Z_{\gamma(n,k)}$ and $Z_{\gamma_S(n,k)}$ coincide at $[E]$ so that the same will hold in a neighborhood of this point. 
A general deformation of $E$ is thus slope-stable and the same will hold for its restriction to $S$.

\end{step}
\end{proof}

This proves Theorem \ref{thm:anticanonical-stable}.
Along the way we have proved Theorem \ref{mainthm:anticanonical-stable} as well.

\section{Fano threefolds with curvilinear Kuznetsov component}

\label{section:curvilinear}

We briefly treat the cases of those Fano threefolds $X$ whose Kuznetsov component $\Ku(X)\subset D^b(X)$ is equivalent to the derived category $D^b(\Gamma)$ of a curve $\Gamma$, briefly, $X$ has a curvilinear Kuznetsov component. These are the Del Pezzo threefold $Y_4$ of degree 4 and the prime Fano threefolds $X_g$ of genus $g\in\{7,9,10\}.$

\subsection{Reminders about Fano threefolds with curvilinear Kuznetsov component}

We recall that the Kuznetsov component $\Ku(X)$ can be thought of as the "non-trivial" component of $\cccO_{X}(q_X)^{\perp}$, or equivalently of ${}^{\perp}\cccO_{X}(-q_X))$. If $X$ is a threefold of one of the types mentioned at the beginning of this section, we have that $\cccO_{X}(q_X)^{\perp}=\langle \Ku(X), \calf \rangle$ where $\calf$ is an exceptional bundle and $\Ku(X)$ is an indecomposable triangulated category. This situation is quite different from the case when $H^3(X)=0$, where $\Ku(X)$ is strongly generated by a full exceptional collection, we will treat this other case further on.

It turns out that the curve $\Gamma$ is identified with a certain fine moduli space of vector bundles on $X$. 
This moduli space is fine and admits a universal bundle $\cale$ on $X\times \Gamma$.
So, we will implicitly identify points of $\Gamma$ with 
vector bundles from that moduli space without further mention.
The equivalence $D^b(\Gamma)\to \Ku(X)$ is induced by a Fourier-Mukai functor having $\cale$ as kernel:
$$\Phi(-):=Rp_*(q^* (-)\otimes \cale).$$
The functor $\Phi$ is fully faithful and admits a left and a right adjoint $\Phi^{*}, \ \Phi^{!}$ defined as follows:
\begin{equation}
    \begin{split}
\Phi^{!}: & D^b(X)\to D^b(\Gamma), \hspace{3mm} \Phi^{!}(-):=Rq_*(p^* (-)\otimes \cale^*\otimes \omega_{\Gamma})[1];\\
\Phi^{*}: & D^b(X)\to D^b(\Gamma), \hspace{3mm} \Phi^{*}(-):=Rq_*(p^* (-)\otimes \cale^*\otimes \omega_{X})[3];
\end{split}
\end{equation}
The equivalence $\Phi:D^b(\Gamma)\to \Ku(X)$ gives rise to the following decomposition of $D^b(X):$
\begin{equation}\label{eq:dec-dc-curve}
D^b(X)=\langle \cccO_{X}(-q_X), \calf^*, \Phi(D^b(\Gamma))\rangle .
\end{equation}

Recall that we are dealing with Fano threefolds which are either isomorphic to $Y_4$ or to $X_g$ for $g\in \{7,9,10\}$.
For each of these threefolds $X$, we will recall the main features of the curve $\Gamma$, of the sheaf $\cale$ and of the derived category $D^b(X)$. We will then study the objects $\Phi^{!}(E(r_X))$ for $E$ a $(n,k)$-instanton on $X$.
We will essentially mimic what is done in the rank-2 cases, treated e.g. in \cite{Faenzi,BF-7, BF-9,Kuz-hs,Kuz-hilb, kapustka-ranestad, FV}. We will adopt these as the main references throughout the current section.

\subsubsection{The Del Pezzo threefold of degree 4} A Del Pezzo threefold $Y_4$ is the complete intersection of 2 quadrics in $\D{P}^5$. In this case, the curve $\Gamma$ is a smooth genus 2 curve, double cover of the line $\D{P}^1$ corresponding to the pencil of quadrics defining $X$, ramified along 6 points given by the singular quadrics in the pencil. As it turns out, the curve $\Gamma$ is isomorphic to the moduli space $\calm_{Y_4}(v)$ of Gieseker-stable rank-2 sheaves having Chern character $v\in \oplus H^{i,i}(Y_4), \ v=2+H_{Y_4}-\frac{1}{3} p_{Y_4}.$ The universal bundle $\cale$ of this moduli space has the following topological invariants:
$$ c_1(\cale)=H_{Y_{4}}+N, \hspace{2mm} c_2(\cale)=2l_{Y_4}+H_{Y_4}M+\eta$$
where $M$ and $N$ are divisors on $\Gamma$ of respective degrees $m$ and $2m-1$. Here $m$ is an arbitrary integer and $\eta$ a class in $H^3(Y_4)\otimes H^1(\Gamma)$ such that $\eta^2=4$. In order to simplify the upcoming computations, we set $m=2$. In addition to this we have that for all $y\in \Gamma$, $\cale_y$ is a slope-stable ACM and globally generated bundle. The evaluation morphism gives rise to a short exact sequence:
\begin{equation}\label{eq:curve-Y_4}
 0\rightarrow \cale_{\bar y}(-1)\rightarrow H^0(\cale_y)\otimes \cccO_{Y_4}\to \cale_y\rightarrow 0,
 \end{equation}
where $\bar y$ is the conjugate point to $y$ in the double cover $\Gamma\to \D{P}^1$.
The derived category $D^b(Y_4)$ of $Y_4$ admits the following semiorthogonal decomposition
  $$ D^b(Y_4)=\langle \cccO_{Y_4}(-1),\cccO_{Y_4},  \Phi(D^b(\Gamma))\rangle. $$

\subsubsection{Prime Fano threefolds of genus 10} Let us consider a 7-dimensional complex vector space $V\simeq \D{C}^7$ and a general 3-form $\omega\in \bigwedge ^3 V^*.$ The stabilizer of $\omega$ in $GL(V)$ is the 14-dimensional group $G_2\subset SO(V).$ Let us consider the Grassmannian $\Gr(2,V)$ and the tautological 2 bundle $\calu$. The form $\omega$ defines a $G_2$-invariant global section $s_{\omega}$ of the rank-5 vector bundle $\calu^{\perp}(1)$. The zero locus of $s_{\omega}$ is a smooth $G_2$-homogeneous  5-dimensional Fano manifold, namely the $G_2$ Grassmannian $G_2\Gr(2,V)$, which is the adjoint variety of type $G_2$.
The vector space $\wedge^2 V$ decomposes as the direct sum of two irreducible $G_2$ representations $\wedge^2 V\simeq V\oplus W$ with $W$ isomorphic to the Lie algebra $\mathfrak{g}_2$. The $G_2$ Grassmannian $G_2\Gr(2,V)$ is non-degenerate in $\D{P}(W)$ and its projective dual $G_2\Gr(2,V)^{\vee}$ is a sextic hypersurface in $\D{P}(W^*).$

Consider now a general line $\D{P}(A)\simeq \D{P}^1\subset \D{P}(W^*).$ The linear section $\D{P}(A^{\perp})\cap G_2\Gr(2,V)$ is a smooth prime Fano threefold $X_{10}\subset \D{P}^{11}$ of genus 10. In the pencil $\D{P}(A)$ there exists 6 singular hyperplane sections given by the points in $\D{P}(A)\cap G_2\Gr(2,V)^{\vee}$. We consider curve $\Gamma$ of genus 2 obtained as double cover $\Gamma\to \D{P}(A)$, ramified along these 6 points.

The curve $\Gamma$ identifies with the fine moduli space $\calm_{X_{10}}(v)$ of slope-stable sheaves with Chern character  $v=3+2H_{X_{10}}+9l_{X_{10}}+\frac{1}{2}p_{X_{10}}.$ The universal sheaf $\cale$ on $X\times \Gamma$ has rank 3 and Chern classes:
$$ c_1(\cale)=2H_{X_{10}}+K_C, \hspace{2mm} c_2(\cale)=27l_{X_{10}}+3H_{X_{10}}p_C+\eta, \hspace{2mm} c_3(\cale)=7p_{X_{10}}+22l_{X_{10}}p_C$$
where $\eta$ is a class in $H^3(X_{10})\otimes H^1(\Gamma)$ such that $\eta^2=4$.
For all $y\in \Gamma$, $\cale_y$ is locally free, slope-stable, globally generated and fits in an exact sequence:
\begin{equation}\label{eq:curve-X_{10}}
 0\to \cale_y(-1)\to \cccO_{X}^6\to {\calu^*}^3\to \cale_y\to 0
 \end{equation}
 for $\calu$ the restriction to $X_{10}$ of the rank-2 tautological bundle on $\Gr(2,V)$, see \cite{kapustka-ranestad, FV}. We recall that $c_1(\calu)=-1.$
The derived category $D^b(X_{10})$ decomposes as
$$ D^b(X_{10})=\langle\cccO_{X_{10}},\calu^{*}, \Phi(D^b(\Gamma))\rangle. $$

\subsubsection{Prime Fano threefolds of genus 9}
Let us consider a 6-dimensional complex vector space $V$ endowed with a non-degenerate skew-symmetric form $\omega$. The maximal $\omega$-isotropic linear subspaces of $V$ are parametrised by the Lagrangian Grassmannian $\Lgr(3,6)\subset \D{P}^{13}$, a smooth 6-dimensional projective variety.
For a general $\D{P}(A)\subset {\D{P}^{13}}^{*}, \ \D{P}(A)\simeq \D{P}^2$, the linear section $\Lgr(3,6)\cap \D{P}(A^{\perp})$ is a smooth prime Fano threefold $X_9$ of genus 9, non-degenerate in $\D{P}^{10}.$

The projective dual $\Lgr(3,6)^{\vee}$ to $\Lgr(3,6)$ is a quartic hypersurface in ${\D{P}^{13}}^{*}$ whose linear section $\Lgr(3,6)^{\vee}\cap \D{P}(A)$ is a smooth (by generality of $\D{P}(A)$) plane curve $\Gamma$ of degree 4, defined as the homological projective dual to $X_9$.

The curve $\Gamma$ identifies with the fine moduli space $\calm_{X_{9}}(v)\simeq \calm_{X_{9}}(\gamma(1,0))$ of slope-stable sheaves with Chern character  $v=2+H_{X_9}+2l_{X_9}-\frac{1}{3}p_{X_9}$. The isomorphism $\calm_{X_9}(v)\simeq \calm_{X_9}(\gamma(1,0))$ is given by $\cale_y\mapsto \cale_y^*\simeq \cale_y(-1)$.
We have the following invariants:
$$ c_1(\cale)=H_{X_9}+N, \hspace{2mm} c_2(\cale)=6L_{X_9}+H_{X_9}M+\eta.$$
Here, $M$ and $N$ are divisors on $\Gamma$ of degrees $m$ and $2m-1$ and $\eta\in H^3(X_9)\otimes H^1(\Gamma)$ satisfies $\eta^2=6.$
The integer $m$ can be chosen arbitrarily, however we fix $m=3$, which will make some computations easier.
For each point $y\in \Gamma,$ $\cale_y$ is a slope-stable ACM bundle on $X_9$ fitting in
\begin{equation}\label{eq:universal-9}
    0\rightarrow \cccO_{X_{9}}\rightarrow \calu^*\rightarrow \cale_y\rightarrow \cccO_{C}\rightarrow 0,
\end{equation}
where $C\subset X_9$ is a conic and $\calu$ is the restriction to $X_9$ of the rank-3 tautological bundle on $\Lgr(3,6).$ We recall that $\calu$ is a rank-3 vector bundle of first Chern class $c_1(\calu)=-1$.
We have the following decomposition of the derived category $D^b(X_9):$
$$D^b(X_9)=\langle\cccO_{X_{9}}, \calu^*, \Phi(D^b(\Gamma))\rangle.$$

\subsubsection{Prime Fano threefolds of genus 7}
Consider the orthogonal Grassmannian $\Ogr(5,10)$ parametrising 4-dimensional linear subspaces contained in a smooth 8-dimensional quadric hypersurface $Q$ in $\D{P}^9$. The variety $\Ogr(5,10)$ is smooth, of dimension 10 and has 2 connected components $\Sigma_{\pm}$. Each component spans a 15-dimensional linear spaces $\langle \Sigma_{\pm}\rangle \simeq \D{P}(V_{\pm})\simeq \D{P}^{15}.$ We denote by $\calu_{\pm}$ the restriction to $\Sigma_{\pm}$ of the tautological rank-5 bundle on $\Ogr(5,10)$.
For a sufficiently general linear subspace $\D{P}(A)\simeq \D{P}^8$ of $\D{P}(V_+)$, the resulting linear section $\D{P}(A)\cap \Sigma_+$ is a smooth prime Fano threefold $X_7$ of genus 7.

Let us consider the orthogonal $\D{P}(A^{\perp})\simeq \D{P}^6\subset \D{P}(V_{-})$. By the generality assumption on $A$, $\D{P}(A^{\perp})\cap \Sigma_{-}\subset \D{P}(V_{-})$ is a smooth curve $\Gamma$. This curve has genus 7 and is referred to as the homological projective dual to $X_7$.
The curve $\Gamma$ is the moduli space of Gieseker-stable bundles with Chern character $v\in \oplus_i H^{i,i}(X_7), \ v=2+ H_{X_7}+l_{X_7}-\frac{1}{2}p_{X_7}$. So $\Gamma\simeq\calm_{X_{7}}(v)\simeq \calm_{X_7}(\gamma(1,0))$. Like for the case of $g=9$, all bundles in this moduli space are duals of minimal rank-2 instantons.
The universal bundle $\cale$ on $X\times \Gamma$ has Chern classes:
$$ c_1(\cale)=H_{X_7}+H_{\Gamma}, \hspace{2mm} c_2(\cale)=\frac{7}{12}H_{X_{7}}H_{\Gamma}+5l_{X_7}+\eta$$
for $\eta\in H^3(X_{7})\otimes H^1(\Gamma)$ a class such that $\eta^2=14$.
Moreover it fits in
\begin{equation}\label{eq:universal-7}
    0\rightarrow \cale^*\rightarrow \calu_{-} \rightarrow \calu_{+}^*\rightarrow \cale\rightarrow 0
    \end{equation}
for $\calu_{\pm}$ the pullback of $\calu_{\pm}$ to $X\times \Gamma\subset \Sigma_+\times \Sigma_-.$ The restriction of $\calu_+$ to $X$ will still be denoted by $\calu_+$. This is a rank-5 vector bundle on $X$ having $c_1(\calu_+)=-2.$
As already recalled in \S \ref{section:prime}, for all $y\in \Gamma,$ $\cale_y$ is a locally free slope-stable and globally generated bundle.
The derived category $D^b(X_7)$ decomposes as
$$ D^b(X_7)=\langle \cccO_{X_{7}}, \calu^*_{+}, \Phi(D^b(\Gamma)) \rangle.$$

\subsection{The canonical resolution of instantons}
For a slope-stable $(n,k)$-instanton $E$ on $X$,
we look at the object $\Phi^{!}(E(r_X))$ in $D^b(\Gamma).$ Let us denote by $\calf$ the exceptional object on $X$ such that $^{\perp}\cccO_{X}(-q_X)=\langle \calf^*, \Phi(D^b(\Gamma))\rangle$. We recall that $\calf\simeq \cccO_{X}$ for $X=Y_4$, $\calf\simeq \calu$ for $X_g, $ $g\in \{9,10\}$ and $\calf\simeq \calu_+$ for $X_7$.
In the following lemma, $X$ is of type $Y_4$ or $X_g$ with $g \in \{7,9,10\}$.

\begin{Lemma}\label{lem:hdi}
Let $E$ be a slope-stable $(n,k)$-instanton on $X$ with $k \ge 1$.
\begin{enumerate}[label=\roman*)]
    \item \label{Ku-i} We have $\ext^i(\cale_y, E(r_X))=0$, for all $i\ne 1$, and all $y\in \Gamma$.
\item \label{Ku-ii} If $X=Y_4$ or $X=X_{10}$, $\hom(\calf^*,E(r_X))=0$.
\end{enumerate}
\end{Lemma}
\begin{proof}
The proof is essentially equivalent to the proofs of \cite[Lemmas 5.3-5.4]{BF-7}  and \cite[Lemmas 4.3-4.4]{BF-9}. For \ref{Ku-i}, let $y\in \Gamma$ be an arbitrary point. The vanishing of $\ext^i(\cale_y,E(r_X))$ for $i=0,3$ is an immediate consequence of stability. The vanishing of $\ext^2(\cale_y,E(1))$ for $X_{7}$ and $X_{9}$ were shown in Lemma \ref{lem:ext2-index1}. Let us prove that $\ext^2(\cale_y,E(r_X))=h^2(E(r_X)\otimes \cale_y^*)=0$ also on $Y_4$ and $X_{10}.$
 For $Y_4$ it suffices to tensor \eqref{eq:curve-Y_4} by $E(-1)$.
 As $\cale_y^*\simeq \cale_y(-1)$ we then use that $H^2(E(-1))=0$ and $H^3(E\otimes \cale_y(-2))=0$, again by stability.
 For $X_{10}$ let us consider the image $\calg$ of the morphism $\cccO_{X_{10}}^6\to {\calu^*}^3$ appearing in the complex \eqref{eq:curve-X_{10}}. As $H^2(E)=0, \ \Ext^2(\cale_y(-1),E)$ injects into $\Ext^3(\calg,E)\simeq \Hom(E,\calg(-1))^*$. But $\Hom(E,\calg(-1))$ injects into $\Hom(E,\calu^3)\simeq \Hom(E,{\calu^*(-1)}^3)$ which is zero by slope-stability and as we are assuming $k\ge 1.$

 Also \ref{Ku-ii} is a direct consequence of the slope-stability assumption. Indeed, for $X=Y_4$ or $X=X_{10}$, $\calf$ is a minimal instanton, in particular $\calf^*$ and $E(r_X)$ are slope-stable of the same slope, but $E(r_X)$ is not isomorphic to $\calf^*$ by the assumption $k\ge 1$.
 \end{proof}

For the cases $X=X_g, \ g\in \{7,9\}$, we will show that the vanishing $\Hom(\calf^*,E(1))$ holds at least on an open non-empty subscheme $\calv(n,k)\subset \calm_X(\gamma(n,k))$.
To prove this we need a preliminary lemma.
\begin{Lemma}\label{lem:morph-inj}
Let $E$ be a slope-stable $(n,k)$-instanton bundle on $X_g, \ g\in\{7,9\}$ for $n\ge 2$. Then a non-zero morphism $\calf^*\to E(1)$ is injective and its cokernel is torsion-free.
The same holds for any non-zero morphism $\calf^*|_S\to E(1)|_S$ for $S\in|\cccO_X(1)|$ general and $E$ a general slope-stable instanton bundle belonging to $\calr_S(n,k)\subset \calmi_X(n,k)$.
\end{Lemma}
\begin{proof}
We first prove that any non-zero morphism $\calf^* \to E(1)$ is injective.
Let us look at $X=X_9$. We have $\calf^*\simeq \calu^*$ is slope-stable of rank 3 and $c_1=1.$
Since we are assuming that $E$ is slope-stable and that $\rk(E)\ge 4,$ we see that for a non-zero morphism $f:\calu^*\to E(1)$, $\im(f)$ must have rank 3. Since $\calu^*$ is torsion-free we conclude that $\im(f)\simeq \calu^*$.
\medskip

Let us pass to the case $g=7$. Recall that $\calf^*\simeq \calu^*_+$ is slope-stable with slope $2/5$. Consider a non-zero map $f : \calf^* \to E(1)$.
Suppose, by contradiction, that $f$ is not injective. By stability of $\calf$ and $E$, we see that $\im(f)$ must have slope $1/2$, so, since $n \ge 2$ and $E$ is slope-stable, we must have $\rk(\im(f))=4$. If $n \ge 3$, this contradicts stability of $E$, so we look at the case $n=2$.
Note that $\ker(f)$ is reflexive of rank one and slope $0$, so $\ker(f) \simeq \cccO_S$, so that $\ch(\im(f))=\ch(\calf^*)-1$.
Therefore, $\coker(f)$ is a torsion sheaf with
$c_1(\coker(f))=0$ and
$c_2(\coker(f))=(k-2)l$, which is impossible if $k \ge 3$. For $k=2$, we get that $\coker(f)$ is a torsion sheaf with
$c_1(\coker(f))=c_2(\coker(f))=0$ and $c_3(\coker(f))=-2p$, where $p$ is a point of $X$,  which is again impossible.
This ends the proof of the injectivity of non-zero morphisms $\calf^*\to E(1).$
\medskip

Given a morphism $f:\calf^*\hookrightarrow E(1)$, we denote by $Q$ its cokernel. Let us check that $Q$ is torsion-free. Otherwise, we would have that $f$ factors through an injective morphism $f':G\rightarrow E(1)$ whose cokernel is the torsion-free part of $Q$. This means that $G$ is an extension of $T(Q),$ the torsion part of $Q$ by $\calf^*$.
Due to the stability of $E(1),$ $\dim(T(Q))\le 1$. But then we have $\ext^1(T(Q),\calf^*)= \ext^2(\calf^*,T(Q)(-1))=h^2(\calf\otimes T(Q)(-1))=0.$ Therefore we would have that $G$ has torsion which is impossible since it injects into $E(1).$

Let us consider now a smooth hyperplane section $S\in |\cccO_X(1)|$ such that $\Pic(S)\simeq \D{Z}$ and assume that $E\in \calr_s(n,k)$, so that $E|_S$ is slope-stable. Then the same arguments presented for $X$ allow us to deal with the case $S\subset X_9$ and for the case $S\subset X_7$ except if $(n,k)=(2,2)$.
In this case, a non-injective morphism $\calu_+^*|_S\to E(1)|_S$ would lead to a short exact sequence:
$$ 0\rightarrow \cccO_S\rightarrow \calu^*_+|_S\rightarrow E(1)|_S\rightarrow 0.$$

But by the generality assumption on $E$, we can assume that $E$ is ACM, see Corollary \ref{cor:ACM}).
This implies that the graded module $H^1_*(E|_S):=\bigoplus H^1(E_S(m))$ is zero which would lead to $\calu_+^* |_S\simeq \cccO_{S}\oplus E|_S(1)$ which is impossible.
The proof of the torsion-freeness of the cokernel of a morphism $\calf^*|_S\to E|_S(1)$ is carried out in an analogous way to the 3-dimensional case.
\end{proof}

\medskip

\begin{proposition}\label{prop:aperto}
Let $X=X_g, \ g\in\{7,9\}$. Then, for all $n\ge 1$, $k\ge 2$ there exists an open non-empty subscheme $\calv(n,k)\subset \calm_X(\gamma(n,k))$ parametrising slope-stable $(n,k)$-instantons $E$ on $X$ such that $\Hom(\calf^*, E(1))=0$ and for $S\in |\cccO_X(1)|$ general, $\Hom(\calf^*|_S,E(1)|_S)=0.$
\end{proposition}
\begin{proof}
 We work by induction on $n$.
 For the rest of the proof $S$ will be a smooth hyperplane section of $X$ with Picard rank one.
 For $n=1$, let $E$ be a $(1,k)$ slope-stable instanton bundle with $k\ge 2$. Let us suppose at first that $X=X_9$.
 Then $\hom(\calu^*, E(1))=0$ by \cite[Lemma 4.4]{BF-9}.
 An argument similar to the one presented in loc. cit. also allows us to prove that, assuming $E$ is general, we can additionally suppose that $\hom(\calu^*|_S,E(1)|_S)=0.$ Indeed, by the assumptions on $S$, $E|_S$ is slope-stable due to the vanishing of $H^0(E|_S)$.

 The kernel of a nonzero morphism $f:\calu^*|_S\to E(1)|_S$ must be a rank-1 reflexive sheaf, so that, as $\Pic(S)\simeq \D{Z}H_S$, $\ker(f)\simeq \cccO_{X_{7}}$. Therefore one computes $\ch(\im(f))=(2,1,0)$ and $\ch(\coker(f))=(0,0,2-k).$ Now, $\coker(f)$ is a torsion sheaf whose length equals $\ch_2(\coker(f))$ so that $k\le 2.$ The case $k=2$ can be excluded as well. Indeed, let $E$ be a $(1,2)$-instanton bundle on $X$. Since we are assuming $E$ general, we might suppose that $E$ is ACM. Assume now, by contradiction, that $E(1)|_S$ fits in:
 $$ 0\rightarrow \cccO_S\rightarrow \calu^*|_S\rightarrow E(1)|_S\rightarrow 0$$
 Since $h^2(\calu^*|_S)=h^0(\calu|_S)=0$ and $h^2(\cccO_S)=h^0(\cccO_{S})=1$, we would get $h^1(E(1)|_S)\ne 0$. Indeed, otherwise, since the sequence cannot split, we would have $\ext^1(E(1)|_S,\cccO_S)=h^1(E|_S)\ne 0)$,
 which is impossible since if $E$ is ACM, $H^1_*(E|_S):=\bigoplus H^1(E|_S(m))=0$.

 For $X=X_7$ the vanishing of $\Hom(\calu_+^*,E(1))$ for a $(1,k)$ instanton, for $k\ge 2$ is shown in \cite[Lemma 5.3]{BF-9}.
 For an arbitrary point $y\in \Gamma,$ we have a short exact sequence of bundles on $S$:
 \begin{equation}\label{eq:taut-restriction}
 0\rightarrow \calg_y|_S\rightarrow \calu_+^*|_S\rightarrow \cale_y|_S\rightarrow 0.
 \end{equation}
 
 The vector bundle $\calu_+^*|_S$ is slope-stable, see e.g. \cite[Theorem 3.3]{Mukai-curves}, and the same holds for $E|_S$, which is still ensured by $H^0(E|_S)=0$, and $\calg_y|_S$.
 For the stability of the latter it is enough to notice that $H^0(\calg_y|_S(-1))\hookrightarrow H^0(\calu_+^*|_S(-1))=0$, the vanishing being ensured by the slope-stability of $\calu_+^*|_S$. This gives that $\calg_y$ cannot be destabilized by rank-1 subsheaves, $H^0(\calg_y^*|_S)=0$ since $H^0(\calu_+|_S)=0$ (again by stability) and $H^1(\cale_y^*|_S)=H^1(\cale_y(-1))=0$ since $\cale_y$ is ACM. This, as already recalled, implies $H^1_*(\cale_y|_S) = 0$. This means that $\calg_y|_S$ cannot be destabilized by rank-2 subsheaves either. Hence
 $\calg_y|_S$ is a slope-stable rank-3 bundle of degree 1 and the same arguments used for $\calu^*$ on $X_9$ allow us to conclude that $\Hom(\calg_y|_S, E(1)|_S)=0$.
 Indeed, again we would have that the kernel of a non-zero map  $\calg_y|_S\to E(1)|_S$ must be $\cccO_S$ and hence a cokernel with Chern character $(0,0,2-k)$, which cannot occur.
 In conclusion, $\Hom(\calu_+^*|_S,E(1)|_S)$ vanishes as well.
 
 \medskip
 Let us now pass to the inductive step.
 We consider $n>1$, $k\ge 2$ and we take a $(n+1,k)$ instanton $E$ obtained as a non-trivial extension $e\in \Ext^1(E_n^k,F_0)$ for a slope-stable $(n,k)$-instanton bundle $E_n^k$ belonging to $\calr_S(n,k)$ and satisfying the proposition. By the assumptions on $E_n^k$ the restriction $e|_S$ does not split and $E_n^k|_S$ is slope-stable. Accordingly, by Lemma \ref{lem:ext-ss}, $E|_S$ is strictly slope-semistable and $F_0|_S$ is its only slope-semistable subbundle of slope $\frac{-1}{2}$. All subsheaves of $E|_S$ having slope $\frac{-1}{2}$ will indeed be elementary transformation of $F_0|_S$ along a zero-dimensional scheme.

Let us consider now an étale neighborhood $\calt$ of $[E]$ in $\Spl_X(\gamma(n+1,k))$ admitting a universal bundle $\cale\in \Coh(\calt\times X)$. Denote by $\cale|_S$ the pullback of $\cale$ to $\calt\times S$. Let $P:=P_{E}-P_{\calu^*(-1)}$, $P_S:=P_{E|_S}-P_{\calu^*(-1)|_S}$ and denote by $Z_{P_S}\subset \calt$, resp. $Z_P\subset \calt$, the closed subscheme parametrising points $x\in \calt$ such that ${(\cale|_S)}_x$ admits a pure 2-dimensional quotient with Hilbert polynomial $P_S$, resp. $\cale_x$ admits a torsion-free quotient with Hilbert polynomial $P$.
We claim that $[E]\not\in Z_{P_S}.$
Suppose, by contradiction that $E|_S$ fits in a short exact sequence of $\cccO_S$-modules:
$$ 0\rightarrow K\rightarrow E|_S\rightarrow Q\rightarrow 0$$
with $Q$ a torsion-free $\cccO_S$ module with Hilbert polynomial $P_S$.
Notice that the torsion-freeness of $Q$ ensures that $K$ is locally free, see e.g. \cite[Prop. 1.1]{Hart-ss}.

We claim that $K$ is slope-stable. Notice that, due to the slope-semistability of $E|_S$, the only possible destabilizing subsheaves $L$ of $K$ must have slope $\frac{-1}{2}$, as for the odd rank cases we must have $\mu(L)< -\frac{m}{2m+1}$ with $m=0$ for $X_9$ and $m=0,1$ for $X_7$. But then $L\hookrightarrow E|_S$ factors through $F_0|_S$ due to the slope-stability of both $F_0|_S$ and $E_n^k|_S$ and Lemma \ref{lem:ext-ss}. So $\rk(L)=2$ and $T:=F_0|_S/L$ is zero-dimensional. Notice that the latter is nothing but the torsion part of $E|_S/L$ -- here the dimension count is due to the slope-semistability of $E|_S$.
But this would mean that $L^{**}\simeq F_0|_S$ and since $K$ is a vector bundle, this would imply that also the inclusion $L\hookrightarrow K$
factors through $F_0|_S$.
Thus, to prove that $K$ is slope-stable, it suffices to check that there is no inclusion $F_0|_S\hookrightarrow K$.
Note that such an inclusion would lead to a commutative diagram:
\begin{equation}\label{cd-non-inst-onemore0}
\begin{tikzcd}
0  \arrow[r]& F_0|_S\arrow[r]\arrow[d] & K\arrow[r]\arrow[d] & I\arrow[d]\arrow[r] & 0  \\
0  \arrow[r] & F_0|_S\arrow[r] & E|_S\arrow[r] & E_n^k|_S\arrow [r]& 0
\end{tikzcd}
\end{equation}
This would provide a short exact sequence
$$ 0\rightarrow I \rightarrow E_n^k|_S\rightarrow Q\rightarrow 0.$$
Note that, since $Q$ is torsion-free, $I$ must be locally free.

\bigskip
Let us show that, when $X\simeq X_7$ or $X \simeq X_9$, this cannot happen.

\begin{itemize}
    \item We first treat the case $X_9$. In this case we get that $I$ must be a line bundle. But this is not possible since $\ch(I)=(1,-1,6)$.
    \item Now let us focus on the case $X_7$.
    In this case we have $\ch(I)=(3,-2, 5)$, so that \mbox{$n\ge 2, \ \rk(E_n^k|_S)\ge 4$.}
    We claim that we can assume that $E_n^k|_S$ admits no injective morphisms from vector bundles with Chern character $\ch(I)$ (which is the same as requiring that $E_n^k|_S$ has no torsion-free quotient with Chern character $\ch(E_n^k|_S/I)$).
    This will be enough to prove that also on $S\subset X_7$ we have slope-stability of $K$.

    We will prove the claim by induction on $n$.
    For $n=1$ this is obvious.
    Supposing now that the claim holds for a general slope-stable instanton bundle $E_n^k, \ k\ge 2,$ such that $E_n^k|_S$ is slope-stable, we show that it holds for any $E$ arising from $\Ext^1(E_n^k,F_0).$
    Suppose indeed by contradiction that $E|_S$ admits a subbundle $I$ with $\ch(I)=(3,-2,5)$. Then using the same arguments previously presented, we get that $I$ must be slope-stable.

    Indeed, if $I$ is not slope-stable, it can only be destabilized by a rank-2 sheaf $F$ with $c_1(-1)$. But then we have an injection $F\hookrightarrow F_0|_S$ and an isomorphism $F^{**}\simeq F_0|_S$, which gives an injection $F_0|_S\hookrightarrow I$. We obtain a rank-1 subsheaf $L$ of $E_n^k|_S$ with Chern character $\ch(L)=(1,-1,4)$. But since $I$ is locally free,
    $E|_S/I\simeq E_{n+1}^k|_S/ L$ is torsion-free so that $L$ must be a line bundle, which is impossible due to its Chern character.
    But if $I$ is a slope-stable bundle, admitting no injective morphism into $E_n^k|_S$ (which is true by the inductive hypothesis) is equivalent to $\Hom(I, E_{n+1}^k|_S)=0$, as it can be shown with arguments equivalent to the ones presented in Lemma \ref{lem:morph-inj}.
    But then $I\hookrightarrow E|_S$ should factor through $F_0|_S$ which is impossible. A general deformation $E'$ of $E$ will satisfy that $E'|_S$ is slope-stable and admits no injective morphisms from vector bundles with Chern character $\ch(I)$.
    \end{itemize}

Summing up, we have that $K$ is a slope-stable vector bundle. But as $\chi(K,K)=2$, we must have $K\simeq \calf^*|_S$, see e.g. \cite{Mukai-bundles}. By the inductive hypothesis $\Hom(\calf^*|_S, E_n^k|_S)=0$ so that $\calf^*|_S\simeq K\hookrightarrow E|_S$ would factor through $F_0|_S$. But this cannot happen. Therefore we can conclude that  $[E]\not\in Z_{P_S}$, hence $[E]\not\in Z_P$. So, for a general point $x\in \calt$ we have $x\not\in Z_{P_S}$ and $x\not\in Z_P$. Assuming that moreover $\cale_x$ is slope-stable and that the same holds for its restriction to $S$, we conclude that, from Lemma \ref{lem:morph-inj}, $\Hom(\calf^*,E')=\Hom(\calf^*|_S,E'|_S)=0.$
\end{proof}

Applying the previous results we can construct canonical resolutions for instantons.

\begin{theorem}
\label{thm:curve}
Let $E$ be a slope-stable $(n,k)$-instanton on $X$ and set $F:=\Phi^{!}(E(r_X))$. Then the following hold:
\begin{enumerate}[label=\roman*)]
    \item \label{Phi!i} $F$ is a vector bundle on $\Gamma$ of rank $r$ and degree $d$ for $(r,d)$ taking the following values:
    \begin{equation}\label{eq:deg-rank}
\begin{tabular}{||c | c | c ||}
 \hline
 X & r & d \\ [0.5ex]
 \hline\hline
$Y_4$ & k& 0\\
 \hline
$X_{10}$ & k & k \\
 \hline
$X_9$ & k& n+k  \\
 \hline
$X_7$ & k& 5k+n \\[1ex]
 \hline
\end{tabular}
\end{equation}
    \item \label{Phi!ii} The sheaf $E(r_X)$ fits in a short exact sequence:
    $$ 0\rightarrow \ext^{1+r_X}(E(r_X),\calf)^*\otimes \calf^*\rightarrow \Phi(F)\rightarrow E(r_X)\rightarrow 0;$$
    \item \label{Phi!iii} If $X=Y_4$ or $X=X_{10}$ and $k\ge 1$, $F$ is simple. The same holds on $X_g, \ g\in\{7,9\}$ with the additional assumption that $E\in \calv(n,k), \ k\ge 2. $
\end{enumerate}
\end{theorem}

\begin{proof}
Item \ref{Phi!i} is an immediate consequence of Lemma \ref{lem:hdi}. Since we have $h^i(E\otimes \cale_y^*)=0$ for all $i\ne 1$ and all $y\in \Gamma$, $F:=\Phi^{!}(E(r_X))$ indeed is a sheaf and as $h^1(E(r_X)\otimes \cale_y^*)=-\chi(\cale_y,E(r_X))$ we conclude that actually $F$ is locally free of rank $-\chi(\cale_y,F_0).$ The computations of the ranks and degrees appearing in the table \eqref{eq:deg-rank} are a direct application of Riemann-Roch and Grothendieck-Riemann-Roch theorem.
\medskip

 Let us now pass to the proof of \ref{Phi!ii}. We first deal with the case where $X\simeq Y_4$ so that $i_X=2$ and $r_X=0.$
For a $(n,k)$-instanton on $Y_4$, we consider the right mutation $\tilde{E}:=\D{R}_{\cccO_{X}}(E)$, defined as the kernel of the evaluation map $E\to \bigoplus_i \Ext^i(E,\cccO_{X})\otimes \cccO_{X}[i]$. By construction $\tilde{E}\in {}^{\perp}\langle \cccO_{X}(-1),\cccO_{X}\rangle\simeq \Phi(D^b(\Gamma))$ and as moreover $\Ext^i(E,\cccO_{X})$ is concentrated in degree 1, we get that $\tilde{E}$ is a sheaf, actually it is a $(k,k)$-instanton, fitting in:
\begin{equation}\label{eq:mutation-dp}
 0\rightarrow \cccO_{X}^{k-n}\to \tilde{E}\to E\rightarrow 0.
\end{equation}

 Let us now see how to obtain a similar complex when $i_X=1$, so that $q_X=0$ and $r_X=1$.
 Given a $(n,k)$-instanton $E$ on $X$, we consider the right mutation $\tilde{E}(1):=\D{R}_{\langle\cccO_{X},\calf^*\rangle}(E(1))$ of $E(1)$ by the exceptional pair $\langle\cccO_{X},\calf^*\rangle.$ By definition this is the kernel of the evaluation morphism \[E(1)\to \bigoplus_i \Ext^{i}(E(1),\cccO_{X})^*\otimes \cccO_{X}[i]\oplus \Ext^{i-1}(E(1),\calf)^*\otimes \calf^*[i-1].\]
 We note that $\Ext^{i}(E,\calf)=0$ except for $i=2$.
 Indeed, $\ext^i(E(1),\calf)=0$ for $i=0,3$ by stability of $\calf$ and $E$ and since $k\ge 1$. The stability assumption also leads to $\ext^1(E(1),\calf)=0$.
 To verify this recall that, for $g=7,9,10$, we have a tautological exact sequence:
 $$ 0\rightarrow \calf\rightarrow \cccO_{X}\otimes V\rightarrow \calq \rightarrow 0,$$
 where $\calq$ is the slope-stable exceptional bundle obtained by restricting to $X$ the quotient bundle on the Grassmannian containing $X$ in view of the embedding provided by $\calf^*$. Also, $\calq \simeq \calf^*$ if $g=7,9$.
 Then, we notice that $\ext^1(E(1),\cccO_{X})=h^2(E)=0$ and $\hom(E(1),\calq)=0$ by stability. So applying $\Hom(E(1), \cdot)$ we get $\ext^1(E(1),\calf)=0$.
Therefore we get that $\tilde{E}(1)$ is a coherent sheaf which, by construction, belongs to $\Phi(D^b(\Gamma))={}^{\perp}\langle \cccO_{X}, \calf^*\rangle$ and fits in
 \begin{equation}\label{eq:mutation-prime}
     0\rightarrow \calf^*\otimes \Ext^2(E(1),\calf)^*\rightarrow \tilde{E}(1)\rightarrow E(1)\rightarrow 0.
 \end{equation}

So the sheaf $\tilde{E}(r_X)$ lies in $\Phi(D^b(\Gamma))$. Accordingly, $\tilde{E}(r_X)\simeq \Phi(G)$ for some $G\in D^b(\Gamma)$. Thus, applying $\Phi^{!}$ to the sequences \eqref{eq:mutation-dp}, \eqref{eq:mutation-prime} and using that $\Phi^{!}(\calf^*)=0$, we obtain:
 \[G\simeq \Phi^{!}(\Phi(G))\simeq \Phi^{!}(\tilde{E}(r_X))=F.
 \]

 This shows that the complexes \eqref{eq:mutation-dp} and \eqref{eq:mutation-prime} are of the form described in \ref{Phi!ii}.
 For the proof of \ref{Phi!iii}, notice that whenever $E$ is a $(n,k)$-instanton with $k\ge 1$ on $Y_4$ or $X_{10}$, resp. a $(n,k)$-instanton on $X_g, \ g\in\{7,9\}$ belonging to $\calv(n,k),\ k\ge 2,$ we have $\Hom(\calf^*,E)=0.$

 By adjunction
 \[\Hom(F,F)\simeq \Hom(F, \Phi^{!}(E(r_X)))\simeq \Hom(\Phi(F), E(r_X))\simeq \Hom(\tilde{E}, E).\] Applying $\Hom(\cdot, E(r_X))$ to \eqref{eq:mutation-dp} and \eqref{eq:mutation-prime} we see that this last vector space is one-dimensional since $E(r_X)$ is simple and $\Hom(\calf^*,E(r_X))=0$. \end{proof}

\section{Monadic description of instantons}

\label{section:H3=0}

In this final section we look at Fano threefolds $X$ of Picard number one and trivial intermediate Jacobian, i.e. $H^3(X,\Z)=0.$ For each value of the index $i_X\in \{1,\ldots, 4\}$ there exists a unique family of Fano threefolds satisfying the condition $H^3(X,\Z)=0$. These are:
\begin{itemize}\label{item:list-tj}
    \item[$(\heartsuit)$] for $i_X=4, \ X\simeq \D{P}^3$;
    \item[$(\clubsuit)$] for $i_X=3, \ X\simeq Q\subset \D{P}^4$;
    \item[$(\diamondsuit)$] for $i_X=2$, $X=Y_5$ is the degree 5 Del Pezzo threefold. This is the projective variety obtained as a general codimension 3 linear section of $\mathbb{G} (2,U_5)\subset \D{P}^9$ where $U_5\simeq \D{C}^5$;
    \item[$(\spadesuit)$] for $i_X=1$, $X=X_{12}$ is the prime Fano threefold of genus $g=12$. The threefold $X$ identifies with a closed subvariety of the Hilbert scheme $\Hilb_{3t+1}(\D{P}(U_4))$ of twisted cubic curves in $\D{P}(U_4)\simeq \D{P}^3$, where $U_4$ is a 4-dimensional vector space. The subvariety $X$ corresponds to twisted cubics whose defining ideal is annihilated by a fixed 2-dimensional net of quadrics $\D{P}^2\simeq\D{P}(U_3)\subset \D{P}(S^2 U_4)$. 
\end{itemize}

The triviality of the intermediate Jacobian of $X$ 
turns out to be equivalent to the fact that the derived category $D^b(X)$ of $X$ admits a full exceptional collection.
We use, in addition to Beilinson’s theorem for $\p{3}$
and Kapranov’s theorem for the quadric, the full exceptional collections described in \cite{Orlov:V5,Kuznetsov:V22,Faenzi:V5,Faenzi:V22}.
In summary, we have that
$$ D^b(X)=\langle \calf_{-1}, \calf_0, \calf_1,\cccO_{X}(q_X)\rangle $$
for vector bundles $\calf_i, i\in \{-1,0,1\}$ such that:
\begin{equation}\label{eq:Kuz-cat}
    \calf_i^*\simeq \calf_{-i}(r_X), \qquad \langle \calf_{-1},\calf_0,\calf_{1}\rangle=\cccO_{X}(q_X)^{\perp}.
\end{equation} 

This structure of the triangulated category $\cccO_{X}(q_X)^{\perp}$ ensures that any object $F\in \cccO_{X}(q_X)^{\perp}$ is quasi-isomorphic to the homology
of a complex whose terms are all of the form $\calf_i^{\oplus n_{i,j}}$, for some integers $n_{i,j}$, with $i \in \{-1,0,1\}$ and $j \in \Z$.

To compute these integers, we use the dual exceptional collection $\langle \calg_{-1},\calg_0,\calg_1\rangle={}^\perp\cccO_{X}(-q_X)$ with respect to $\langle \calf_{-1},\calf_0, \calf_1\rangle = \cccO_{X}(q_X)^{\perp}$. 
The objects $\cG_i$, with $i \in \{-1,0,1\}$, are exceptional vector bundles
satisfying, for $i \in \{-1,0,1\}$:
\begin{equation}
  \label{mutation}
  \cG_{-i}^*(-q_X) \simeq \rL_{\cccO_X}(\cG_i(q_X))[-1],  
\end{equation}
where $\rL$ is the left mutation functor. The mutation functor is characterized by the property that, for any pair of
objects $\cF$ and $\cG$, there is a distinguished triangle
\[
\rL_\cF(\cG)[-1] \to \Ext_X^\bullet(\cF,\cG) \otimes \cF \to \cG \to  \rL_\cF(\cG), 
\]
where the map $\Ext_X^\bullet(\cF,\cG) \otimes \cF \to \cG$ is the natural evaluation.

Given a coherent sheaf $F$ on $X$, we write down the cohomology table of $F$ with respect to the collection mentioned above:
\begin{equation}
\begin{tabular}{||c | c | c ||}
 \hline
 $\calf_{-1}$ & $\calf_0$ & $\calf_1$ \\ [0.5ex]
 \hline\hline
 $ h^3(F\otimes \calg_{-1})$ & $h^3(F\otimes \calg_0)$ & $h^3(F\otimes \calg_1)$\\
\hline
$ h^2(F\otimes \calg_{-1})$ & $h^2(F\otimes \calg_0)$ & $h^2(F\otimes \calg_1)$\\
\hline
 $ h^1(F\otimes \calg_{-1})$ & $h^1(F\otimes \calg_0)$ & $h^1(F\otimes \calg_1)$\\
 \hline
 $ h^0(F\otimes \calg_{-1})$ & $h^0(F\otimes \calg_0)$ & $h^0(F\otimes \calg_1)$ \\
[1ex]
 \hline
\end{tabular}
\end{equation}

The sheaf $F$ is isomorphic to the homology of a complex $C^{\bullet}_F\in D^b(X)$ whose terms are:
\begin{equation}\label{eq:Beilinson-complex}
C^l_F=\bigoplus _{i-j=l+1} H^i(F\otimes \calg_{j})\otimes \calf_{j}.
\end{equation}

We give below an explicit description of the collections $\calf_i, \ \calg_i$, for $i \in \{-1,0,1\}$, on each of the Fano threefolds $X$ with  $H^3(X,\Z)=0$.

    \subsection*{The projective 3-space.}
    
    For $X=\p{3}$, we just use the classical exceptional collection:
    \begin{align*}
        \calf_{-1}&=\cccO_{\p{3}}(-1), && \calf_0=\cccO_{\p{3}}, && \calf_1=\cccO_{\p{3}}(1);\\
        \calg_{-1}&=\calt_{\D{P}^3}(-3), &&  \calg_0=\Omega_{\D{P}^3}, && \calg_1=\cccO_{\p{3}}(-1).
    \end{align*}
    Of course, this agrees with Beilinson's exceptional collection, appropriately twisted, for $\Omega^2_{\D{P}^3}(1) \simeq \calt_{\D{P}^3}(-3)$.
    All these vector bundles are slope-stable, with
    $\mu(\cG_{-1})=-5/3$, $\mu(\cG_{0})=-4/3$.
    
    \subsection*{The quadric threefold.}
    
    For a smooth quadric threefold $Q_3$ inside $\p{4}$, we use a version of Kapranov's exceptional collection.
    \begin{align*}
        \calf_{-1}&=\cccO_{Q}(-1), && \calf_0=\cals, && \calf_1=\cccO_{Q};\\
        \calg_{-1}&=\calt_{\D{P}^4}|_Q(-2), && \calg_0=\cals, && \calg_1=\cccO_{Q}.
    \end{align*}
    Here, as before, $\cals$ is the spinor bundle, which is slope-stable of rank $2$ and slope $-1/2$.
    The vector bundle $\calt_{\D{P}^4}|_Q(-2)$ is slope-stable of slope $-3/4$.
    
    \subsection*{The quintic Del Pezzo threefold.}
    
    The smooth quintic Del Pezzo threefold $Y=Y_5$ sits as a linear section in the Grassmannian $\mathbb{G}(2,U_5)$, where $U_5$ is a 5-dimensional vector space.
    The bundles $\calu_Y$ and $\calq_Y$ are obtained restricting to $Y$ the rank-2 tautological subbundle and the rank-3 tautological quotient bundle on $\mathbb{G}(2,U_5)$.  The linear section $Y$ is given by the choice of a 3-dimensional vector space $U_3\subset \wedge^2 U_5^*$.
    We have a natural identification $U_5=H^0(\calu_Y^*)^*$.
    
    The next ingredient is a vector bundle $\cR_Y$ of rank $9$, introduced in \cite[\S 3.2.2]{Faenzi}. It is defined as the kernel of the natural evaluation morphism $U_5^*\otimes \calu_Y\to \cccO_{Y}$, so it fits into
     \begin{equation} \label{define-R}
      0 \to \cR_Y \to U_5^*\otimes \calu_Y\to \cccO_Y \to 0.
     \end{equation}
        We have $h^0(\cR_Y(1))=18$ and the vector bundle $\cR_Y(1)$ is globally generated. 
        The isomorphism \eqref{mutation} is given by the evaluation map of global sections, which provides a natural exact sequence 
        \[
        0 \to \cR_Y^*(-1) \to U_{18} \otimes \cccO_{Y} \to \cR_Y(1) \to 0.
        \]
        Here, $U_{18}=H^0(\cR_Y(1))$ is an 18-dimensional vector space. 
    To see this sequence explicitly, we note that $U_3$ is naturally contained in $U_5^* \otimes U_5^*$ and, using 
     \cite[Exact Sequence (10)]{Faenzi}, we identify the quotient $U_{18}=U_5^* \otimes U_5^*/U_3$ with $H^0(\cR_Y(1))$ and the kernel of the evaluation of global sections of $\cR_Y(1)$ with $\cR_Y^*(-1)$.
    In summary, for $Y=Y_5$, we have
      \begin{align*}
        & \cF_{-1} = \calu_Y, && \cF_0 = \cccO_Y, && \cF_1 = \calu_Y^{\vee},\\
        & \cG_{-1} = \calq_Y(-1), && \cG_0 = \cR_Y, && \cG_1 = \calu_Y.
      \end{align*}
    We know that $\calu_Y$ and $\calq_Y(-1)$ are slope-stable (see \cite{Faenzi:V5}), with $\mu(\calq_Y(-1))=-2/3$, $\mu(\calu_Y)=-1/2$.  
      
    \begin{claim}
    The bundle $\cR_Y$ is slope-stable, with slope $-5/9$.
    \end{claim}
    
    \begin{proof}
    The statement about $\calu_Y$ is well-known. The slope of $\cR_Y$ is computed immediately from 
    \eqref{define-R}. To check stability, by contradiction, we take a subsheaf $K$ destabilising $\cR_Y$. We may assume that $K$ is slope-stable, and, up to passing to double dual, that $K$ is reflexive.
    In view of \eqref{define-R}, since $\calu_Y^*$ is stable of slope $-1/2$ and $K$ injects in the polystable bundle $\calu_Y^{\oplus 5}$, we see that $K$ must have slope $-1/2$ and that the reflexive hull of $K$ is isomorphic to $\calu_Y$.  
    Then $K$  must split off a copy of $\calu_Y$ from $\cR_Y$, which cannot happen, since $\cR_Y$ is exceptional, hence indecomposable.
    \end{proof}
    
    \subsection*{Prime Fano threefolds of genus $12$.} 
    Consider $X=X_{12}$, a smooth prime Fano threefold of genus $12$.
    We mentioned that $X$ is a closed subvariety of the Hilbert scheme of cubics in $\p{}(U_4)$.
    More precisely, $X$ parametrises twisted cubics whose ideals are generated by quadrics lying in the kernel of $S^2 U_4^*\to U_3^*$. Writing $U_7$ for the 7-dimensional kernel of this map, we get that $X$ embeds in $\mathbb{G} (3, U_7)$. 
    The threefold $X$ is cut out as a subvariety of $\mathbb{G} (3, U_7)$ as vanishing locus of a net of skew-symmetric $2$-forms, according to Mukai's famous description.
    Then we get vector bundles $\calu_X$ and $\calq_X$ by restricting to $X$ the tautological rank-3 subbundle, and the tautological rank-4 quotient bundle on  $\mathbb{G}(3,U_7)$. 
    
    Also, $X$ can be seen as a subscheme of the moduli space of twisted cubics in $\p{3}=\p{}(U_4)$, where $U_4$ is a $4$-dimensional vector space. Then the minimal instanton $F_0$, which is denoted here by $\calt$, is the rank-2 vector bundle whose fibre over a twisted cubic curve $[C]\in X$ is the two-dimensional linear space generated by the linear syzygies of the quadrics defining $I_C$. 
    
    Finally, we have a rank-10 vector bundle $\cR_X$, which is the dual of the right mutation of $\calt$ with respect to $\calu_X$, so that the dual of the fibre of $\cR_X$ at $[C]$ represents the 10-dimensional space of cubics of $\p{3}$ that vanish on $[C]$. One has an exact sequence
    \begin{equation} \label{define-RX}
        0\to \cR_X\to U_4^*\otimes \calu_X^* \to \calt^* \to 0.
    \end{equation}
    The bundle $\cR_X$ is introduced and studied in \cite{Faenzi:V22}, where it is denoted by $M$ in loc. cit. 
    The mutation induced by global sections of $\cR_X=\cG_0$, see \eqref{mutation}, here takes the form 
     \begin{equation} \label{define-W}
    0 \to \cR_X^* \to S^3 U_4^* \otimes \cccO_X \to \cR_X \to 0,
    \end{equation}
    where the map $\cccO_X^{20} \to \cR_X$ is the canonical evaluation of global sections. All this is proved in \cite[Lemma 6.7]{Faenzi:V22}.
    In summary, we have, for $X=X_{12}$:
        \begin{align*}
            \calf_{-1}&=\calu_X^*(-1), && \calf_0=\calt, &&  \calf_1=\calu_X;\\
              \calg_{-1}&=\calq_X, &&  \calg_0=\cR_X, &&  \calg_1=\calu_X^*.
        \end{align*}
    
    The bundles $\calu_X$ and $\calq_X$ are slope-stable (see \cite[Lemma 6.2]{Faenzi:V22}), with $\mu(\calq_X)=1/4$, $\mu(\calu_X^*)=1/3$.

    \begin{claim}
    The bundle $\cR_X$ is slope-stable, with slope $3/10$.
    \end{claim}
    
    \begin{proof}
    One computes the slopes of $\cR_X$ from \eqref{define-W}. A destabilizing reflexive, stable subsheaf $K$ of $\cR_X^*$ sits in the polystable bundle $\calu_X^{\oplus 4}$ by \eqref{define-RX}.
    Since there is no rational number $p/q$ with $q \le 9$ with $-1/3 \le p/q \le -3/10$, we get that $K$ splits off a summand $\calu_X$ from $\cR_X$, which is impossible, since $\cR_X$ is exceptional.
    \end{proof}
    
    Using the complex \eqref{eq:Beilinson-complex}, one obtains a monadic representation of an instanton bundle. 
    We use similar arguments as in \cite{Faenzi} to determine the cohomology table \eqref{eq:cohom-table} for non-locally free and higher-rank instantons. This leads to the next result.
    
    \begin{theorem}
    \label{prop-monad}
    Let $E$ be a Gieseker-semistable $(n,k)$-instanton on a Fano threefold $X$ such that $H^3(X)=0$. 
    Suppose that moreover $k\ge 2$ for $X=Y_5$ and $k\ge 1$ for $X=X_{12}.$
    Then $E$ is the homology of a monad of the form:
    \begin{equation}\label{eq:monad}
    \calf_{-1}^{\oplus k} \rightarrow W\otimes \calf_0\rightarrow \calf_1^{\oplus k}
    \end{equation}
    where $W$ is a complex vector space of dimension $w:=\dim(W)$, with:
    \begin{equation}\label{eq:dim-w}
    \begin{tabular}{||c |c  ||}
 \hline
 $i_X$ & w \\ [0.5ex]
 \hline\hline
4 & 2k+r \\
 \hline
3 & k+n\\
 \hline
2 & 4k+r\\
 \hline
1 & 3k+n\\[0.5ex]
 \hline
\end{tabular}
\end{equation}
\end{theorem}

\begin{proof}
To start with, we recall that, as observed at the beginning of the section, a Gieseker-semistable $(n,k)$-instanton $E$ lies in $\cccO_{X}(q_X)^{\perp}=\langle \calf_{-1},\calf_0,\calf_{1}\rangle$.

We now adapt to the higher rank case the arguments presented in \cite{Faenzi}.
Doing so, we show that the cohomology table of a Gieseker-semistable $(n,k)$-instanton satisfying the hypotheses of the statement, with respect to the collection $\calg_{-1},\calg_0,\calg_1$, has the following form:
\begin{equation}\label{eq:cohom-table}
\begin{tabular}{||c | c | c ||}
 \hline
 $\calf_{-1}$ & $\calf_0$ & $\calf_1$ \\ [0.5ex]
 \hline\hline
0 & 0& 0 \\
 \hline
0 & 0& 0  \\
 \hline
$k$ & $w$ & $k$ \\
 \hline
0 & 0& 0 \\[0.5ex]
 \hline
\end{tabular}
\end{equation}
The complex \eqref{eq:Beilinson-complex} reduces therefore to a monad of the form appearing in \eqref{eq:monad}.

To exhibit the table \eqref{eq:cohom-table}, it is enough to prove the vanishing results appearing in the table. Indeed, the dimension of the remaining cohomology group is computed using Riemann-Roch, relying on the Chern character $\gamma(n,k)$, which is computed by \eqref{character of instanton}. To get the vanishing results appearing in the table, we only need that $E$ is slope-semistable and satisfies $H^1(E(-q_X))=0$.

We first check that $h^0(E \otimes \cG_1)=0$, which implies the vanishing results summarized in the bottom row of \eqref{eq:cohom-table}.
By stability of the involved sheaves, this holds if 
$\mu(\cG_1)<-\mu(E)=r_X/2$. However, we have $\mu(\cG_1)<0$ if $r_X=0$, while $\mu(\cG_1)=0 < 1/2=-\mu(E)$ for $i_X=3$ and $\mu(\cG_1)=1/3 < 1/2=-\mu(E)$ for $i_X=1$.

Next, we check the top row of \eqref{eq:cohom-table}. Note that it is enough to show $h^3(E \otimes \cG_{-1})=0$.
By slope-semistability of $E$ and $\cG_{-1}$, coupled with Serre duality, this follows once we show $-\mu(E)<\mu(\cG_{-1}) + i_X$, which in turn is checked immediately.

Finally, we have to justify $h^2(E \otimes \cG_i)=0$ for $i \in \{-1,0,1\}$. We use the mutation formula \eqref{mutation} to write down an exact sequence:
\[
0 \to \cG_{-i}^*(-2q_X) \to H^0(\cG_i(q_X)) \otimes \cccO_X(-q_X) \to \cG_{i} \to 0, \qquad \mbox{for $i \in \{-1,0,1\}.$}
\]
We tensor the above sequence with $E$ and use the vanishing $h^p(E(-q_X))=0$ for $p=2,3$ to get $h^2(E \otimes \cG_{i})=h^3(E \otimes \cG_{-i}^*(-2 q_X))$ for $i\in \{-1,0,1\}$.
Again, this vanishes in view of Serre duality, because of stability of $E$ and $\cG_{-i}$, together with the calculation 
$$-\mu(E) + \mu(\cG_{-i}) +2 q_X-i_X=r_X/2+\mu(\cG_{-i})< 0, \qquad \mbox{for $i\in \{-1,0,1\}$}.$$
\end{proof}
\providecommand{\bysame}{\leavevmode\hbox to3em{\hrulefill}\thinspace}
\providecommand{\MR}{\relax\ifhmode\unskip\space\fi MR }
\providecommand{\MRhref}[2]{%
  \href{http://www.ams.org/mathscinet-getitem?mr=#1}{#2}
}
\providecommand{\href}[2]{#2}

\end{document}